\documentclass[reqno,english]{amsart}

\usepackage{amsmath,amsfonts,amssymb,graphicx,amsthm,enumerate,url}
\usepackage[noadjust]{cite}
\usepackage{stmaryrd}
\usepackage{mathrsfs,booktabs,tabularx}
\usepackage{xifthen,xcolor,tikz,setspace}
\usetikzlibrary{decorations.pathmorphing,patterns,shapes,calc,decorations}
\usetikzlibrary{decorations.pathreplacing}
\usepackage{mathtools,upgreek}
\usepackage[final]{showkeys} %add in 'final' into parameter to remove showkeys

\definecolor{refkey}{gray}{.75}
\definecolor{labelkey}{gray}{.5}
\usepackage[algo2e,boxed,vlined,algoruled]{algorithm2e}

\usepackage[letterpaper]{geometry}
\usepackage[colorinlistoftodos]{todonotes}
\presetkeys{todonotes}{inline, color=green}{}

\usepackage{accents}

\usepackage[colorlinks=true]{hyperref}
\colorlet{DarkGreen}{green!50!black}
\colorlet{DarkGray}{gray!60!black}
\numberwithin{equation}{section}

\renewcommand{\restriction}{\mathord{\upharpoonright}}
\renewcommand{\epsilon}{\varepsilon}
\newcommand{\given}{\;\big|\;}
\newcommand{\one}{\mathbf{1}}

\usepackage{color}
 \definecolor{refkey}{gray}{.5}
 \definecolor{labelkey}{gray}{.5}
\definecolor{light}{gray}{.9}

\newtheorem{maintheorem}{Theorem}

\newtheorem{theorem}{Theorem}[section]
\newtheorem*{theorem*}{Theorem}
\newtheorem{lemma}[theorem]{Lemma}

\newtheorem{claim}[theorem]{Claim}
\newtheorem{proposition}[theorem]{Proposition}

\newtheorem{observation}[theorem]{Observation}
\newtheorem{fact}[theorem]{Fact}
\newtheorem{corollary}[theorem]{Corollary}

\theoremstyle{definition}{

\newtheorem{definition}[theorem]{Definition}

\newtheorem*{definition*}{Definition}

\newtheorem{remark}[theorem]{Remark}

}

\newcommand{\E}{\mathbb E}

\newcommand{\R}{\mathbb R}
\newcommand{\Z}{\mathbb Z}

\newcommand{\cC}{\ensuremath{\mathcal C}}

\newcommand{\cF}{\ensuremath{\mathcal F}}
\newcommand{\cG}{\ensuremath{\mathcal G}}
\newcommand{\cH}{\ensuremath{\mathcal H}}
\newcommand{\cI}{\ensuremath{\mathcal I}}
\newcommand{\cJ}{\ensuremath{\mathcal J}}
\newcommand{\cK}{\ensuremath{\mathcal K}}
\newcommand{\cL}{\ensuremath{\mathcal L}}

\newcommand{\cP}{\ensuremath{\mathcal P}}

\newcommand{\cS}{\ensuremath{\mathcal S}}

\newcommand{\llb }{\llbracket}
\newcommand{\rrb }{\rrbracket}

\newcommand{\sT}{{\ensuremath{\mathscr T}}}

\newcommand{\fm}{\mathfrak{m}}

\newcommand{\fG}{\mathfrak{G}}

\newcommand{\fP}{\mathfrak{P}}

\newcommand{\fW}{\mathfrak{W}}
\newcommand{\fX}{\mathfrak{X}}

\newcommand{\sB}{{\ensuremath{\mathscr B}}}
\newcommand{\sC}{{\ensuremath{\mathscr C}}}
\newcommand{\sE}{{\ensuremath{\mathscr E}}}
\newcommand{\sF}{{\ensuremath{\mathscr F}}}

\newcommand{\sX}{{\ensuremath{\mathscr X}}}

\newcommand{\sfh}{{\ensuremath{\mathsf{h}}}}

\newcommand{\g}{{\ensuremath{\mathbf g}}}
\newcommand{\br}{{\ensuremath{\mathbf r}}}
\newcommand{\bB}{{\ensuremath{\mathbf B}}}
\newcommand{\bC}{{\ensuremath{\mathbf C}}}
\newcommand{\bD}{{\ensuremath{\mathbf D}}}

\newcommand{\bF}{{\ensuremath{\mathbf F}}}

\newcommand{\bH}{{\ensuremath{\mathbf H}}}
\newcommand{\bI}{{\ensuremath{\mathbf I}}}

\newcommand{\bV}{{\ensuremath{\mathbf V}}}
\newcommand{\bW}{{\ensuremath{\mathbf W}}}
\newcommand{\bX}{{\ensuremath{\mathbf X}}}
\newcommand{\bY}{{\ensuremath{\mathbf Y}}}

\newcommand{\bP}{{\ensuremath{\mathbf P}}}

\newcommand{\fs}{\mathfrak{s}}

 \renewcommand{\epsilon}{\varepsilon}

\DeclareMathOperator{\diam}{diam}

\DeclareMathOperator{\ber}{Ber}

\newcommand{\Clust}{{\mathsf{Clust}}}
\newcommand{\Iso}{{\mathsf{Iso}}}
\newcommand{\iso}{{\mathsf{iso}}}
\newcommand{\Cone}{{\mathsf{Cone}}}
\newcommand{\swap}{{\mathsf{swap}}}
\newcommand{\Cyl}{{\mathsf{Cyl}}}
\newcommand{\ext}{{\mathsf{ext}}}

\DeclareMathOperator{\hgt}{ht}
\newcommand{\tv}{{\textsc{tv}}}

\newcommand{\trivincr}{\varnothing}
\newcommand{\hull}[1]{{\mathbullet{#1}}}

\DeclareMathOperator{\isodim}{\dim_{\sf ip}}

\newcommand{\mufloor}{{\widehat\upmu}}

\makeatletter
\newcommand{\superimpose}[2]{%
  {\ooalign{$#1\@firstoftwo#2$\cr\hfil$#1\@secondoftwo#2$\hfil\cr}}}
  
\newcommand{\sbullet}{%
  \hbox{\fontfamily{lmr}\fontsize{.4\dimexpr(\f@size pt)}{0}\selectfont\textbullet}}
\DeclareRobustCommand{\mathbullet}{\accentset{\sbullet}}

\makeatother

\newcommand{\udarrow}{{\ensuremath{\scriptscriptstyle \updownarrow}}}
\newcommand{\lrvec}[1]{\overset{\,{}_\leftrightarrow}{#1}\!}
\newcommand{\red}{\textsc{red}}
\newcommand{\blue}{\textsc{blue}}
\newcommand{\green}{\textsc{green}}

\begin{document}

\title{Approximate domain Markov property for rigid Ising interfaces}

\author{Reza Gheissari}
\address{R.\ Gheissari\hfill\break
Department of Statistics and EECS \\ UC Berkeley }
\email{gheissari@berkeley.edu}

\author{Eyal Lubetzky}
\address{E.\ Lubetzky\hfill\break
Courant Institute\\ New York University\\
251 Mercer Street\\ New York, NY 10012, USA.}
\email{eyal@courant.nyu.edu}

\vspace{-1cm}

\begin{abstract}
Consider the Ising model on a centered box of side length $n$ in $\mathbb Z^d$ with $\mp$-boundary conditions that are minus in the upper half-space and plus in the lower half-space. 
Dobrushin famously showed that in dimensions $d\ge 3$, at low-temperatures the Ising interface (dual-surface separating the plus/minus phases) is rigid, i.e., it has $O(1)$ height fluctuations. Recently, the authors decomposed these oscillations into pillars and identified their typical shape, leading to a law of large numbers and tightness of their maximum.

Suppose we condition on a height-$h$ level curve of the interface, bounding a set $S \subset \mathbb Z^{d-1}$, along with the entire interface outside the cylinder $S\times \mathbb Z$: what does the interface in $S\times \mathbb Z$ look like? Many models of random surfaces (e.g., SOS and DGFF) fundamentally satisfy the \emph{domain Markov property}, whereby their heights on $S$ only depend on the heights on $S^c$ through the heights on $\partial S$. The Ising interface importantly does not satisfy this property; the law of the interface depends on the full spin configuration outside $S\times \mathbb Z$. 

Here we establish an approximate domain Markov property inside the level curves of the Ising interface. 
We first extend Dobrushin's result to this setting, showing the interface in $S\times \mathbb Z$ is rigid about height $h$, with exponential tails on its height oscillations.  Then we show that the typical tall pillars in $S\times \mathbb Z$ are uniformly absolutely continuous with respect to tall pillars of the unconditional Ising interface. Using this we identify the law of large numbers, tightness, and Gumbel tail bounds on the maximum oscillations in $S\times \mathbb Z$ about height $h$, showing that these only depend on the conditioning through the cardinality of $S$. 
\end{abstract}

% {\mbox{}
% \vspace{-.8cm}
\maketitle
% }
\vspace{-0.8cm}

\section{Introduction}

The Ising model $\mu_\Lambda^\mp$ on a finite graph $\Lambda\subset \Z^d$ at inverse-temperature $\beta>0$ is the following distribution on $\pm 1$-\emph{spin} configurations $\sigma$ on $\sC(\Lambda)$, the
$d$-dimensional cells of $\Lambda$, which we  identify with their midpoints in $(\Z+\frac 12)^d$. Setting $\sigma(u)=-\operatorname{sign}(u_d)$ 
for $u=(u_1,\ldots,u_d)\in \sC(\Z^d)\setminus \sC(\Lambda)$ (Dobrushin $\mp$-boundary conditions),
\[ \mu_\Lambda^\mp(\sigma) \propto e^{-\beta \cH(\sigma)}\,,\qquad\mbox{where}\qquad\cH(\sigma)= \sum_{u,v:\, d(u,v)= 1} \!\!\!\one\{\sigma(u)\neq\sigma(v)\}\,,
\]
the sum is on pairs in $\sC(\Lambda\cup \partial \Lambda) = \{u\in \sC(\Z^d): d(u,\sC(\Lambda))\le 1\}$,
and $d(\cdot,\cdot)$ denotes Euclidean distance. This extends to infinite subgraphs of $\Z^d$ via weak limits.  We consider the low-temperature regime ($\beta$ large)~on  
\[ \Lambda_n= \cL_{0,n} \times \llb-\infty,\infty\rrb\,,\qquad\mbox{where}\qquad \cL_{0,n} = \llb - n,n\rrb^{d-1}\,\qquad \mbox{and}\qquad d\ge 3\,.\]
The Ising \emph{interface} $\cI$ corresponding to a configuration $\sigma$ under $\mu_{\Lambda_n}^\mp$ is the unique infinite $*$-connected component of faces ($(d-1)$-cells) separating plus and minus $d$-cells 
(two faces are $*$-adjacent, if they share a common bounding vertex).
Informally, $\cI$ separates the plus phase (below) from the minus phase (above). 

Our focus in this work is on the domain Markov property (DMP), a fundamental feature of many well-studied models of height functions (viewed as random surfaces), e.g., the Discrete Gaussian Free Field and the family of $|\nabla\varphi|^p$ models, which includes Solid-On-Solid and the Discrete Gaussian. The DMP states that, for any subset $S$, conditioning on the values of the field on $\partial S$ gives the same model on $S$, with the induced boundary conditions, and this is conditionally independent of its values on~$S^c$ (see, e.g.,~\cite{Funaki05,IV18,Velenik06,Zeitouni16} for accounts on these models and the progress made in the last two decades, where DMP played a crucial role).

Unlike these random height functions, the law of the Ising interface within $S$ \emph{does} depend on the interface in $S^c$ beyond $\partial S$, e.g., through the finite bubbles in the spin configurations above and below the interface. In dimension $d=2$,  Ornstein--Zernike (OZ) theory (cf.~\cite{PV97,PV99,CIV03}) for high-temperature connections yields an approximate DMP for the low-temperature interfaces by duality, which was useful for proving scaling limits, entropic repulsion, and wetting and pinning phenomena~(see, e.g., the recent survey~\cite{IV18}). This theory has no low temperature analog in dimension $d\geq 3$, where the interfaces are random surfaces (rather than curves).

Here we show that for $d\ge 3$ and low temperatures, when $\partial S$ is a height-$h$ \emph{level line}, the law of the oscillations of the interface in $S$ about height $h$, conditional on the interface in $S^c$, resembles the law of the interface oscillations about height $0$ in the unconditional Ising distribution on $S$ with $\mp$-boundary conditions.
For simplicity of notation, we present our results for $d=3$ (their adaptation to $d>3$ is straightforward).

Let us first recall the properties of the interface under the unconditional law $\mu_{\Lambda_n}^\mp$.
In a pioneering work by Dobrushin~\cite{Dobrushin72a} in 1972, it was shown that the interface is \emph{rigid}, in that the
height fluctuations in $\cI$ above any point in $\cL_{0,n}$, say the origin, are $O_{\textsc p}(1)$, and moreover obey an exponential tail: for $\beta$ large enough,
\[ \mu_{\Lambda_n}^\mp\left( \cI \cap (\{0,0\}\times[h,\infty)) \neq \emptyset\right)\leq \exp(-\beta h /3)\qquad\mbox{for every $h\geq 1$}\,.\]
Dobrushin's proof relied on a novel decomposition of the interfaces into \emph{walls and ceilings} (whose definitions we will recall in~\S\ref{subsec:results}) and a delicate grouping of the walls that allowed a Peierls argument to flatten $\cI$. It follows from Dobrushin's results that, if $\beta$ is large enough, then with high probability, the interface would have height $0$ above at least $0.99$ of the faces in $\cL_{0,n}$, denoted $\sF(\cL_{0,n})$, and the maximum height of the interface $M_{\Lambda_n}$ satisfies (by a union bound on the above tail estimate)  $M_{\Lambda_n} \leq (C_0/\beta) \log n$ for any $C_0>6$. 

In the recent work~\cite{GL19a}, the authors  introduced the notion of the \emph{pillar} $\cP_x$ above a face $x\in \sF(\cL_{0,n})$ in order to describe the local height oscillations of the interface near $x$. That work established the shape and 
limiting large deviation rate for $\cP_x$ attaining height $h$, from which it follows that
\begin{align}\label{eq:ld-rate-height} 
\lim_{h\to\infty} -\frac1h\log \mu_{\Z^3}^\mp\left(\cI \cap (\{0,0\}\times[h,\infty))\neq \emptyset\right)= \alpha\,,
\end{align}
where 
 $\alpha$ is also the exponential decay rate of the probability that a $*$-connected plus chain connects the origin to height $h$ in $\Z^2\times [0,h]$ under $\mu^\mp_{\Z^3}$, and satisfies $\alpha \in (4\beta-C,4\beta+C)$ for some absolute constant $C$.
The framework of~\cite{GL19a} then led to a law of large numbers (LLN) for the maximum: $M_{\cL_{0,n}}/\log n \xrightarrow{\,\mathrm{p}\,} 2/\alpha$.

Via a substantially more refined analysis, the follow-up~\cite{GL19b} established tightness of the centered maximum:
\[ M_{\cL_{0,n}} - \E[M_{\cL_{0,n}}] = O_{\textsc p}(1)\,,\qquad\mbox{and}\qquad \E[M_{\cL_{0,n}}]-m^*_{|\cL_{0,n}|}=O(1)\]
for an explicit deterministic sequence $m_s^*$. Furthermore, the centered maximum obeys Gumbel tail bounds:
\[ 
e^{- C\exp\left(-4\beta k + C|k|\right)}
\leq \mu_{\Lambda_n}^\mp\big(M_{\cL_{0,n}} -m_{|\cL_{0,n}|}^* < k \big)
\leq e^{- c\exp\left(- 4\beta k - C|k|\right)}\quad\mbox{for any fixed $k\in\Z$}\,.\]

Even though the interface is not a height function (rather it is the boundary of a 3D connected component), Dobrushin's rigidity result implies that, in a typical interface $\cI$, at least $0.99$ of the
vertical columns of $\cL_{0,n}\times \mathbb \Z$ will intersect $\cI$ in exactly $1$ horizontal face.
The height of $\cI$ above those points is unique, giving rise to level lines: for a face $x$ in $\sF(\cL_{0,n})$, let $\hgt_\cI(x)$ be the height of the horizontal face of $\cI$ above $x$, whenever it is unique; a height-$h$ \emph{level line} is the external boundary of a $*$-connected component of $\{x\in \sF(\cL_{0,n}) : \hgt_\cI(x)=h\}$, i.e., the connected set of edges separating it from the infinite component of~$\Z^2$. We denote by $\mathfrak{L}_h=\mathfrak{L}_h(\cI)$ the set of height-$h$ level lines of the interface $\cI$; see Fig.~\ref{fig:level-lines}.
(We note in passing that, in terms of Dobrushin's definitions which we recall in~\S\ref{subsec:results}, a level line is the external boundary of a \emph{ceiling}.)

By analogy with random surface models which satisfy DMP, one would like to reason that the Ising distribution over interfaces inside a height-$h$ level line $\gamma_n$ bounding a set $S$ is essentially the same as that under $\mu_{S\times\Z}^\mp$ but shifted by $h$, even conditionally on the value of $\cI \cap (S^c \times \Z)$. However, if we were to expose the entire spin configuration (which \emph{does} satisfy DMP) outside $S\times \Z$ under this conditional measure, it would have a constant (sub-critical) fraction of plus-sites along the boundary $\partial S$ above height $h$, and minus sites below height $h$, and moreover, the law of these sub-critical bubbles is affected by the conditioning.

Estimates on the covariance of 3D Ising interface oscillations about the flat interface date back to follow-up work of Dobrushin~\cite{Dobrushin73} as well as~\cite{BLP79b}; these showed that under $\mu_{\Lambda_n}^{\mp}$, the covariance between those oscillations (formally, between \emph{walls}, which are connected oscillations of the interface supporting its level lines) decays exponentially. In particular, the total variation  distance between the joint law of level-lines through $x$ and $y$, and a product measure on these, decays exponentially in $d(x,y)$. Turning this bound---which is essentially sharp---
into one for the conditional distribution of the oscillations through $x$, given an $h$-level line $\gamma_n$ passing through $y$, is problematic: the conditioning has the effect of dividing the bound by the exponentially small quantity $\mu_{\Lambda_n}^{\mp}(\gamma_n\in \mathfrak L_h)\leq \exp ( - (\beta -C)|\gamma_n|)$. Alternatively, using cluster expansion~\cite{Minlos-Sinai} and viewing the conditional distribution over interfaces in $S\times \Z$ as a tilt of  $\mu_{S\times \Z}^{\mp}$, we would similarly find that the former is a tilt by at least $\exp ( C|\gamma_n|)$ of the latter (see e.g.,~\cite{HolickyZahradnik} for a sketch of this bound)---which is sharp due to the presence of sub-critical bubbles above and below the interface along $\partial S$. Using either approach, one has no control over conditional probabilities of events inside $S$ whose probabilities are greater than $\exp ( -C |\gamma_n|)$. 

In this work we find that for any $S\subset \cL_{0,n}$,  \emph{conditioning} on $\partial S$ being a height-$h$ level line generated by any $\cI \cap (S^c \times \Z)$, leads to a rigid interface in $S$ about height $h$. Moreover, if $S$ is ``thick" in that its perimeter is negligible in terms of its volume, then, the resulting distribution resembles the Ising interface under $\mu_{S\times \Z}^\mp$ shifted by $h$. The following special case of our Theorem~\ref{thm:gumbel} below, demonstrates this for the maximum. 

\begin{theorem*}[approximate DMP for a maximum, special case]
Fix $\beta>\beta_0$ large. There exists an explicit sequence  $m_s^*\asymp\log s$ so the following holds for all fixed $k\in\Z$:
If $\gamma_n\subset \sE(\Z^2)$ is a simple closed curve whose interior $S_n$ has $|S_n|>|\gamma_n|^{1.1}$, and $M_{S_n}$ is the maximum of $\cI$ above $S_n$, then for all $h\geq 1$ and large~$n$,
\begin{align*}
e^{- C\exp\left(- 4\beta k + C|k|\right)}&\leq \mu_{\Lambda_n}^\mp\left(M_{S_n}-h -m_{|S_n|}^* < k\given \gamma_n \in\mathfrak{L}_h\,,\, \cI \cap (S_n^c \times \Z) \right)
\leq e^{- c\exp\left(- 4\beta k -C|k|\right)}\,. 
\end{align*}
\end{theorem*}
The proof of this result requires several new ingredients, which will be outlined in~\S\ref{subsec:methods} below,
including a modified version of Dobrushin's walls-and-ceilings rigidity argument which (unlike the original one) does apply within a level line (Theorem~\ref{thm:intro-rigidity-inside-ceiling}); a coupling of the shifted pillar at a point in the bulk of $S_n$ under the conditional law $\mu_{\Lambda_n}^\mp(\cdot\mid\gamma_n\in\mathfrak{L}_h\,,\cI\cap(S_n^c\times\Z))$ with pillars under the unconditional law $\mu_{\Z^3}^\mp$ (Theorem~\ref{thm:cond-uncond-simple}); and sharp tail bounds on the maximum in the absence of FKG. 
These yield sharp bounds on the maximum $M_{S_n}$ conditioned on the walls in $S_n^c$---Theorem~\ref{thm:gumbel}---of which the above theorem is a special case.

\begin{figure}
\vspace{-0.05in}
    \centering
    \includegraphics[width=0.5\textwidth]{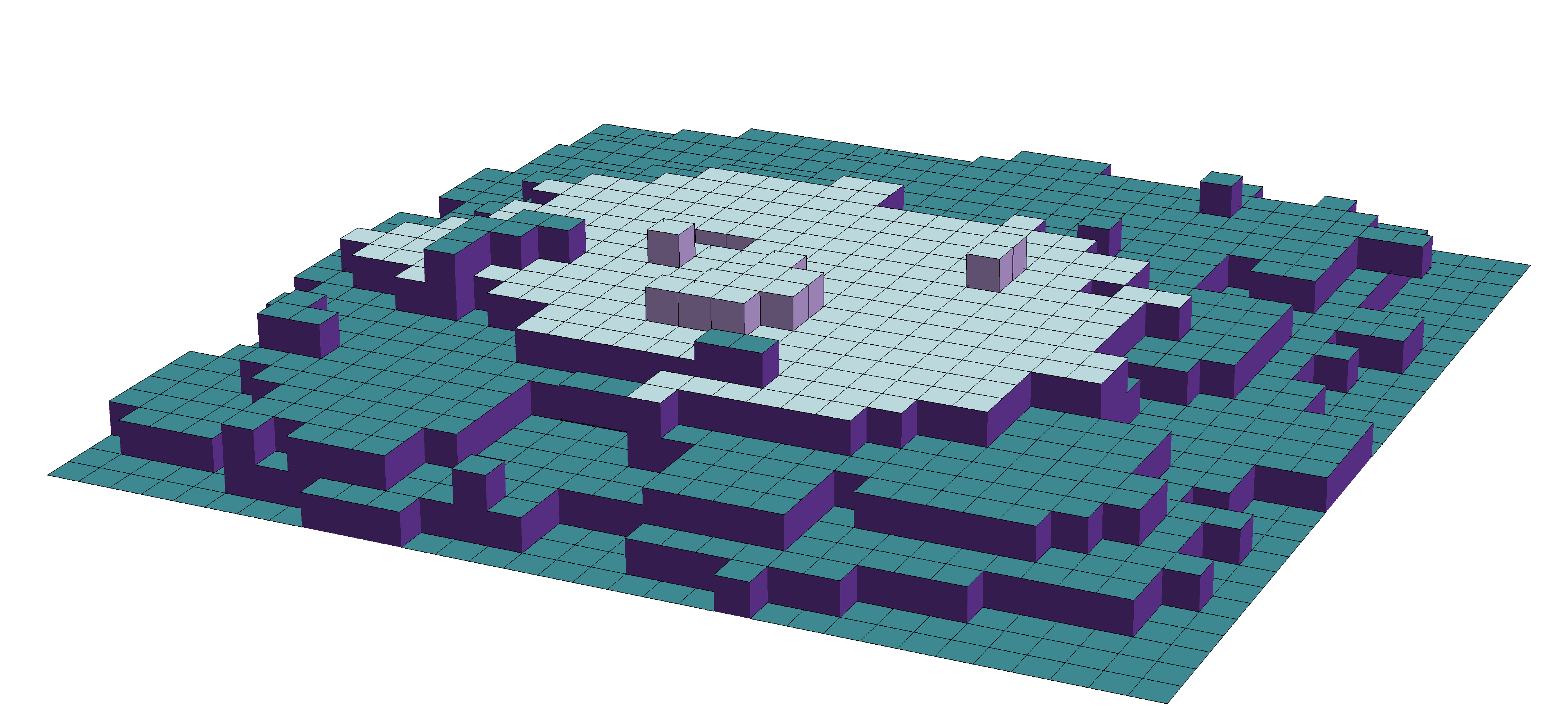}
    \hspace{-.3in}
    \raisebox{-20pt}{\includegraphics[width=0.4
    \textwidth]{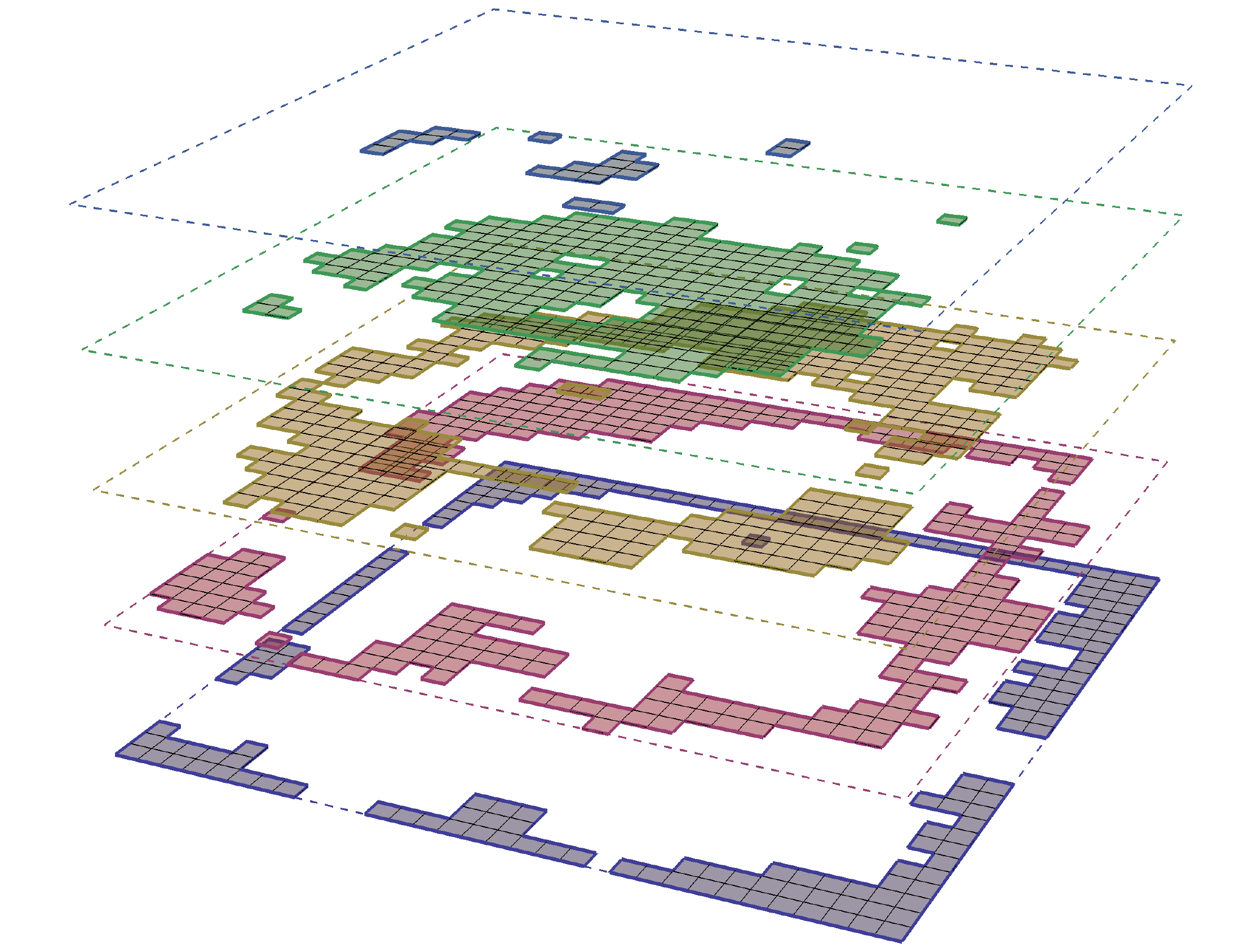}}
    \vspace{-0.1in}
    \caption{Level sets $\{x:\hgt_\cI(x) = h\}$ of the interface $\cI$ (right). The set $S_n$ is the interior of the largest height-3 level line, bounding the restriction $\cI \cap (S_n \times \Z)$ (highlighted, left).}
    \label{fig:level-lines}
\end{figure}

\subsection{Main results}\label{subsec:results}
In this section, we state our full results precisely. Towards that, let us begin by recalling the wall and ceiling decomposition introduced in~\cite{Dobrushin72a} to prove rigidity of the interface under $\mu_{\Lambda_n}^\mp$. Viewing  $\cL_{0,n}= \llb -n,n\rrb^{d-1}\times \{0\}$ as a subset of $\Lambda_n$, the ground state (lowest energy) interface is the flat one $\sF(\cL_{0,n})$; the vertical oscillations of $\cI$ about $\cL_{0,n}$ are grouped into \emph{walls}, defined next. For this it will help to define a \emph{projection} of a set $A\subset \Lambda_n$ into $\cL_{0,n}$ via $\rho(A) = \{(x_1,x_2,0): (x_1,x_2,x_3)\in A\mbox{ for some $x_3$}\}$.

\begin{definition*}[Ceilings and walls: see also~Def.~\ref{def:ceilings-walls}]
A face $f\in\cI$ is a \emph{ceiling} face if it is a horizontal face (having normal vector $e_3$) such that no other horizontal face $f'\in\cI$ has $\rho(f') = \rho(f)$. A face $f\in \cI$ is a \emph{wall} face if it is not a ceiling face. 
A \emph{wall} $W$ is a $*$-connected set of wall faces of $\cI$ and a \emph{ceiling} $\cC$  is a $*$-connected set of ceiling faces of $\cI$.  We assign the walls of $\cI$ to the index set of faces of $\cL_{0,n}$, setting $W_f := W$ if $f\in \sF(\cL_{0,n})$ is incident to and interior to $\rho(W)$; if there is no such wall in $\cI$, set $W_f = \trivincr$.
\end{definition*}

\noindent (Note that the boundary faces of a ceiling $\cC$ are at the same height, denoted $\hgt(\cC)$; thus, when projected down to $\cL_{0,n}$ the boundary of the ceiling forms a level line at this height, relating it to $\mathfrak{L}_{\hgt(\cC)}$ from above.)

For a set $A_n\subset \cL_{0,n}$ such that $\cL_{0,n}\setminus A_n$ is simply-connected, and an admissible set of walls $\bW_n=(W_z)_{z\in A_n}$ (some possibly $\trivincr$), let $\bI_{\bW_n}$ be the set of interfaces such that for every $z\in A_n$, the wall assigned to $z$ is $W_z$. 

The generalization of the level-line conditioning from above is to fix a simply-connected $S_n\subset \cL_{0,n}$ and condition on all exterior walls, i.e., condition on $\bI_{\bW_n}$ for $\bW_n$ such that $S_n \cap \bigcup_{W\in \bW_n} \rho(W)= \emptyset$. In this setting, all interfaces $\cI \in \bI_{\bW_n}$ have the same restriction $\cI \cap (S_n^c \times \Z)$. The collection of walls $\bW_n$ dictate a height $\hgt(\cC_{\bW_n})$ as follows: letting $\cI_{\bW_n}$ denote the unique interface whose only walls are $\bW_n$, then $\cI_{\bW_n}$ has a single ceiling called $\cC_{\bW_n}$ such that $S_n\subset\rho(\cC_{\bW_n})$. In the case where $S_n,\bW_n$ are such that $\partial S_n = \partial \rho(\cC_\bW)$, the conditioning on $\bI_{\bW_n}$ corresponds exactly to conditioning on $\partial S_n \in \mathfrak L_{\hgt(\cC_{\bW})}$ and $\cI \cap (S_n^c \times \Z)$; see Fig.~\ref{fig:hole-in-ceiling}.

Given $\bI_{\bW_n}$, the lowest-energy interface is the one that has all horizontal faces at $\hgt(\cC_{\bW_n})$ inside $S_n$; our aim is therefore to show that the interface $\cI \cap (S_n\times \Z)$ resembles an interface under $\mu_{S_n\times \Z}^{\mp}$ shifted up by $\hgt(\cC_{\bW_n})$. 
The approximate DMP for the interface $\cI \cap (S\times \Z)$ cannot hold for all events, as emphasized by the fact that the tilt on the partition function is of order $\exp ( C |\partial S_n|)$. Nonetheless, we show that the distribution $\mu_{\Lambda_n}^{\mp} (\cI\cap (S\times \Z)\in \cdot \mid \bI_{\bW_n})$ behaves like a vertical shift of $\mu_{S_n\times \Z}^{\mp}$ by $\hgt(\cC_{\bW_n})$ in two key aspects: the behavior of its maximum height oscillation, and the shape of its typical pillars.  

\subsubsection*{Approximate domain Markov for the maximum}
In order to present our main result on the maximum height oscillation within $S$, we recall certain quantities from~\cite{GL19a,GL19b} which will be used to explicitly determine the location of the maximum. For any subset $S \subset \cL_{0,n}$, let 
\begin{equation}\label{eq:def-Mn}
    M_S = M_S(\cI) := \max \{ x_3: (x_1, x_2, x_3) \in \cI, (x_1, x_2, 0)\in S\}\,.
\end{equation}
The LLN established for $M_{\cL_{0,n}}$ in~\cite{GL19a} states that there exists $\beta_0$ such that for every $\beta>\beta_0$ we have ${M_{\cL_{0,n}}} = \frac{2}{\alpha}{\log n}+ o(\log n)$ in $\mu_{\Lambda_n}^\mp$-probability, for $\alpha = \alpha(\beta)$ given as follows. 
If  $v\xleftrightarrow[A~]{+}w$ denotes existence of a $*$-adjacent path of plus spins connecting $v$ and $w$ in $A$, then
\begin{equation}\label{eq:alpha-alpha-h-def}
\alpha := \lim_{h\to\infty} \frac{\alpha_h}h\qquad\mbox{for}\qquad\alpha_h(\beta) = -\log \mu_{\Z^3}^\mp\Big((\tfrac12,\tfrac12,\tfrac12) \xleftrightarrow[\R^2\times [0,\infty)]{+} (\Z+\tfrac 12)^2\times\{h-\tfrac12\}\Big)\,.
\end{equation}
(The existence of the limit above is non-trivial and was an important step in the proof of~\cite{GL19a} of the LLN.) 
The tightness of $(M_{\cL_{0,n}}-\E[M_{\cL_{0,n}}])$, obtained in~\cite{GL19b}, was accompanied by  $|\E[M_{\cL_{0,n}}] - m_{|\cL_{0,n}|}^*|\leq 1$ for
\begin{equation}\label{eq:m*-def}
m^*_{s_n} =  \inf\{ h \geq 1 \;:\, \alpha_h(\beta) > \log (s_n) - 2\beta\}\,,
\end{equation}
as well as uniform lower and upper Gumbel tails for the 
centered maximum $M_{\cL_{0,n}}-m_{|\cL_{0,n}|}^*$.

We will be interested in the height fluctuations of $\cI$ within $S$ conditioned on $\bI_{\bW_n}$ for $\bW_n$ satisfying $\rho(\bW_n) \cap S = \emptyset$. We wish to study $M_S$ relative to the height induced by $\cI \cap (S^c \times \Z)$, namely $\hgt(\cC_\bW)$. Towards that, define the relative maximum
 \[ \bar M_{S} = M_S -\hgt(\cC_\bW)\,.\]

\begin{definition*}[isoperimetric dimension of face sets]
A simply-connected subset of faces $S\subset \sF(\cL_{0,n})$ is said to have isoperimetric dimension at most $d$, denoted $\isodim(S) \leq d$, if $|\partial S| \leq |S|^{(d-1)/d}$.
\end{definition*}

\begin{figure}
    \centering
    \includegraphics[width=0.55\textwidth]{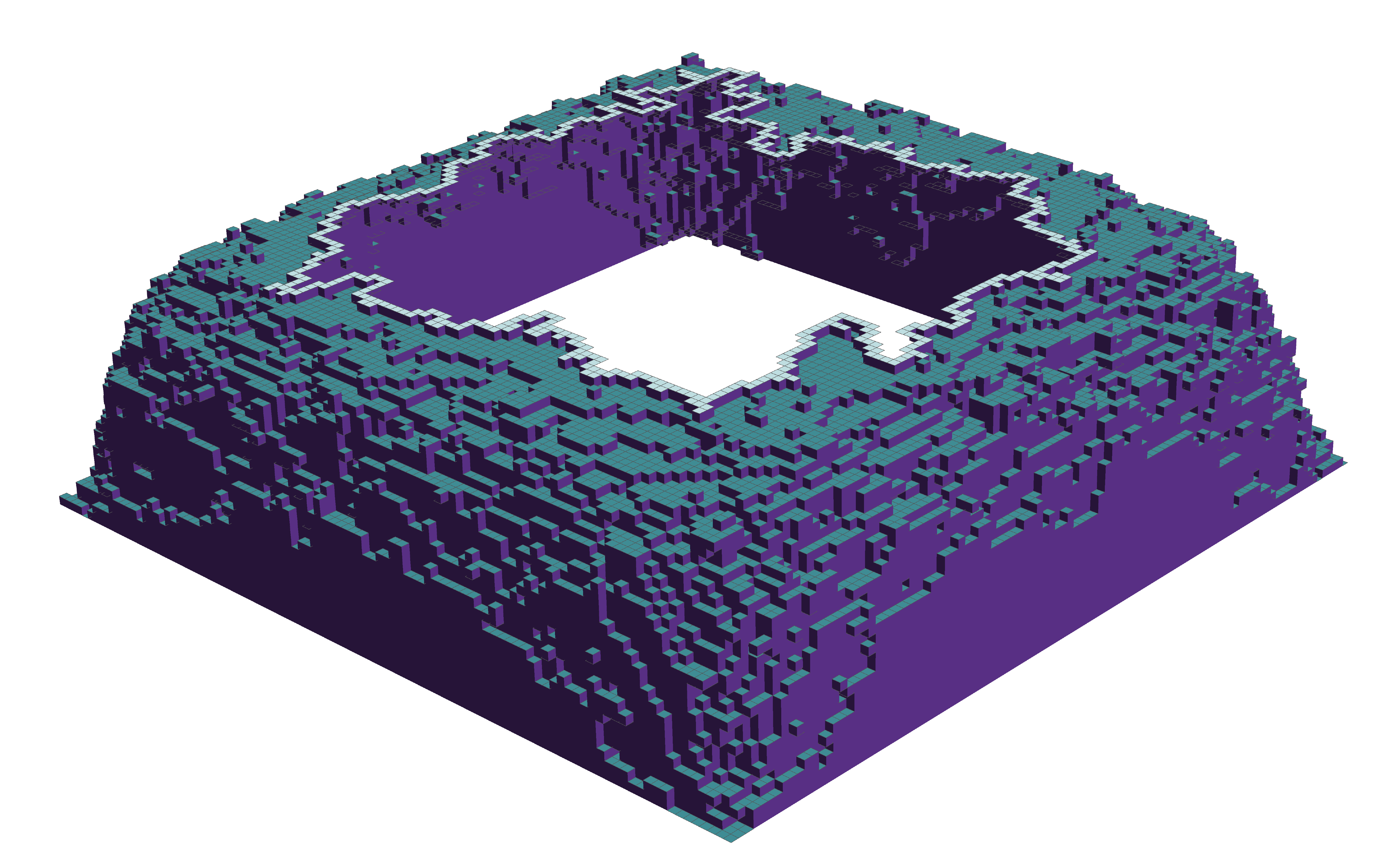}
    \hspace{-8pt}
    \raisebox{35pt}{\includegraphics[width=0.44\textwidth]{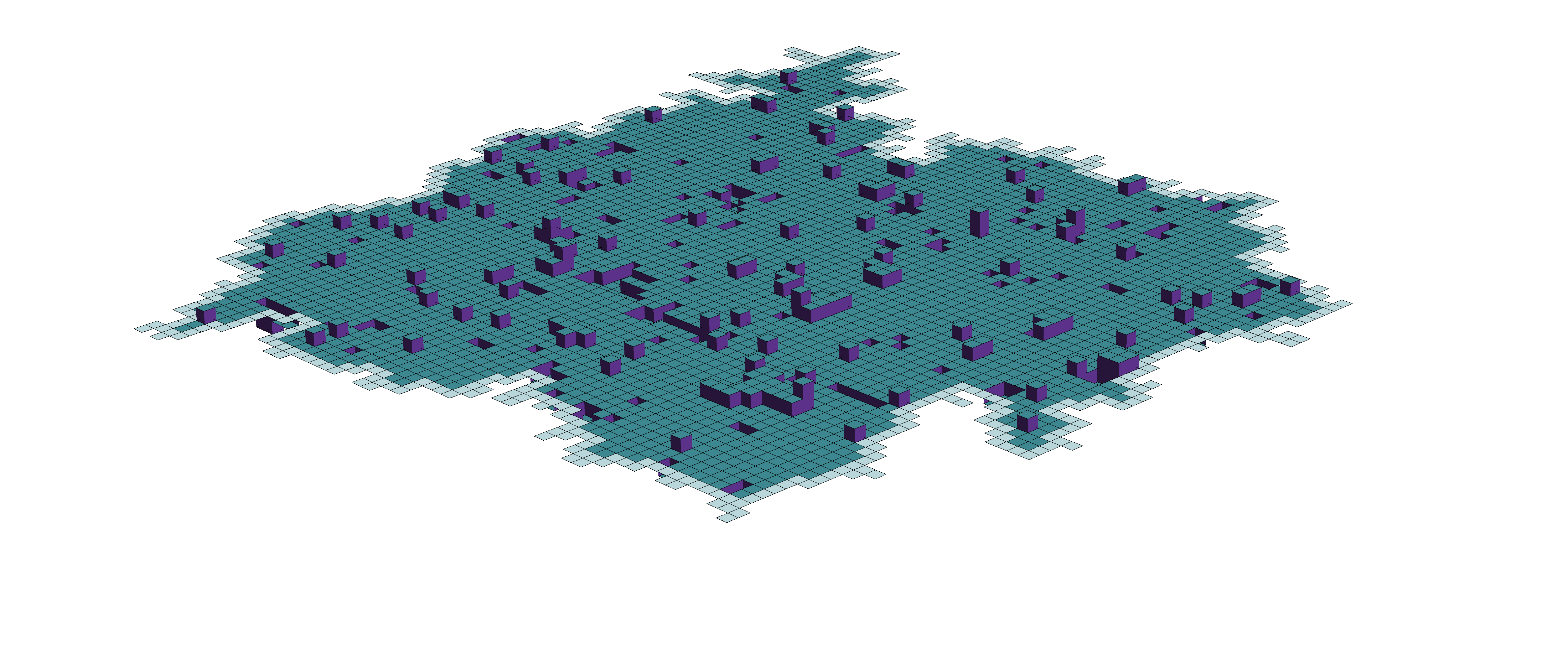}}
    \caption{Left: $\cI\cap (S_n^c\times \Z)$ consists of $\bW_n$ supporting $\cC_{\bW_n}$ (whose boundary faces are highlighted) at height $\hgt(\cC_{\bW_n})=30$. Right: $\cI\cap (S_n\times \Z)$ shifted down by~$\hgt(\cC_{\bW_n})$.}
    \label{fig:hole-in-ceiling}
\end{figure}

\begin{maintheorem}\label{thm:gumbel}
For every $d> 2$ there exists $\beta_0>0$ so that the following holds for all $\beta>\beta_0$. For every sequence of simply-connected sets $S_n\subset \cL_{0,n}$ with $\isodim(S_n)\leq d$ and $s_n:=|S_n|\to\infty$ as $n\to\infty$, and every set of walls $\bW_n = ( W_z)_{z\notin S_n}$ such that $\rho(\bW_n)\cap S_n = \emptyset$, we have for every fixed $k\geq 1$ and large enough~$n$,
\begin{align}\label{eq:M-right}
(1-\epsilon_\beta)\gamma
 \exp\left(-\overline\alpha \, k\right)&\leq \mu_{\Lambda_n}^\mp\left(\bar M_{S_n} -m_{s_n}^* \geq \;\;k \;\given \bI_{\bW_n} \right)
\leq (1+\epsilon_\beta)\gamma \exp\left(-\alpha_k\right)\,,\\
\label{eq:M-left}
e^{-(1+\epsilon_\beta)\gamma \exp\left(\overline\alpha\, k\right)}&\leq \mu_{\Lambda_n}^\mp\left(\bar M_{S_n} -m_{s_n}^* < -k\given \bI_{\bW_n} \right)
\leq e^{-(1-\epsilon_\beta) \gamma \exp\left(\alpha_k\right)}\,,
\end{align}
with $m_{s_n}^*$ as in~\eqref{eq:m*-def}, constants $\epsilon_\beta,\overline\alpha>0$ depending only on $\beta$, satisfying $\epsilon_\beta\downarrow 0$ and $\overline\alpha\downarrow \alpha$ as $\beta\uparrow\infty$, and 
\begin{align}
\label{eq:gamma-def}
\gamma &:= s_n \exp(-\alpha_{m_{s_n}^*})\in (e^{-2\beta-\epsilon_\beta},e^{2\beta})\,.
\end{align}
\end{maintheorem}

\begin{remark}\label{rem:thm-1-k=0}
The estimates~\eqref{eq:M-right}--\eqref{eq:M-left} actually hold for every $1 \leq k = k_n \leq \frac1{\beta^2}\log s_n$.
We further find that~\eqref{eq:M-left} also holds for $k=0$, and so $\gamma$ from the above theorem is such that
\[ (1-\epsilon_\beta)\gamma \leq -\log\mu_{\Lambda_n}^\mp(\bar M_{S_n} < m^*_{s_n} \mid \bI_{\bW_n}) \leq (1+\epsilon_\beta)\gamma\,.\]
\end{remark}

\begin{remark}
The same proof, with $\bW_n = \{\varnothing\}_{z\in S_n^c}$, gives the same Gumbel tail bounds~\eqref{eq:M-right}--\eqref{eq:M-left} for $M_{S_n}$ under the (unconditional) $\mu_{S\times \Z}^\mp$, to which, we recall, a true DMP would couple $\cI\cap (S\times \Z) - (0,0,\hgt(\cC_{\bW_n}))$ under $\mu_{\Lambda_n}^{\mp}(\cdot \mid \bI_{\bW_n})$. We recover the tightness and Gumbel tails of $M_{\cL_{0,n}}$~\cite{GL19b} as the special case $S_n  = \cL_{0,n}$. 
\end{remark}

\begin{figure}
    \centering
    \includegraphics[width=0.52\textwidth]{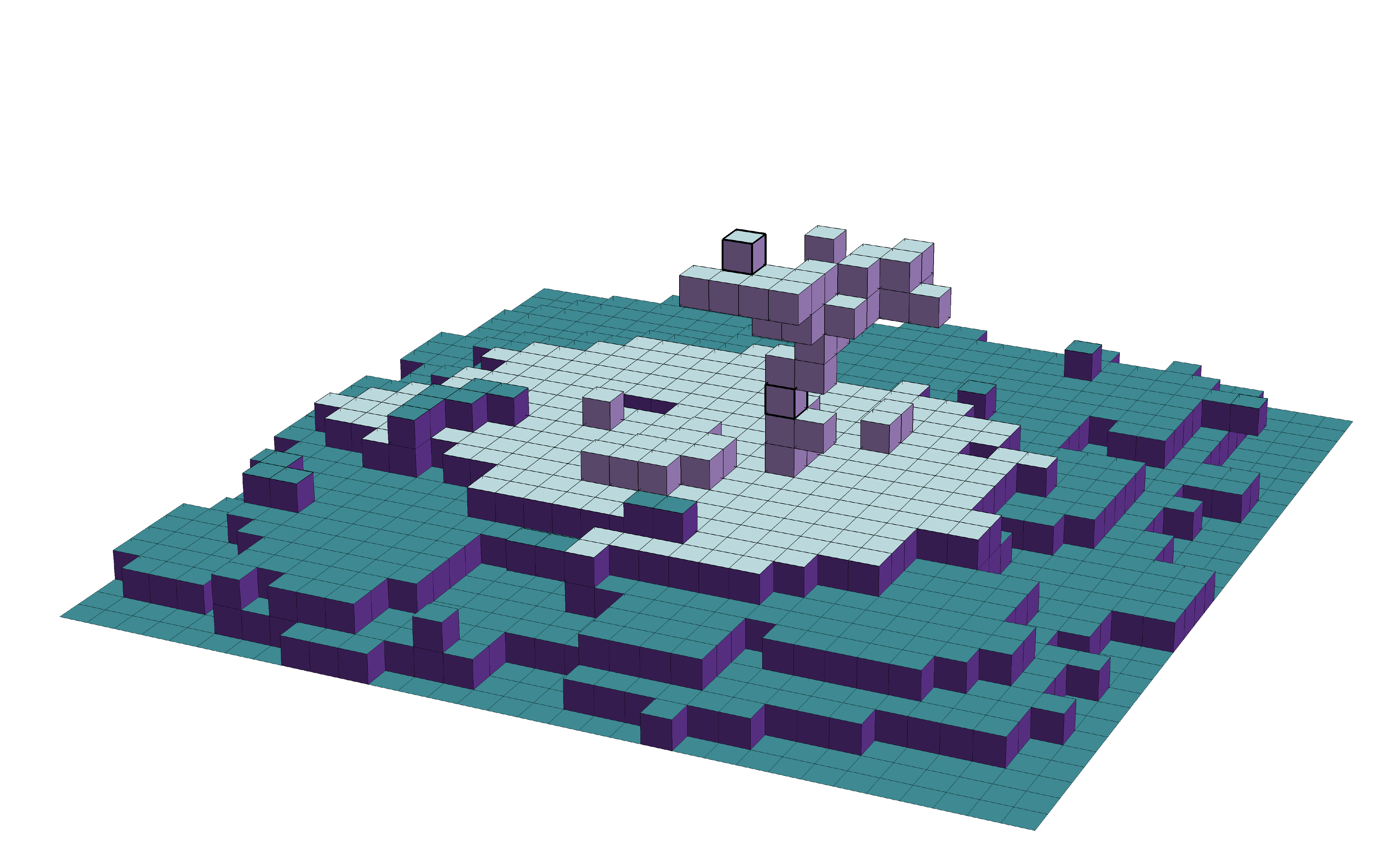}
    \hspace{-.35in}
    \includegraphics[width=0.48\textwidth]{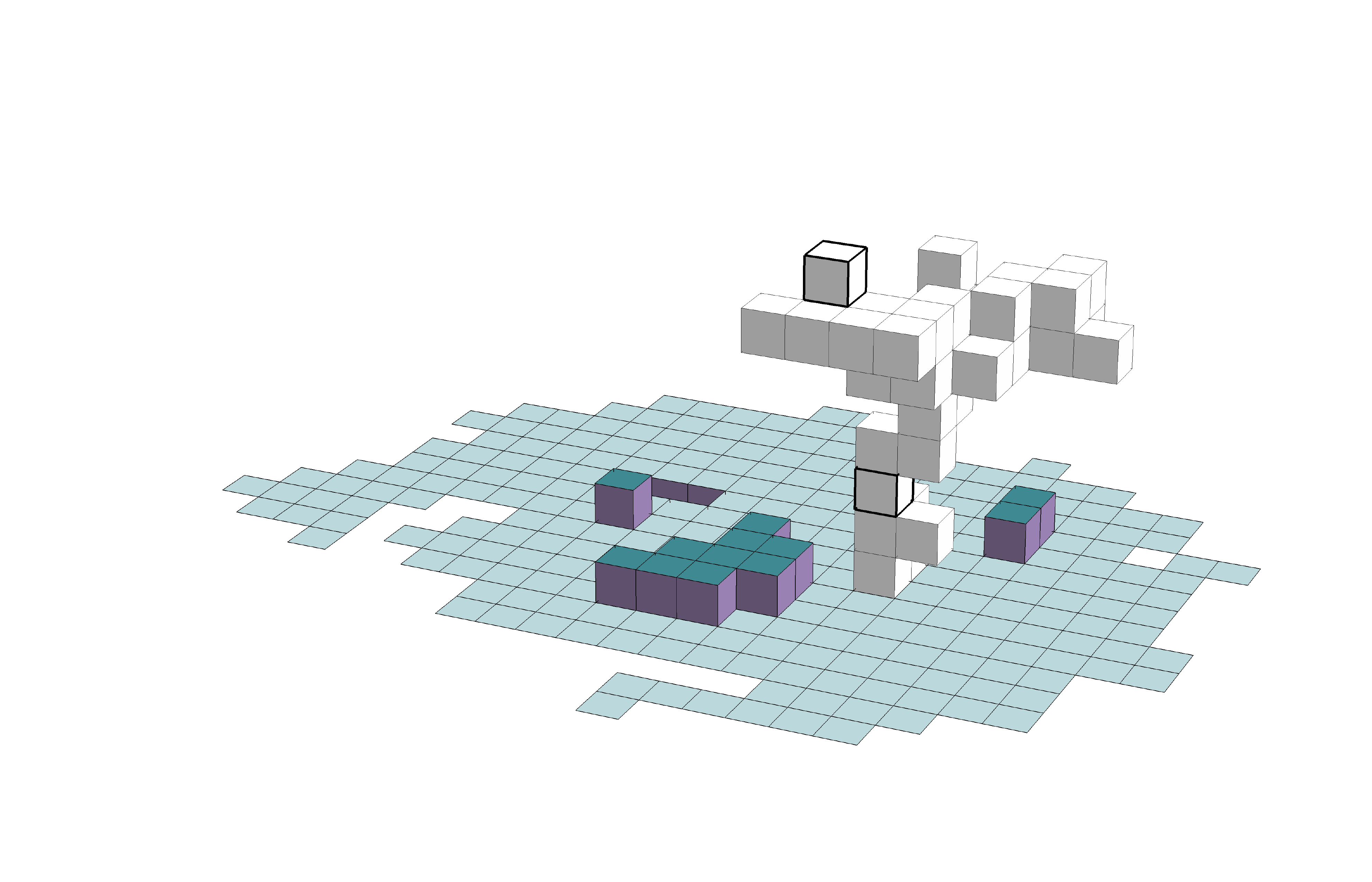}
    \caption{Left: $\cI \cap (S_n \times \Z)$ (highlighted) has a tall pillar above height $\hgt(\cC_{\bW_n}) = 3$. Right: the restricted pillar $\cP_{x,S}$ (white) is extracted from $\cI \cap (S_n \times \Z)$.}
    \label{fig:shifted-pillar}
\end{figure}

\subsubsection*{Approximate domain Markov for pillars}
We next give a version of the approximate DMP for the individual oscillations of the interface within $S_n$. While we could pursue this at the level of individual walls, we choose to work with a decomposition of the interface into \emph{pillars}, which were introduced in~\cite{GL19a} and are more tailored to studying the height profile of the interface.  Let us define \emph{restricted pillars} of the restricted interface $\cI \cap (S\times \Z)$, which generalizes the notion of pillars from~\cite{GL19a}; see Fig.~\ref{fig:shifted-pillar}. 

\begin{definition*}[Restricted pillars]
Fix a simply-connected $S_n\subset \cL_{0,n}$, and a set of walls $\bW_n = \{W_z: z\in S_n^c\}$ such that $S_n \cap \rho(\bW_n) = \emptyset$, and inducing $\hgt(\cC_{\bW_n})$. For every interface $\cI\in \bI_{\bW_n}$, the \emph{restricted pillar} of $x\in S_n$, denoted $\cP_{x,S_n}=\cP_{x,S_n}(\cI)$, 
is the following subset of faces of $\cI$ with a marked root face.
\begin{enumerate}[(i)]
\item Let $\sigma(\cI)$ be the (unique) spin configuration on $\Lambda_n$ such that $f\in \cI$ if and only if it separates sites $u,v$ having differing spins under $\sigma(\cI)$. 
 \item Let $x'=x+(0,0,\hgt(\cC_\bW))$ and let $\sigma(\cP_{x,S_n})$ be the $*$-connected plus-component of $x' + (0,0, \frac12)$  in $\sigma(\cI)$ restricted to $\cL_{>\hgt(\cC_{\bW})}=\{z:z_3 >\hgt(\cC_\bW)\}$. 
The pillar $\cP_{x,S_n}$, viewed both as a subset of $\cI$ and as a face set rooted at $x'$ (modulo translations in $\Z^3$), is the set of bounding faces of $\sigma(\cP_{x,S_n})$ in $\cL_{>\hgt(\cC_{\bW})}$.
 \end{enumerate}
The height $\hgt(\cP_{x,S_n})$ is defined as $\max\{x_3:(x_1,x_2,x_3)\in \cP_{x,S}\} - \hgt(\cC_\bW)$. For an interface $\cI$, we let $\cP_x = \cP_{x,\cL_{0,n}}$ denote the (unrestricted) pillar wherein we take $S_n = \cL_{0,n}$ and thus $\hgt(\cC_{\bW_n}) = 0$. 
\end{definition*}

Observe that, in terms of this definition, $\bar M_{S_n} $ in Theorem~\ref{thm:gumbel} is nothing but $ \max\{\hgt(\cP_{x,S_n}):x\in S_n\}$.

\begin{maintheorem}\label{thm:cond-uncond-simple}
There exist $\beta_0>0$ and $\epsilon_\beta\downarrow 0$ as $\beta\uparrow\infty$ such that the following hold for every~$\beta>\beta_0$.
Let $S_n\subset \cL_0$ be simply-connected, with $s_n = |S_n|\to\infty$ with $n$,
and let $\bW_n  = (W_z)_{z\notin S_n}$ be a wall set with $S_n \cap \rho(\bW_n) = \emptyset$. 
For every $x=x_n\in S_n$ and $1 \le h=h_n\le o(\Delta_n)$, where $\Delta_n := d(x,\partial S_n)$, we have
\begin{align}
    \label{eq:cond-uncond-coupling-1}
1-\epsilon_\beta \leq \frac{\mu_{\Lambda_n}^\mp\bigl(\hgt(\cP_{x,S_n})\geq h \mid \bI_{\bW_n}\bigr)}{ \mu_{\Z^3}^\mp(\hgt(\cP_o)\geq h)} \leq 1+\epsilon_\beta \,,
\end{align}
where $o\in \cL_{0,n}$ is a fixed \emph{origin} face, say $(\frac 12,\frac 12, 0)$. Moreover, 
\begin{align}
    \label{eq:cond-uncond-coupling-2}
\left\| \mu_{\Lambda_n}^\mp\bigl(\cP_{x,S_n}\in\cdot  \mid \hgt(\cP_{x,S_n})\geq h\,,\,
\bI_{\bW_n}\bigr) - \mu_{\Z^3}^\mp\bigl(\cP_o\in\cdot \mid \hgt(\cP_o)\geq h\bigr) \right\|_{\tv} \leq \epsilon_\beta \,.
\end{align}
\end{maintheorem}
 A routine comparison of  $\{\hgt(\cP_o)\geq h\}$ to the event in~\eqref{eq:alpha-alpha-h-def} (see~\cite[Eq.~(5.3)]{GL19b}) 
then gives the following.

\begin{corollary}[LD rate of restricted pillars]\label{cor:cond-pillar-rate}
In the setting of Theorem~\ref{thm:cond-uncond-simple}, if $x=x_n\in S_n$ is such that $\log s_n = o(d(x,\partial S_n))$ then for all $n$ and all $1\leq h= h_n\leq O(\log s_n)$, 
\begin{align*}
(1-\epsilon_\beta)e^{ - \alpha_h} \leq \mu^\mp_{\Lambda_n} \bigl(\hgt(\cP_{x,S_n})\geq h \given \bI_{\bW_n}\bigr) \leq (1+\epsilon_\beta)e^{-\alpha_h}  \,.
\end{align*}
\end{corollary}
In addition, iterating Theorem~\ref{thm:cond-uncond-simple} immediately gives the following $k$-point coupling to infinite-volume.

\begin{corollary}[$k$-point decorrelation]\label{cor:cond-uncond-k-simple}
Fix $k\geq 2$. In the setting of Theorem~\ref{thm:cond-uncond-simple}, if $x_1,\ldots,x_k\in S_n $ and $\Delta_n:=\min_i d(x_i,\partial S_n)$ are such that $d(x_i,x_j)\geq \Delta_n$ for all $i\neq j$, then for $1\leq h_1\leq\ldots\leq h_k = o(\Delta_n)$,
\begin{align}
    \label{eq:cond-uncond-k-coupling-1}
1-\epsilon_\beta \leq \frac{\mu_{\Lambda_n}^\mp\big(\bigcap_{i=1}^k \{ \hgt(\cP_{x_i,S_n})\geq h_i\} \mid\bI_{\bW_n}\big)}{ \prod_{i=1}^k \mu_{\Z^3}^\mp(\hgt(\cP_o)\geq h_i)} \leq 1+\epsilon_\beta \,,
\end{align}
and
\begin{align}
    \label{eq:cond-uncond-k-coupling-2}
\Bigl\| \mu_{\Lambda_n}^\mp\Big(\big(\cP_{x_i,S_n})_{i=1}^k \in\cdot  \given \mbox{$\bigcap_{i=1}^k \{ \hgt(\cP_{x_i,S_n})\geq h_i\}$} \,,\,\bI_{\bW_n} \Big) - \prod_{i=1}^k \mu_{\Z^3}^\mp\left(\cP_o\in\cdot\mid \hgt(\cP_o)\geq h_i\right) \Bigr\|_{\tv} \leq \epsilon_\beta \,.
\end{align}
\end{corollary}

\subsection{Method of proof}\label{subsec:methods} We now outline some of the key innovations in the proofs of Theorems~\ref{thm:gumbel}--\ref{thm:cond-uncond-simple}.

\subsubsection*{Conditional rigidity}
The first step of the proof is reproving the rigidity of Ising interfaces in $S_n$, conditionally on $\bI_{\bW_n}$. The general approach to proving rigidity considered a Peierls--type map whereby a wall $W_x$ is deleted, and its interior ceilings and walls are correspondingly shifted vertically. 

After an application of the map, the $\mu_{\Lambda_n}^{\mp}$-weight gained by the deletion of $W_x$ is compared to the entropy from the choices of such a wall.
However, if $W_x$ is close to some large wall $W_y$, the interior of $W_x$ may interact with $W_y$ through the sub-critical bubbles in the full Ising configuration, and this interaction may even be larger than the energy gain of $|W_x|$. 
Dobrushin's remedy to this~\cite{Dobrushin72a} was to delete additional walls via a subtle notion of a \emph{group-of-walls}, balancing the need to control the interaction with other walls against the multiplicity cost when deleting them.
This grouping of walls, essentially verbatim, has been central to proofs of rigidity in other low-temperature separating surfaces in dimension $d\ge 3$, such as the Widom--Rowlinson model~\cite{BLOP79a,BLP79b}, Falicov--Kimball model~\cite{DMN00}, and percolation and random-cluster/Potts models~\cite{GielisGrimmett}. 

In our conditional setting, the group-of-walls containing $x$ may ``bypass'' $\partial S_n$ to contain walls of $\bW_n$, whereby the deletion of the group-of-walls will take us outside $\bI_{\bW_n}$, breaking the argument. We introduce an alternative \emph{one-sided} grouping criterion: a wall $W$ will only be grouped with $W'$ if $W$ is interior to the external boundary of $W'$. This directed notion leads to a new grouping of walls
that, in particular, remains confined to $S_n$ (see Definition~\ref{def:wall-cluster}), and enables us to prove the following.

\begin{maintheorem}\label{thm:intro-rigidity-inside-ceiling}
There exist $\beta_0>0$ and $C>0$ such that the following holds for every~$\beta>\beta_0$.
Let $S_n\subset \cL_{0,n}$ be simply-connected, and let $\bW_n  = (W_z)_{z\notin S_n}$ be such that $S_n \cap \rho(\bW_n)=\emptyset$. 
For every $x=(x_1,x_2,0)\in S_n$, 
\[
\mu_{\Lambda_n}^\mp\left( \max\big\{x_3:(x_1,x_2,x_3)\in\cI\big\} \geq \hgt(\cC_{\bW_n})+h \given \bI_{\bW_n} \right) \leq \exp\big(-(4\beta-C)h\big) \qquad \mbox{for every $h\ge 1$}\,.
\]
\end{maintheorem}

\subsubsection*{Conditional law of tall pillars}With Theorem~\ref{thm:intro-rigidity-inside-ceiling} in hand, we are interested in comparing the rate of the exponential in Theorem~\ref{thm:intro-rigidity-inside-ceiling} with the unconditional rate~\eqref{eq:ld-rate-height}. Indeed the behavior of the maximum oscillations in $S_n$ are governed by the large deviation rate for its height profile in its \emph{bulk}, i.e., when $d(x,S_n^c)\gg h$.  

The decorrelation estimates of~\cite{Dobrushin73,BLP79b} adapted to pillars in~\cite[Corollary 6.6]{GL19a} imply the covariance bound 
\begin{align*}
    \|\mu_{\Lambda_n}^{\mp} (\cP_{x,S_n} \in\cdot ,  \bI_{\bW_n}) - \mu_{\Lambda_n}^{\mp} (\cP_{x}\in \cdot) \mu_{\Lambda_n}^{\mp}(\bI_{\bW_n}) \|_\tv \le \exp (  -   d(x,S_n^c)/C)\,;
\end{align*}
however, the bound on the conditional law, obtained via dividing by $\mu_{\Lambda_n}^{\mp}(\bI_{\bW_n})$, would leave the right-hand side as $e^{ -( d(x,S_n^c)+ \beta |\bW_n|)/C}$ which is useless since $d(x,S_n^c)\le |\partial S_n| \leq |\bW_n|$ in any nontrivial situation. 

Our approach to proving Theorem~\ref{thm:cond-uncond-simple} is to construct a bijection on pairs of interfaces $(\cI,\cI')$ that swaps the pillar $\cP_{x,S_n}$ in the interface $\cI\in \bI_{\bW_n}$ with the pillar $\cP_{x'}$ in the interface $\cI'$, to obtain two new interfaces. One could deduce Theorem~\ref{thm:cond-uncond-simple}  if the product of probabilities of $(\cI,\cI')$ were always within a factor of $1\pm \epsilon_\beta$ of the product of probabilities of the resulting pair after the swap. However, not only is this false, but the swap operation may not even be well-defined. To ensure the validity of this swap, and then control the interactions of the pillars with their new environments, we introduce a notion of an \emph{isolated pillar} (see Definition~\ref{def:isolated-pillars}). 

Informally, an interface is said to have an \emph{isolated pillar} at $x$ if (a) for height $K$ (some large constant), the pillar is simply a straight column of width $1$, then grows moderately while being confined to a cone, and (b) the walls are empty in a disk of radius $K$ about $x$, and then grow at most logarithmically in their distance to $x$. It turns out that if $\cI\cap (S_n \times \Z)$ has an isolated restricted pillar $\cP_{x,S_n}$, and $\cI'$ has an isolated pillar~$\cP_{x'}$, then (i) the pillar swap  described above is well-defined, and (ii) the tilt induced by the change in interactions between the pillars and their environments through sub-critical bubbles is $1\pm \epsilon$ (see Theorem~\ref{thm:cond-uncond}). Of course, one must also show that the isolated events have probability at least $1- \epsilon$ under both the conditional and unconditional measures: this is established by Theorem~\ref{thm:shape-inside-ceiling}, which entails adapting the shape theorem of tall pillars used for tightness~\cite[Theorem 4]{GL19b} to tall restricted pillars conditionally on $\bI_{\bW_n}$.

\subsubsection*{Law of the conditional maximum height oscillation}
In~\cite{GL19b}, the tightness and Gumbel tails of $M_{\cL_{0,n}}$ were proved using a modified second moment method, and a multiscale coupling of the maximum on $\cL_{0,n}$ to the maximum on independent copies at smaller scales. This coupling importantly relied on the aforementioned covariance bounds of~\cite{Dobrushin72b,BLP79b}; in our conditional setting, we do not have access to general covariance estimates, only the single-pillar decorrelation of Theorem~\ref{thm:cond-uncond-simple}. We instead follow a different approach similar to that of~\cite{CLMST14,LMS16} for the maximum heights in the (2+1)D SOS and DG models: there, the approach was to use FKG to exponentiate the pillar LD rate for the upper bound on the maximum, and to plant typical tall pillars wherever possible for the lower bound. Unfortunately, the conditional measure $\mu_{\Lambda_n}^{\mp}(\cdot \mid \bI_{\bW_n})$ does \emph{not} satisfy FKG, and we cannot plant pillars with the correct LD rate of $\alpha_h$, as this rate is obtained not by a fixed pillar, but by the infinite-volume distribution over random-walk-like pillars in nice environments.

We overcome these obstacles by tiling $S_n$ by smaller scale boxes and iteratively exposing the collection of restricted pillars in each of these boxes: this revealing procedure is designed to enable application of Corollary~\ref{cor:cond-pillar-rate}, by (1) only having revealed walls external to the next pillar we consider, and (2) ensuring the revealed walls are at distance $\gg h$ from the next pillar we consider. See  \S\ref{subsec:thick-ceils} for more details.

We emphasize that unlike~\cite{GL19b}, we never appeal to the FKG inequality; this suggests that the tightness and Gumbel tails of the (restricted) maximum can, in principle, be extended to other models e.g., wired-wired random-cluster and ordered-ordered Potts separating surfaces in $d\ge 3$, to which the rigidity results of~\cite{Dobrushin72a} have been extended~\cite{GielisGrimmett}, but whose underlying measures do not satisfy the FKG inequality.

\subsection{Future applications}
Whereas having a macroscopic height-$h$ level line for $h\neq 0$ is exponentially  unlikely under $\mu_{\Lambda_n}^\mp$, it becomes the typical scenario in various settings of physical interest. 
E.g., in the presence of a hard barrier (either conditioning on the interface to be in the upper half-space, or placing plus boundary conditions in the lower half-space), the Ising model is expected to exhibit \emph{entropic repulsion}, a ubiquitous feature of random surfaces: the presence of a hard barrier drives the bulk of the interface to a \emph{typical} height $h_n \gg 1$ as in Fig.~\ref{fig:hole-in-ceiling} (see the arguments sketched in~\cite{Basuev,HolickyZahradnik}). 
Basic features, such as the asymptotics of the new typical height $h_n$ and the shape of the $h_n$-level line, remain unknown. As an ingredient, one would need to understand the law of the oscillations inside a level line, our focus in this work.

The following picture is known for the $(2+1)$D Solid-On-Solid (SOS) model, a distribution on integer valued height functions, viewed as random surfaces, approximating the low-temperature 3D Ising interface.
The classical work of Bricmont, El-Mellouki and Fr\"ohlich~\cite{BEF86} showed that if $H_n$ is the height at the origin in the SOS model, and $\widehat H_n$ is the analogous variable after conditioning that the surface is nonnegative (a hard barrier), then at low enough temperature, 
while $H_n=O_{\textsc{p}}(1)$ one has that $(c/\beta)\log n \leq \E \widehat H_n \leq (C/\beta)\log n$. As part of a detailed analysis of the model in~\cite{CLMST12,CLMST14,CLMST16}, it was shown that if $M_n$ is the maximum height, and $\widehat M_n$ is its analog conditional on the surface being nonnegative, then for some deterministic sequence $h_n\asymp \log n$,
\begin{equation}\label{eq:1:2:3-SOS} \widehat H_n / h_n \xrightarrow{\,\mathrm p\,} 1\,,\qquad
M_n / h_n \xrightarrow{\,\mathrm p\,} 2\,,\qquad
\widehat M_n / h_n \xrightarrow{\,\mathrm p\,} 3\,.
\end{equation}
Theorem~\ref{thm:gumbel} is consistent with this picture in the context of the 3D Ising model, and moreover reduces the task of establishing this 1:2:3 scaling for heights in the 3D Ising interface to the analysis of the ratio $\widehat H_n/h_n$. For instance, if one showed that its $h_n$-level set is macroscopic (bounding a fixed proportion of the sites, as is the case for the SOS model), then Theorem~\ref{thm:gumbel} would immediately imply the lower bound $\widehat M_n \geq h_n + M'_n$, where $M'_n$ has the same centering and Gumbel tails as the unconditional maximum $M_n$. 

Indeed, in the follow-up work~\cite{GL21} we use the results of this paper as a key element towards understanding the phase transition between rigidity at height $0$ and entropic repulsion, related to the discussion above. Let~$\mufloor_n^\sfh$ be the Ising measure conditionally on its interface lying above a hard floor at height $-\sfh$, for~$\sfh \ge 0$. 
The main result identifies the critical $\sfh$ below which the floor would induce entropic repulsion: letting
\begin{equation}\label{eq:h-star-def}
h_n^* = h_n^* (\beta) : = \inf\{ h\geq 1 \,:\; \alpha_h > \log (2n) - 2\beta\}\,,
\end{equation}
we establish there that for $\beta>\beta_0$, with probability $1-o(1)$, the interface $\cI\sim\mufloor_n^\sfh$ satisfies
\begin{align}
\label{eq:A-subcritical-h}
\big|\cI \cap (\llb- \tfrac{n}{2},\tfrac{n}{2}\rrb^2\times\{0\})\big| &< \epsilon_\beta \, n^2   \qquad\qquad \mbox{if $\sfh < h_n^*-1$}\,, \\
\label{eq:A-supercritical-h} 
 \big|\cI\cap (\llb- \tfrac{n}{2},\tfrac{n}{2}\rrb^2\times\{0\})\big| &> (1-\epsilon_\beta)n^2 \qquad \mbox{if $\sfh \geq h_n^*$}\,.
\end{align}

Theorem~\ref{thm:gumbel} also has implications for boundary conditions that are slightly \emph{tilted}. Consider, e.g., boundary conditions forming a single step ---the plus/minus split is about height $0$ in the left half-space $\Z_-\times \Z^2$ and height~$1$ in the right half-space. Miracle-Sol\`e~\cite{Miracle-Sole1995} showed that this interface (and more generally, the interface for an angled 1-step boundary) typically comprises two macroscopic ceilings (at heights $0$ and $1$); the new result shows that each of these ceilings would be rigid and feature the same oscillations as the unconditional Ising model under a flat boundary condition. Theorem~\ref{thm:gumbel} shows that here the centered maximum $M_n$ will be the maximum of two i.i.d.\ tight random variables with Gumbel tails. More generally, if we have a $k$-step boundary condition, $M_n$ will be the maximum of $k$ independent tight r.v.'s, one for each of the ceilings.

 \subsection{Organization of the paper}
 In \S\ref{sec:preliminaries}, we recall the main notations, definitions, and properties we use from preceding works, primarily~\cite{Dobrushin72a} and~\cite{GL19a,GL19b}. In \S\ref{sec:wall-clusters}, we define the directed grouping of walls, and use it to prove Theorem~\ref{thm:intro-rigidity-inside-ceiling}. In \S\ref{sec:tall-pillar-shape}, we define isolated pillars and establish that tall pillars are typically isolated. In \S\ref{sec:pillar-couplings}, we use a swap map on pairs of interfaces with isolated pillars to establish Theorem~\ref{thm:cond-uncond-simple}. In \S\ref{sec:fluctuations-in-ceiling}, we establish Theorem~\ref{thm:gumbel}. Finally, in \S\ref{sec:deferred-proofs}, we compile proofs of technical lemmas whose proofs we deferred from \S\ref{sec:tall-pillar-shape}.

\section{Preliminaries}\label{sec:preliminaries}
In this section, we lay out notation we use throughout the paper, then recall the definitions of the Ising interface, and its decomposition into walls and ceilings from~\cite{Dobrushin72a}, leading to its standard wall representation. We then recall the definitions of pillars as introduced in~\cite{GL19a}, and their decomposition into a base and a spine composed of increments. We conclude the section with the key input from cluster expansion which we rely on, giving an approximate expression for the projection of the Ising distribution $\mu_{\Lambda_n}^\mp$ onto interfaces.  

\subsection{Notation}\label{sec:notation} In this section we compile much of the notation used throughout the paper. 

\subsubsection*{Lattice notation}
Because our primary interest in this paper is the interface separating plus and minus sites, it will be convenient for us to consider the Ising model on the vertices of the \emph{dual} graph $(\Z^3)^*= (\Z+\frac 12)^3$. 
To be precise, let $\Z^3$ be the integer lattice graph with vertices at $(x_1,x_2,x_3)\in \Z^3$ and edges between nearest neighbor vertices (at Euclidean distance one). A \emph{face} of $\Z^3$ is the open set of points bounded by four edges (or four vertices) forming a square of side length one, lying normal to one of the coordinate directions. A face is \emph{horizontal} if its normal vector is $\pm e_3$, and is \emph{vertical} if its normal vector is one of $\pm e_1$ or $\pm e_2$. 
A \emph{cell} or \emph{site} of $\Z^3$ is the set of points bounded by six faces (or eight vertices) forming a cube of side length one. 

We will frequently identify edges, faces, and cells with their midpoints, so that points with two integer and one half-integer coordinate are midpoints of edges,  points with one integer and two half-integer coordinates are midpoints of faces, and points with three half-integer coordinates are midpoints of cells. A subset $\Lambda \subset \Z^3$ identifies an edge, face, and cell collection via the edges, faces, and cells of $\mathbb Z^3$ all of whose bounding vertices are in $\Lambda$; denote the resulting edge set $\sE(\Lambda)$, face set $\sF(\Lambda)$, and cell set $\sC(\Lambda)$.  

We call two edges adjacent if they share a vertex, two faces adjacent if they share a bounding edge, and two cells adjacent if they share a bounding face. 
We will denote adjacency by the notation $\sim$.
It will also be useful to have a notion of connectivity in $\R^3$ (as opposed to $\Z^3$); we say that an edge/face/cell is $*$-adjacent, denoted $\sim^*$, to another edge/face/cell if they share a bounding vertex. A set of faces (resp., edges, cells) is connected (resp., $*$-connected) if for any pair of faces (edges, cells), there is a sequence of adjacent (resp., $*$-adjacent) faces (edges, cells) starting at one and ending at the other. 

We use the notation $d(A,B)=\inf_{x\in A, y\in B} d(x,y)$ to denote the Euclidean distance in $\mathbb R^3$ between two sets $A,B$. 
We then let $B_r(x) = \{y: d(y,x)\le r\}$. When these balls are viewed as subsets of edges/faces/cells, we include all those edges/faces/cells whose midpoint falls in $B_r(x)$.

\subsubsection*{Subsets of $\Z^3$}The main subsets of $\Z^3$ with which we will be concerned are of the form of cubes and cylinders. In view of that, define the centered $2n\times 2m \times 2h$ box,  
\begin{align*}
\Lambda_{n,m,h} :=\llb -n,n\rrb \times \llb -m,m\rrb \times \llb -h,h\rrb \subset \Z^3\,,
\end{align*}
where $\llb a,b\rrb := \{a,a+1,\ldots,b-1,b\}$. We can then let $\Lambda_n$ denote the special case of the cylinder $\Lambda_{n,n,\infty}$. The (outer) boundary $\partial \Lambda$ of the cell set $\sC(\Lambda)$ is the set of cells in $\sC(\Z^3)\setminus \sC(\Lambda)$ adjacent to a cell in $\sC(\Lambda)$. 

Additionally, for any $h \in \Z$ let $\cL_h$ be the subgraph of $\Z^3$ having vertex set $\Z^2\times \{h\}$  and correspondingly define edge and face sets $\sE(\cL_h)$ and $\sF(\cL_h)$. For a half-integer $h\in \Z + \frac 12$, let $\cL_h$ collect the faces and cells in $\sF(\Z^3) \cup \sC(\Z^3)$ whose midpoints have half-integer $e_3$ coordinate $h$. Finally we use $\cL_{>h} = \bigcup_{h'>h} \cL_{h'}$ and $\cL_{<h} = \bigcup_{h'<h} \cL_{h'}$ for half-spaces.

\subsubsection*{Projections onto $\cL_0$}
For a face $f\in \sF(\Z^3)$, its \emph{projection} $\rho(f)$ is the edge or face given by $\{(x_1,x_2,0):(x_1,x_2,s)\in f \mbox{ for some $s\in \R$}\}\subset \cL_0$. Specifically, the projection of a \emph{horizontal} face is a face in $\sF(\cL_0)$, while the projection of a \emph{vertical} face is an edge in $\sE(\cL_0)$. 
The projection of a collection of faces $F$ is $\rho(F) := \bigcup_{f\in F} \rho(f)$, which may consist both of edges and faces of $\cL_0$. 

With this notation in hand, define the cylinder of radius $r$ about $x\in \cL_0$ as $\Cyl_{x,r} = \{z: \rho(z)\in B_r(x)\}$.  

\subsubsection*{Ising model}
An Ising configuration $\sigma$ on $\Lambda \subset \Z^3$ is an assignment of $\pm1$-valued spins to the cells of $\Lambda$, i.e., $\sigma \in \{\pm 1\}^{\sC(\Lambda)}$.  For a finite connected subset $\Lambda\subset \Z^3$, the Ising model on $\Lambda$ with boundary conditions $\eta \in \{\pm 1\}^{\sC(\Z^3)}$ is the probability distribution over $\sigma \in \{\pm 1\}^{\sC(\Lambda)}$ given by 
\begin{align*}
\mu_{\Lambda}^{\eta} (\sigma) \propto \exp \left[ - \beta  \cH(\sigma)\right]\,, \qquad \mbox{where}\qquad \cH(\sigma) =  \sum_{\substack{ v,w\in \sC(\Lambda) \\  v\sim w }} \one\{\sigma_v\neq \sigma_w\} +  \sum_{\substack{ v\in \sC(\Lambda), w\in \sC(\Z^3)\setminus \sC(\Lambda) \\  v\sim w}} \one\{\sigma_v\neq \eta_w\}\,. 
\end{align*}
Throughout this paper, we will be considering the boundary conditions $\eta_w = -1$ if $w$ is in the upper half-space ($w_3 > 0$) and $\eta_w = +1$ if $w$ is in the lower half-space ($w_3 <0$). We refer to these boundary conditions as \emph{Dobrushin boundary conditions}, and denote them by $\eta = \mp$; for ease of notation, let $\mu_{n,m,h}^\mp = \mu_{\Lambda_{n,m,h}}^{\mp}$.

\subsubsection*{Infinite-volume measures} Care is needed to define the Ising model on infinite graphs; infinite-volume Gibbs measures are defined via what is known as the \emph{DLR conditions}. Namely, for an infinite graph $G$, a measure $\mu_G$ on $\{\pm 1\}^G$, defined in terms of its finite dimensional distributions, satisfies the DLR conditions if for every finite subset $\Lambda \subset G$, 
\begin{align*}
\E_{\mu_G(\sigma_{G\setminus \Lambda}\in \cdot)}\big[\mu_G(\sigma_\Lambda \in \cdot\mid \sigma_{G\setminus \Lambda} )\big ] = \mu_{G}(\sigma_\Lambda\in \cdot)\,.
\end{align*}
On $\Z^d$, infinite-volume measures arise as weak limits of finite-volume measures, say $n\to\infty$ limits of the Ising model on boxes of side length $n$ with prescribed boundary conditions. At low temperatures $\beta>\beta_c(d)$, the Ising model on $\Z^d$ admits multiple infinite-volume Gibbs measures; taking plus and minus boundary conditions on boxes of side-length $n$ yield the distinct infinite-volume measures $\mu^+_{\Z^3}$ and $\mu^-_{\Z^3}$. While in $\mathbb Z^2$, all infinite-volume  measures are mixtures of $\mu_{\Z^2}^+$ and $\mu_{\Z^2}^-$, the rigidity result of~\cite{Dobrushin72a} showed that when $d\ge 3$, the weak limit $\mu_{\Z^3}^{\mp}:= \lim_{n\to\infty} \mu_{n,n,n}^{\mp}$ is a DLR measure distinct from any mixtures of $\mu_{\Z^2}^+$ and $\mu_{\Z^2}^-$.

\subsection{Interfaces under Dobrushin boundary conditions}\label{sec:dobrushin-definitions}
We begin by formally defining the interface induced by an Ising configuration with $\mp$ boundary conditions. We then recall the key combinatorial decomposition from~\cite{Dobrushin72a} into walls and ceilings. We refer the reader to~\cite{Dobrushin72a} for more details. 

\begin{definition}[Interfaces]\label{def:interface} For a domain $\Lambda_{n,m,h}$ with Dobrushin boundary conditions, and an Ising configuration
$\sigma$ on $\sC(\Lambda_{n,m,h})$, the \emph{interface} $\cI=\cI(\sigma)$
is defined as follows:  
\begin{enumerate}
\item Extend $\sigma$ to a configuration on $\sC(\mathbb Z^3)$ by taking $\sigma_v=+1$ (resp., $\sigma_v=-1$) if $v\in \cL_{<0}\setminus \sC(\Lambda_{n,m,h})$ (resp., $v\in \cL_{>0}\setminus \cC(\Lambda_{n,m,h})$). 
\item Let $F(\sigma)$ be the set of faces in $\sF(\mathbb Z^3)$ separating cells with differing spins under $\sigma$.
\item  Call the (maximal) $*$-connected component containing $\cL_0 \setminus \sF(\Lambda_{n,m,h})$ in $F(\sigma)$, the \emph{extended interface}. (This is also the unique infinite $*$-connected component in $F(\sigma)$.)
\item The interface $\cI$  is the restriction of the extended interface to $\sF(\Lambda_{n,m,h})$. 
\end{enumerate}
\end{definition}

Taking the $h\to\infty$ limit $\mu_{n,m,h}^\mp$ to obtain the infinite-volume measure $\mu_{n,m,\infty}^\mp$, the interface defined above stays finite almost surely. Thus, $\mu_{n,m,\infty}^\mp$-almost surely, the above process also defines the interface for configurations on all of $\sC(\Lambda_{n,m,\infty})$. 

\begin{remark}\label{rem:interface-spin-config}
Every interface uniquely defines a configuration with exactly one $*$-connected plus component and exactly one $*$-connected minus component. For every $\cI$, we can obtain this configuration, which we call $\sigma(\cI)$, by iteratively assigning spins to $\cC(\Lambda_{n,m,h})$, starting from $\partial \Lambda$ and proceeding inwards, in such a way that adjacent sites have differing spins if and only if they are separated by a face in $\cI$. 
 Informally, $\sigma(\cI)$ is indicating the sites that are in the ``plus phase" and ``minus phase" given the interface $\cI$.
\end{remark}

\subsubsection*{Walls and ceilings} Following~\cite{Dobrushin72a}, we can decompose the faces in $\cI$ into wall faces and ceiling faces. 

\begin{definition}[Walls and ceilings]\label{def:ceilings-walls}
A face $f\in \cI$ is a \emph{ceiling face} if it is horizontal and there is no $f'\in \cI \setminus  \{f\}$ such that $\rho(f)=\rho(f')$. A face $f\in \cI$ is a \emph{wall face} if it is not a ceiling face. 
A \emph{wall} is a (maximal) $*$-connected set  of wall faces.  A \emph{ceiling} of $\cI$ is a (maximal) $*$-connected set of ceiling faces. 
\end{definition}

Projections of walls can be viewed as (relaxed) contours in $\mathbb R^2$, defining an important notion of interior/exterior with respect to a wall. 

\begin{definition}[Nesting of walls]\label{def:nesting-of-walls}
For a wall $W$, the complement of its projection (a subset of $\R^2$) 
$$\rho(W)^c:= (\sE(\cL_0)\cup \sF(\cL_0)) \setminus \rho(W)\,,$$
splits into one infinite component, and some finite ones. We say an edge or face $u\in \sE(\cL_0)\cup \sF(\cL_0)$ is \emph{interior} to (or \emph{nested} in) a wall $W$, denoted by $u\Subset W$, if $u$ is not in the infinite component of $\rho(W)^c$ (and strictly interior to or strictly nested in, if it is in one of the finite components of $\rho(W)^c$). 
A wall $W$ is nested in a wall $W'$, denoted $W\Subset W'$, if every element of $\rho(W)$ is interior to $W'$. Similarly, a ceiling $\cC$ is nested in a wall $W'$ if every element of $\rho(\cC)$ is interior to $W'$. 
\end{definition}

We can then identify the connected components of $\rho(W)^c$ with the ceilings incident to $W$. 

\begin{lemma}[{\cite{Dobrushin72a}}]\label{lem:wall-ceiling-bijection}
For a projection of the walls of an interface, each connected component of that projection (as a subset of edges and faces) corresponds to a single wall. Moreover, there is a 1-1 correspondence between the ceilings adjacent to a standard wall $W$ and the connected components of $\rho(W)^c$. Similarly, for a wall $W$, all other walls $W'\neq W$ can be identified to the connected component of $\rho(W)^c$ they project into, and in that manner they can be identified to the ceiling of $W$ to which they are interior.
\end{lemma}

The above correspondence can be made more transparent by introducing the following notion. 

\begin{definition}
For a wall $W$, the ceilings incident to $W$ can be decomposed into \emph{interior ceilings} of $W$ (those ceilings identified with the finite connected components of $\rho(W)^c$), and a single exterior ceiling, called the $\emph{floor}$ of $W$, identified with the infinite connected component of $\rho(W)^c$. 
\end{definition}

\begin{definition}\label{def:hull}
The hull of a ceiling $\cC$, denoted $\hull{\cC}$ is the minimal simply-connected set of horizontal faces containing $\cC$. The hull of a wall, $\hull{W}$ is the union of $W$ with the hulls of its interior ceilings. Abusing notation, for a collection of walls $\bW$, we let $\hull{\bW}$ be the union of the hulls of $W\in \bW$. 
\end{definition}

\begin{observation}
For every interior ceiling $\cC$ of $W$, the projection of its hull, $\rho(\hull{\cC})$, is exactly the finite component of $\rho(W)^c$ it projects into. On the other hand, the projection of the hull of the floor of $W$ is all of $\cL_{0,n}$. Finally, the set $\rho(\hull{W})$ is the union of $\rho(W)$ with all the finite components of $\rho(W)^c$. 
\end{observation}

Finally, we can assign the index points of $\cL_0$ the walls of an interface $\cI$ as follows. 

\begin{remark}\label{rem:indexing-walls}
Given an interface $\cI$, for every face $x\in \sF(\cL_0)$, assign $x$ the wall $W$ of $\cI$ if $x\Subset W$ and $x$ shares an edge with $\rho(W)$. If there is no $W$ for which this is the case, let $W_x = \trivincr$. Importantly, this labeling scheme is such that $x$ is only assigned one wall, but the same wall may be assigned to many index faces. 
\end{remark}

\subsubsection*{The standard wall decomposition} While it is evident that the wall faces of an interface uniquely determine the interface, a key property of the walls and ceilings decomposition of~\cite{Dobrushin72a} is that the vertical positions of the walls are not needed to recover the interface. 

\begin{definition}[Standard walls]\label{def:standard-walls}
A wall $W$ is a \emph{standard wall} if there exists an interface $\cI_W$ such that $\cI_W$ has exactly one wall, $W$; as such the floor of $W$ in $\cI$ must be a subset of $\cL_0$. 
A collection of standard walls is \emph{admissible} if they have pairwise (vertex) disjoint projections. 
\end{definition}

\begin{definition}[Standardization of walls]\label{def:std-of-wall}
To each ceiling $\cC$ of $\cI$, we can identify a unique \emph{height} $\hgt(\cC)$ since all faces in the ceiling have the same $x_3$ coordinate. For every wall $W$ of $\cI$ , we can define its \emph{standardization} $\Theta_{\textsc {st}}W$ which is the translate of the wall by $(0,0,-s)$ where $s$ is the height of its floor. (The standardization of  a wall $W$ may depend on $\cI$, but we leave this dependence to be contextually understood.)
\end{definition}

We then have the following important bijection between interfaces and their standard wall representation, defined as the collection of standard walls given by standardizing all walls of $\cI$: see Fig.~\ref{fig:1-1-correspond-std} for a depiction. 

\begin{lemma}[{\cite{Dobrushin72a}, as well as~\cite[Lemma 2.12]{GL19a}}]\label{lem:interface-reconstruction}
The standard wall representation yields a bijection between the set of all interfaces and the set of all admissible collections of standard walls. 
In particular, the standardization $\Theta_{\textsc {st}}W$ of a wall $W$ is a standard wall. 
\end{lemma}

We note the following important observation based on the bijection given by Lemma~\ref{lem:interface-reconstruction}.  

\begin{observation}\label{obs:1-to-1-map-faces}
Consider interfaces $\cI$ and $\cJ$, such that the standard wall representation of $\cI$ contains that of $\cJ$ (and additionally has the standard walls $\Theta_{\textsc{st}}\bW=(\Theta_{\textsc{st}}W_1,\ldots,\Theta_{\textsc{st}}W_r)$). There is a 1-1 map between the faces of $\cI \setminus \bW$ and the faces of $\cJ \setminus \bH$ where $\bH$ is the set of faces in $\cJ$ projecting into $\rho (\bW)$. Moreover, this bijection can be encoded into a map $f\mapsto \theta_{\udarrow} f$ that only consists of vertical shifts, and such that all faces projecting into the same connected component of $\rho(\bW)^c$ undergo the same vertical shift. 
\end{observation}

\subsubsection*{Nested sequences of walls} Finally, we introduce a notion of nested sequences of walls. 
 
 \begin{definition}\label{def:nested-sequence-of-walls}
To any edge/face/cell $z$, we can assign a \emph{nested sequence of walls} $\fW_z = (W_1,\ldots,W_s)$ where $(W_1,\ldots,W_s)$ are the set of all walls in $\cI$ nesting $\rho(z)$ (by Definition~\ref{def:nesting-of-walls}, this forms a nested sequence of walls, and can be ordered such that $W_i$ is nested in $W_j$ for all $j\ge i$). 
\end{definition}

\begin{observation}\label{obs:height-nested-sequence-of-walls}
For $u\in \cL_0$, one can read off the height of the face(s) of $\cI$ projecting onto $u$ from $\Theta_{\textsc{st}}\fW_u$. In particular, if a face $f\in \cI$ has height $h$, its nested sequence of walls must be such that the sum of the heights of the walls in $\Theta_{\textsc{st}}\fW_{\rho(f)}$ exceeds~$4h$.
\end{observation}

We conclude this section with the following fact enumerating over walls and nested sequences of walls rooted at a given face (see e.g.,~\cite[Lemma 2]{Dobrushin72a} as well as~\cite[Observation 2.27]{GL19a}). 

\begin{fact}\label{fact:number-of-walls}
There exists a universal (lattice dependent) constant $C$ such that the number of $*$-connected face sets of cardinality at most $M$, nesting a fixed $x\in \sF(\mathbb Z^3)$, is at most $C^M$. In particular, there exists a universal constant $C$ such that the number of nested sequences of standard walls of total cardinality at most~$M$, nesting $x$, is at most $C^M$. 
\end{fact}

\subsection{Restricted interfaces}
Recall that much of the paper will regard statements that are conditional on the interface outside some level set. Towards that, as in the introduction, we will fix a set $S_n \subset \cL_{0,n}$ and condition on the wall collection in $S_n^c = \cL_{0,n}\setminus S_n$. Namely, let $\bW_n =(W_z)_{z\in S_n^c}$---where if $z\in S_n^c$ is not assigned any wall, we say $W_z = \trivincr$---such that that $\rho(\bW_n)\subset S_n^c$. 
If $\cI_\bW$ is the interface whose standard wall collection is $\Theta_{\textsc{st}}\bW_n$, there is a unique ceiling, called $\cC_{\bW}$, of $\cI_{\bW}$ such that $\rho(\cC_\bW)\supset S_n$.  

Our object of interest is the interface $\cI\restriction_{S_n}$: this is the interface whose standard wall representation consists of the standardizations of all walls in $\cI$ whose projection is in $S_n$ (for $\cI\in \bI_\bW$, this is the interface whose standard wall representation is that of $\cI$ minus the standardizations of $\bW_n$). 
Observe, then, that for any such $S_n, \bW_n$, every $\cI\in \bI_{\bW_n}$ can be decomposed as
\begin{align}\label{eq:I-inside-outside-decomp}
    \cI \cap (S^c \times \Z) = \cI\restriction_{S_n^c}\cap (S^c\times \Z)\,, \qquad \mbox{and}\qquad \cI \cap (S\times \Z) = \cI\restriction_{S_n} \cap (S\times \Z) + (0,0,\hgt(\cC_\bW))\,.
\end{align}
In particular, the restriction $\cI \cap (S\times \Z)$ is given by the vertical shift by $\hgt(\cC_{\bW})$ of the interface $\cI\restriction_{S_n}$.

We define restricted nested sequences of walls, $\fW_{x,S}$ as $\fW_{x}\setminus \bW$ and note that $\fW_{x,S}$ is the vertical shift by $\hgt(\cC_\bW)$ of the nested sequence $\fW_{x}$ of the restricted interface $\cI\restriction_{S_n}$.  

\subsection{Interface pillars}\label{sec:pillar-def}
The standard wall representation of interfaces was central to~\cite{Dobrushin72a}, but is too coarse to characterize sharp asymptotics of the maximum height fluctuation. Instead, a pillar decomposition was introduced by~\cite{GL19a} to study the large height deviations of the interface.

\begin{figure}
    \begin{tikzpicture}
   \node (fig1) at (0,0) {
    	\includegraphics[width=0.5\textwidth]{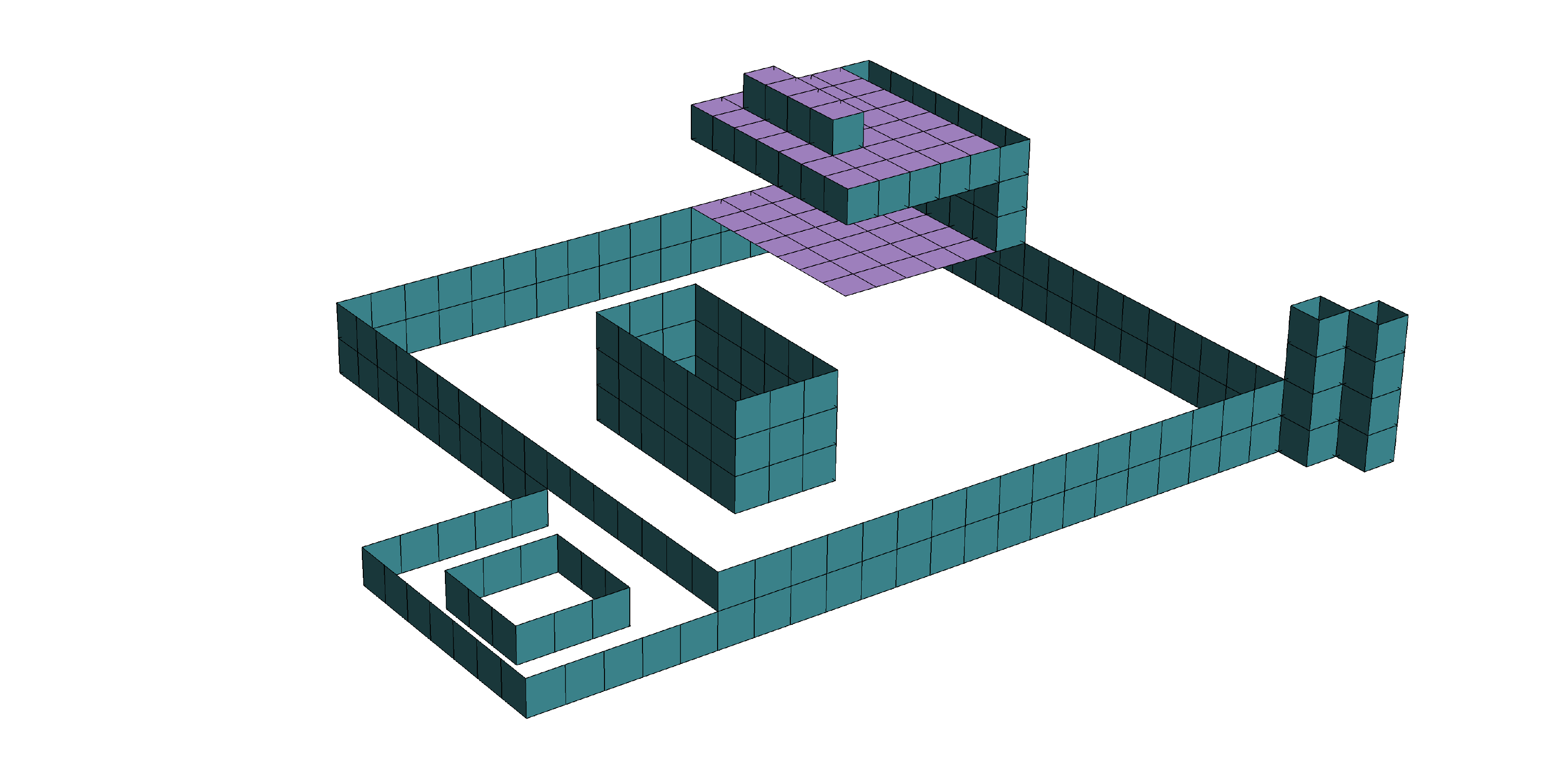}};
   \node (fig2) at (7,0) {
   	\includegraphics[width=0.48\textwidth]{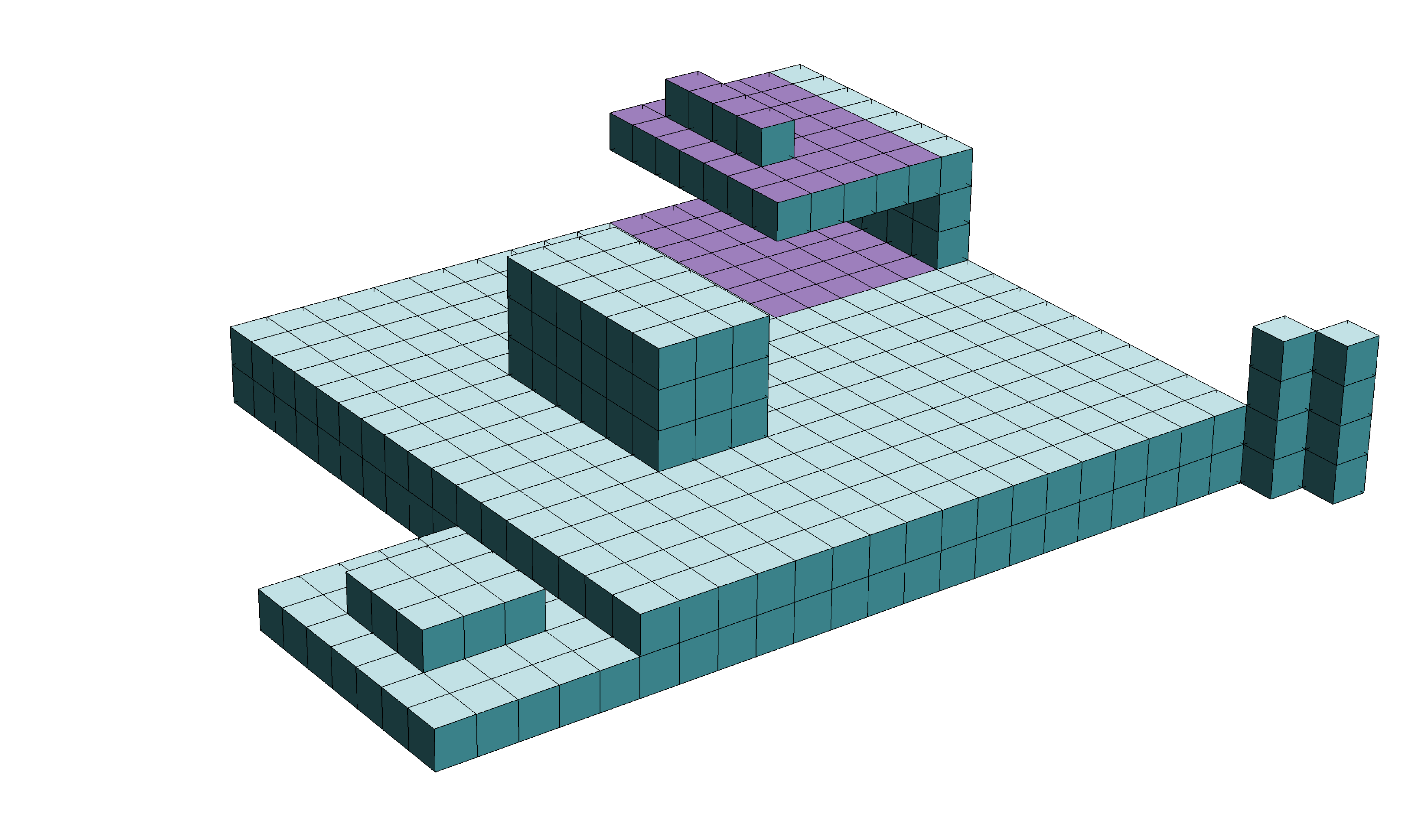}};
    \end{tikzpicture}
    \caption{Left: an admissible collection of three standard walls (with horizontal wall faces depicted in purple). Right: the interface with that standard wall representation per Lemma~\ref{lem:interface-reconstruction}, obtained by appropriate vertical shifts of the standard walls on the right.}
    \label{fig:1-1-correspond-std}
\end{figure}

\begin{definition}
For an interface $\cI$ and a face $x \in \sF(\cL_0)$, we define the \emph{pillar} $\cP_x$ as follows: consider the Ising configuration $\sigma(\cI)$ and let $\sigma(\cP_x)$ be the $*$-connected plus component of the cell with mid-point $x+(0,0, \frac 12)$ in the upper half-space $\cL_{>0}$. Define the face set $\cP_x$ as the set of bounding faces of $\sigma(\cP_x)$ in $\cL_{>0}$.  
\end{definition}

\begin{remark}
If we recall the definition of the restricted pillar from the introduction, we see that this corresponds to  $\cP_{x,\cL_{0,n}}$. More generally, we have that the restricted pillar $\cP_{x,S_n}$ is precisely the vertical shift by $\hgt(\cC_\bW)$ of the (unrestricted) pillar $\cP_x$ of the interface $\cI \restriction_{S_n}$, and by such a translation, the definitions below are extended naturally to restricted pillars.
\end{remark}

The following observation relates pillars and nested sequences of walls. 

\begin{observation}\label{obs:pillar-nested-sequence-of-walls}
Let $\fG_x$ be the union of $\fW_x$ together with all walls $(W_z)_{z\in \rho(\hull{\fW}_x)}$ nested in a wall of $\fW_x$. 
The wall faces of $\cP_x$ are a subset of $\fG_x$. In particular, if the wall collections of $\cI$ and $\cJ$ agree on $\fG_x$, then $\cP_x^\cI = \cP_x^\cJ$. Thus, if $f\in \cP_x$, there must exist $W\in \fG_x$ such that both $\rho(f)$ and $\rho(x)$ are nested in~$W$. 
\end{observation}

It will be useful to identify the height oscillations of the interface, with the pillar that attains those heights. 

\begin{definition}\label{def:heights}
For a point $(x_1,x_2, x_3)\in \R^3$, we say its height is $\hgt(x) = x_3$. The height of a cell is the height of its midpoint. For a pillar  $\cP_x \subset \cI$, its height is given by $\hgt(\cP_x) = \max\{x_3: (x_1,x_2,x_3)\in f, f\in \cP_x\}$, and for a restricted pillar that height is given by $\hgt(\cP_{x,S}) = \max\{x_3: (x_1,x_2,x_3)\in f, f\in \cP_x\} - \hgt(\cC_\bW)$.
\end{definition}

\subsubsection*{Decomposition of the pillar}We next recall the structural decomposition from~\cite{GL19b} of a pillar into a base $\sB_x$ and a spine $\cS_x$, which is further decomposed into an increment sequence $(\sX_i)_{i\ge 1}$. 

\begin{definition}[Cut-points]
A half-integer $h\in \{\frac 12, \frac 32, \ldots \}$ is a \emph{cut-height} of $\cP_x$ if $\sigma(\cP_x) \cap \cL_{h}$ consists of a single cell. In that case, that cell (identified with its midpoint $v\in (\Z+\frac 12)^3$) is a cut-point of $\cP_x$. 
We enumerate the cut-points of $\cP_x$ in order of increasing height as $v_1 , v_2 , \ldots$.  
\end{definition}

\begin{definition}[Spine and base]
The spine of $\cP_x$, denoted $\cS_x$ is the set of cells in $\sigma(\cP_x)$ (resp., faces in $\cP_x$) intersecting $\cL_{>\hgt(v_1)- \frac 12}$. The base $\sB_x$ of $\cP_x$ is the set of cells in $\sigma(\cP_x) \setminus \cS_x$ (resp., faces in $\cP_x \setminus \sF(\cS_x)$).
\end{definition}

\begin{definition}[Increments]
Consider a spine $\cS_x$ with cut-points $v_1,v_2,\ldots,v_{\sT+1}$. Then, for every $i\leq \sT$, define the $i$-th increment of $\cS_x$ as
\[\sX_i = \cS_x \cap (\R^2 \times [\hgt(v_i) - \tfrac 12,\hgt(v_{i+1})+\tfrac 12])\,,
\]
so that the $i$-th increment is the subset of $\cS_x$ delimited from below by $v_i$ and from above by $v_{i+1}$ and there are exactly $\sT$ increments. (If there are fewer than two cut-points, we say that $\sT=0$.)

Beyond the $\sT$ increments, the spine additionally may have a \emph{remainder} $\sX_{>\sT}$, which we define as the set of faces intersecting $\R^2 \times [\hgt(v_{\sT+1})-\frac 12, \infty)$. For readability, for a spine $\cS_x$ with increments $\sX_1, \ldots, \sX_{\sT}, \sX_{>\sT}$, we use the notation $\sX_{\sT+1}:= \sX_{>\sT}$ to simultaneously index over increments and the remainder.
\end{definition}

Abusing notation, we may view increments not as subsets of an interface, but as  finite $*$-connected sets of at least two cells, whose only cut-points are its bottom-most and top-most cells  (modulo lattice translations, achieved by, say, rooting them at the origin). The face-set of such an increment consists of all its bounding faces except its bottom-most and top-most horizontal ones.  A remainder increment is defined similarly, but its only cut-point is its bottom-most cell. With this, we obtain the following decomposition of the spine. 

\begin{lemma}\label{lem:spine-reconstruction}
There is a 1-1 correspondence between triplets of $v_1$, a sequence of increments  $X_1 ,\ldots, X_{T}\in \fX^{T}$ and a remainder $X_{>T}$, and possible spines of $T$ increments whose first cut-point is at $v_1$.    
\end{lemma}

The correspondence of Lemma~\ref{lem:spine-reconstruction} follows in the obvious manner, by identifying the bottom cut-point of $X_1$ with $v_1$ and sequentially translating the increments in the increment sequence to identify their bottom cut-point with the top cut-point of the previous increment. For more details, see \cite[Section~3]{GL19a}.

Finally, it will be useful to refer to the simplest increment, consisting of two vertically consecutive cells, one on top  of the other (resp., its eight bounding vertical faces) the \emph{trivial increment} denoted $X_\trivincr$.

\subsection{Excess area} 
For a pair of interfaces, we need to quantify the energy cost from having one interface over the other. The competition of this energy cost  as compared to the reference interface $\cL_{0,n}$, with the entropy gain from additional fluctuations governs the typical behavior of the interface.  

\begin{definition}[Excess area]
\label{def:excess-area}
For two interfaces $\cI,\cJ$,
the \emph{excess area} of $\cI$ with respect to $\cJ$,
denoted $\fm(\cI;\cJ)$, is given by $$\fm(\cI;\cJ):= |\cI|-|\cJ|\,.$$
Evidently, for any valid interface $\cI$, we have that $\fm(\cI; \cL_{0,n})\geq 0$.   

We can extend the definition of excess area to walls and increments. For instance, for a wall $W$, if we denote by $\cI_{\Theta_{\textsc{st}}W}$ the interface whose only wall is $\Theta_{\textsc{st}}W$, then  $\fm(W)=\fm(\cI_{\Theta_{\textsc{st}}W}; \cL_{0,n})$. The excess area of a collection of walls $\bW$ is analogously defined, and one can easily see that $\fm(\bW) = \sum_{W\in \bW} \fm(W)$.
\end{definition}

\begin{remark}\label{rem:excess-area-properties}
Notice that for a wall $W$, its excess area is exactly given by 
\begin{align*}
\fm(W) = \fm(\theta_{\textsc {st}}(W))= |W| - |\sF(\rho(W))|\,.
\end{align*}
 This form of the excess area makes a few key properties clear: 
\begin{align}\label{eq:excess-area-wall-relations}
\fm(W) \geq \frac{1}{2} |W|\,, \qquad \mbox{and} \qquad \fm(W) \geq |\rho(W)| = |\sE(\rho(W)|+ |\sF(\rho(W))|\,.
\end{align} 
Moreover, any two faces $x,y\in \cL_{0,n}$ nested in $W$ satisfy $d(x,y) \leq \fm(W)$. 
\end{remark}

Finally, we define excess areas of increments following the conventions of~\cite{GL19a,GL19b}, with respect to trivial increments. For an increment $\sX_i$ ($1\leq i \leq \sT$), define $\fm(\sX_i)$ as
 \begin{align}\label{eq:increment-excess-area}
     \fm(\sX_i) = |\sF(
     \sX_i)| - 4(\hgt(v_{i+1}) - \hgt(v_i)+1)
 \end{align}
 (recall that $\sF(X)$ does not include the top most and bottom most faces bounding $X$). This can be viewed as the difference in the number of faces from $\sX_i$ versus a stack of trivial increments of the same height as $\sX_i$. For the remainder increment $\sX_{\sT+1} = \sX_{>T}$, this can be defined consistently by arbitrarily setting $\hgt(v_{\sT+2}):=\hgt(\cP_x)- \frac 12$. With these definitions, we notice that if $\sX_i\neq X_\trivincr$ then
 \begin{align}\label{eq:increment-excess-inequalities}
     \fm(\sX_i) \geq 2(\hgt(v_{i+1})- \hgt(v_i) -1)\,\vee\, 2\qquad \mbox{and} \qquad |\sF(\sX_i)|\leq 3\fm(\sX_i)+4\,.
 \end{align}

\subsection{The induced distribution over interfaces}\label{sec:cluster-expansion}
 Using the tool of cluster expansion,~\cite{Minlos-Sinai} proved refined properties of the single-phase Ising measures $\mu^-_{\Z^3}$ and $\mu^+_{\Z^3}$ at sufficiently low temperatures. 
Dobrushin~\cite{Dobrushin72a} subsequently used these techniques to express the Ising measure over interfaces as a perturbation of the distribution $\exp(- \beta |\cI|)$, where the perturbative term takes into account the bubbles of the low-temperature Ising configurations in the minus and plus phases above and below the interface respectively. We recall this expression for the Ising distribution over interfaces in what follows. 

Here and throughout the paper, let $\mu_n^{\mp} = \mu_{\Lambda_n}^\mp = \lim_{h\to\infty} \mu_{n,n,h}^\mp$.

\begin{theorem}[{\cite[Lemma 1]{Dobrushin72a}}]
\label{thm:cluster-expansion} There exists $\beta_0>0$ and a function $\g$ such that for every $\beta>\beta_{0}$ and any
two interfaces $\cI$ and $\cI'$, 
\begin{align*}
\frac{\mu_n^\mp(\cI)}{\mu_n^\mp(\cI')}= & \exp\bigg(-\beta \fm(\cI; \cI') +\Big(\sum_{f\in\cI}\g(f,\cI)-\sum_{f'\in\cI'}\g(f',\cI')\Big)\bigg)\,,
\end{align*}
and $\g$ satisfies the following for some $\bar{c},\bar{K}>0$
independent of $\beta$: for all $\cI, \cI'$ and $f\in \cI$ and $f'\in\cI'$,
\begin{align}
|\g(f,\cI)| & \leq\bar{K} \label{eq:g-uniform-bound} \\
|\g(f,\cI)-\g(f',\cI')| & \leq \bar K e^{-\bar{c}\br(f,\cI;f',\cI')} \label{eq:g-exponential-decay}
\end{align}
where $\br(f,\cI;f',\cI')$ is the largest radius around
the origin on which $\cI-f$ ($\cI$ shifted by the midpoint of the face $f$) is \emph{congruent} to $\cI'-f'$. That is to say,
\begin{align*}
\br(f,\cI; f', \cI') := \sup \{r: (\cI - f) \cap B_{r}(0) \equiv (\cI' - f') \cap B_r(0)\}\,,
\end{align*}
where the congruence relation $\equiv$ is equality as subsets of $\R^3$.  
\end{theorem}

We will say that the radius $\br(f,\cI; f',\cI')$ is \emph{attained by} a face $ g \in \cI$ (resp., $g' \in \cI'$) of minimal distance to $f$ (resp., $f'$) whose presence prevents $\br(f, \cI; f',\cI')$ from being any larger.

\subsection{Notational comments}
We end this section with some comments on the notation we use. The key object of study in this paper is the restricted interface $\cI\restriction_{S_n}$, whose height shift by $\hgt(\cC_{\bW_n})$, we recall gives the interface $\cI$ inside the cylinder $S_n\times \mathbb Z$. We will frequently use a subscript of $S$ to denote that the relevant object is defined with respect to the interface $\cI\restriction_{S_n}$ as opposed to $S_n$, then shifted by $\hgt(\cC_\bW)$. For example, for a given $\cI$, the nested sequence of walls $\fW_{x,S_n}$, the pillar $\cP_{x,S_n}$, and its base $\sB_{x,S_n}$ and spine $\cS_{x,S_n}$ are all the corresponding quantities in the interface $\cI_{S_n}$, veritcally shifted by $\hgt(\cC_\bW)$. 

All statements in the paper will hold for $\beta$ sufficiently large, which we indicate by $\beta>\beta_0$ for a universal $\beta_0$. These statements concern the asymptotic regime of $n$ large, and are to be interpreted as holding uniformly over all $n$ large enough; when contextually understood we may therefore drop the $n$ dependence from certain notation---e.g., $S = S_n, \bW = \bW_n, x = x_n$. For ease of
presentation, when we divide regions into sub-regions of some fixed size and we do not differentiate
between the remainder sub-regions if there are divisibility problems.
All such rounding issues and integer effects can be handled via the obvious modifications. 

Finally, throughout the paper we let $\bar c,\bar K$ be the constants of~\eqref{eq:g-uniform-bound}--\eqref{eq:g-exponential-decay}, while using $C, C'$ to denote the existence of some universal constant (independent of $\beta, n$), allowing these to change from line to line.

\section{Wall clusters and conditional rigidity}\label{sec:wall-clusters}
Our goal in this section is to prove Theorem~\ref{thm:intro-rigidity-inside-ceiling}.
The rigidity proof of~\cite{Dobrushin72a} relied crucially on a grouping of nearby walls into \emph{groups of walls}, and then proving an exponential tail on the group-of-walls through an $x\in \cL_{0,n}$. However, under $\bI_{\bW}$, the group-of-walls through $x$ may include the walls in $\bW$, precluding proving a bound on the wall through $x$ conditionally on $\bI_\bW$ using the groups of walls decomposition.

Recall the definition of hulls of ceilings and walls, denoted $\hull{\cC}, \hull{W}$ from Definition~\ref{def:hull}. We replace the groups of walls formalism with a one-sided clustering of close walls, to obtain the following conditional exponential tail on the walls of $\cI\in \bI_\bW$, analogous to~\cite[Theorem 1]{Dobrushin72a} in the unconditional setting. 

\begin{theorem}\label{thm:rigidity-inside-wall}
There exists $C>0$ such that for every $\beta>\beta_0$ the following holds. Fix any two admissible collections of walls $(W_1, \ldots, W_r)$, and $\bW = (W_z)_{z\in A}$ such that $\rho(\bW) \cap  \rho(\bigcup_{i\le r} \hull{W}_i)= \emptyset$. Then 
\[
\mu_n^\mp(W_1,\ldots,W_r \mid \bI_\bW) \leq \exp\Big(-(\beta-C)\sum_{i\le r}\fm(W_i)\Big)\,,
\]
(where the event on the left-hand side indicates that $W_1,\ldots,W_r$ are walls of $\cI$). 
\end{theorem}

A consequence of Theorem~\ref{thm:rigidity-inside-wall} is that the height fluctuation of the interface at $x$ about the height of $x$ in $\cI_{\Theta_{\textsc{st}}\bW}$ has an exponential tail. In particular, when $\bW$ generates a ceiling $\cC_\bW$ nesting $x$, the height oscillations of $x$ about $\hgt(\cC_\bW)$ are tight with exponential tails. The following corollary (adapting~\cite[Eq.~(2.4)]{GL19a} to this conditional setting) on exponential tails for nested sequences of walls implies Theorem~\ref{thm:intro-rigidity-inside-ceiling}, by Observation~\ref{obs:pillar-nested-sequence-of-walls}. 

\begin{corollary}\label{cor:nested-sequence-inside-a-ceiling}
There exists $C>0$ such that for every $\beta>\beta_0$ the following holds.  Fix a set $S_n\subset \cL_{0,n}$ such that $\cL_{0}\setminus S_n$ is connected, and an admissible collection of walls $\bW_n = (W_z)_{z\notin S_n}$ such that $\rho(\bW_n) \subset S_n^c$. 
For every $x\in S_n$ and every $r\ge 1$,
\begin{align*}
    \mu_n^\mp\left( \fm(\fW_{x,S_n}) \ge r \mid \bI_{\bW_n}
    \right) \leq \exp ( -  (\beta - C)r)\,.
\end{align*}
\end{corollary}

Finally, we can use the above to deduce upper and lower exponential tail bounds on the probability of a pillar attaining a height $h$: this is the conditional analogue of~\cite[Theorem 2.26 and Proposition 2.29]{GL19a}.

\begin{corollary}\label{cor:pillar-inside-a-ceiling}
There exists $C>0$ such that for every $\beta>\beta_0$ the following holds.  Fix a set $S_n\subset \cL_{0,n}$ such that $\cL_0\setminus S_n$ is connected, and an admissible collection of walls $\bW_n = (W_z)_{z\notin S_n}$ such that $\rho(\bW_n) \subset S_n^c$. For every $x\in S_n$ and every $h\ge 1$, 
\begin{align*}
    \exp ( - 4(\beta + C)h) \le \mu_n^{\mp}\big(\hgt(\cP_{x,S_n})\ge h \mid \bI_{\bW_n}\big) \le \exp( - 4(\beta - C)h)\,.
\end{align*}
\end{corollary}

\subsection{Ceiling and wall clusters}
In~\cite{Dobrushin72a} a notion of \emph{closeness} of walls was introduced whereby two walls $W_1,W_2$ are close if there exist two edges/faces $u_1\in \rho(W_1)$ and $u_2\in \rho(W_2)$ such that the number of faces of $\cI$ projecting onto $u_i$, denoted $N_\rho(u_i)$, satisfy $d(u_1,u_2) \le \sqrt{N_\rho(u_1)}+ \sqrt{N_\rho(u_2)}$. 
This grouped walls together if they interact too strongly through the bubbles in the spin configuration (the function $\g$ in Theorem~\ref{thm:cluster-expansion}).
To control the wall through a point $x\in \cL_{0,n}$ conditionally on some large nesting wall in $\bW$, we must decouple the wall $W_x$ from $\bW$. This entails replacing the \emph{group-of-wall} clustering by a new grouping scheme, which we call \emph{wall clusters}. The key change is that our notion of closeness will be one-sided, so that the wall cluster of $W_x$ only consists of walls interior to $W_x$, enabling the conditioning on arbitrary external walls. 
\begin{definition}\label{def:closely-nested}
We say a wall $W$ is \emph{closely nested} in a ceiling $\cC$ if $W$ is nested in $\cC$  ($W\Subset \cC$) and 
\[
d_\rho(\partial \hull{\cC}, W) := d(\rho(\partial \hull{\cC}),\rho(W)) \le \fm(W)\,.
\]
(where $\rho(\partial \hull \cC) = \partial \rho(\hull\cC)$ is the edge boundary of $\rho(\hull \cC)$ in $\cL_{0}$). 
We say a wall $W_1$ is closely nested in $W_2$, if $W_1$ is closely nested in an interior ceiling of $W_2$. 
\end{definition}

\begin{remark}
While we choose to use the above definition, where the horizontal distance is compared to the excess area, there are many one-sided choices here that would work. For instance, a definition that more closely resembles the original definition of~\cite{Dobrushin72a} is as follows: $W_1$ is closely nested in $W_2$ if $W_1 \Subset W_2$ and there exists $u_1 \in \rho(W_1)$ such that $d_\rho(W_2,u_1)\leq \sqrt {N_\rho(u_1)}$. We choose to stick with Definition~\ref{def:closely-nested}.
\end{remark}

\begin{definition}\label{def:ceiling-cluster}
For a ceiling $\cC$, define the \emph{ceiling cluster} $\mathsf{Clust}(\cC)$ as follows (see Fig.~\ref{fig:wall-cluster} for a depiction):
\begin{enumerate}
    \item Initialize $\mathsf{Clust}(\cC)$ with all walls $W$ that are closely nested in $\cC$.
    \item Iteratively, add to $\mathsf{Clust}(\cC)$ all walls $W$ that are closely nested in some wall $W'\in \mathsf{Clust}(\cC)$. 
\end{enumerate}
\end{definition}

\begin{definition}\label{def:wall-cluster}
For a set of walls $\bV$ in $\cI$ with interior ceilings $\cC_1,\ldots,\cC_r$, define its \emph{wall cluster} to be
$$\mathsf{Clust}(\bV):= \bV\cup \bigcup_{i=1}^r \mathsf{Clust}(\cC_i)\,.$$
\end{definition}

\begin{remark}
As discussed above, the most important benefit of wall clusters is that unlike groups of walls, the wall cluster of $W$ is completely interior to $W$, i.e., $\rho(\Clust(\bV))\subset \rho(\hull{\bV})$.
Wall clusters also have the nice property that if $W_1 \Subset W_2 \Subset W_3$ and $W_1$ is closely nested in $W_3$, then $W_1$ is closely nested in $W_2$.
\end{remark}

\begin{figure}
    \centering
     \includegraphics[width=0.6\textwidth]{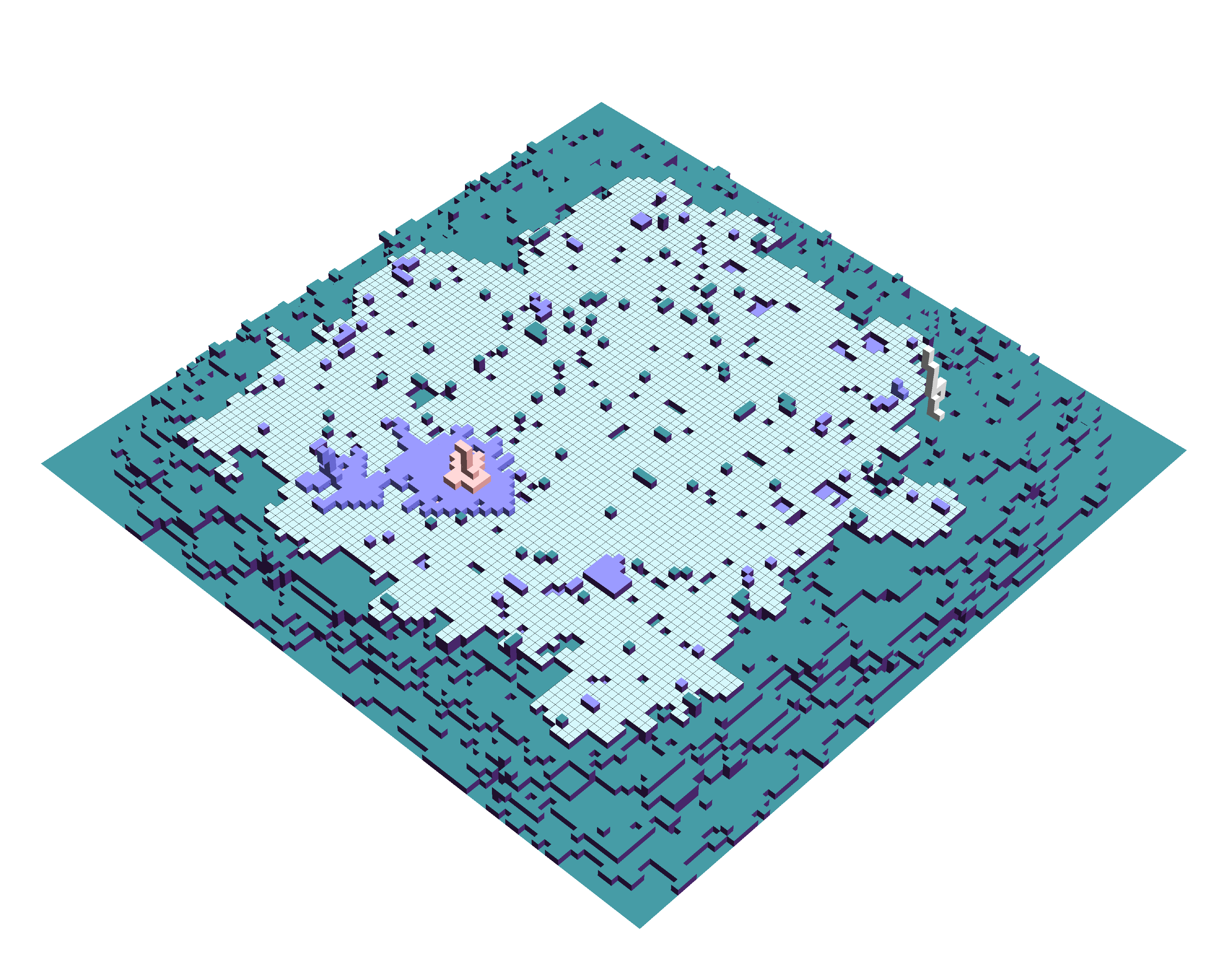}
    \caption{The ceiling cluster of the identified ceiling $\cC$ (highlighted in light teal) is indicated  by purple and orange---blue marks those walls that are closely nested in $\cC$, and orange marks those that are in turn closely nested in some blue wall (but not themselves closely nested in $\cC$). In white is a large wall near, but outside, $\cC$, and therefore not in the ceiling cluster of $\cC$.}
    \label{fig:wall-cluster}
\end{figure}

\subsection{Rigidity conditionally on nesting walls}
The goal of this section is to prove Theorem~\ref{thm:rigidity-inside-wall}.
Throughout this section, fix $\bV = (W_1,\ldots,W_r)$, and fix $\bW$ having $\rho(\bW) \cap \rho(\bigcup_{i\le r} \hull{W}_i)= \emptyset$. 

\begin{definition}
Let $\Phi_{\bV}$ be the following map on interfaces $\cI \in \bI_{\bW}\cap \bI_{\bV}$:  
\begin{enumerate}
    \item Remove from the standard wall representation of $\cI$ all standard walls in $\Theta_{\textsc{st}}\Clust(\bV)$. 
    \item From the remaining standard walls, generate the interface $\Phi_{\bV}(\cI)$ via Lemma~\ref{lem:interface-reconstruction}.
\end{enumerate}
\end{definition}
Recall that by Definition~\ref{def:wall-cluster}, $\rho(\Clust(\bV))\subset \rho(\bigcup_{i\le r} \hull{W}_i)$; as such, the standard wall representation of $\Phi_{\bV}(\cI)$ has the same wall collection indexed by faces outside $\rho(\bigcup_{i\le r} \hull {W}_i)$ as $\cI$ does, and, $\Phi_{\bV}(\cI)\in \bI_{\bW}$. This property enables us to make the exponential tail probability of $\bV$ conditional on $\bI_{\bW}$. 

We first analyze the weight gain under the map $\Phi_{\bV}$. 

\begin{lemma}\label{lem:wall-cluster-weights}
There exists $C>0$ such  that for every $\beta>\beta_0$, and for every $\cI \in \bI_{\bW}\cap \bI_{\bV}$, 
\begin{align*}
   \Big|\log\frac{\mu_n^\mp (\cI)}{\mu_n^\mp(\Phi_{\bV}(\cI))} + \beta\fm(\Clust(\bV))\Big| \le C \fm(\Clust(\bV))\,.
\end{align*}
\end{lemma}

\begin{proof}
For ease of notation, let $\cJ = \Phi_{\bV}(\cI)$.  By definition, $\fm(\cI;\cJ)=\fm(\Clust(\bV))$. By Theorem~\ref{thm:cluster-expansion}, 
\begin{align*}
    \Big|\log\frac{\mu_n^\mp (\cI)}{\mu_n^\mp(\cJ)} + \beta\fm(\Clust(\bV))\Big| \leq \Big|\sum_{f\in \cI} \g(f,\cI)- \sum_{f'\in \cJ} \g(f',\cJ)\Big|\,.
\end{align*}
Now we split up the set of faces in $\cI$ and $\cJ$ as follows. The interface $\cI_{\Theta_{\textsc{st}}\Clust(\bV)}$ has  (interior) ceilings which we can enumerate by $\cC_1,\ldots,\cC_s$. 
Now partition the faces of $\cI$ as follows.  
\begin{itemize}
    \item $\bF_{\ext}= \{f\in \cI: \rho(f)\notin \bigcup_{i\le r} \rho(\hull{W}_i)\}$ is the set of all faces projecting outside the hulls of $(W_i)_{i\le r}$.   
    \item $\Clust(\bV)$ is the set of wall faces removed by $\Phi_{\bV}$.
    \item $(\bB_{i})_{i\le s}$: collects all other faces of $\cI$, indexed by their innermost nesting ceiling among $(\cC_i)_{i \le s}$. 
\end{itemize}
With the above splitting in hand, we now decompose $\cJ$ as follows. 
\begin{itemize}
    \item $\bF_{\ext}$ is as defined above.
    \item $\bH$ is the set of faces in $\cJ$ whose projection is in $\sF(\rho(\Clust{(\bV)}))$. 
    \item $(\theta_{\udarrow} \bB_i)_{i \le s}$ collects all other faces in $\cJ$.
\end{itemize}
Observation~\ref{obs:1-to-1-map-faces} describes a 1-1 correspondence between $\bB_i$ and $\theta_{\udarrow} \bB_i$ given by the vertical shifts induced by the ceilings of deleted walls of $\Theta_{\textsc{st}}\Clust{(\bV)}$ from the standard wall representation of $\cI$: we encode this 1-1 correspondence into $f\mapsto \theta_{\udarrow} f$. Evidently, for each $i$, all faces in $\bB_i$ undergo the same vertical shift.  

We can then decompose 
\begin{align}\label{eq:wall-cluster-exp-tail-splitting}
    \Big|\sum_{f\in \cI}\g(f,\cI) - \sum_{f'\in \cJ} \g(f',\cJ)\Big|& \leq \sum_{f\in \Clust{(\bV)}} |\g(f,\cI)|+\sum_{f'\in \bH} |\g(f',\cJ)|+\sum_{f\in \bF_{\ext}} |\g(f,\cI) - \g(f,\cJ)| \nonumber \\ 
    & \qquad + \sum_{i=1}^{s} \sum_{f\in \bB_i} |\g(f,\cI) - \g(\theta_\udarrow f,\cJ)|\,.
\end{align}

We control the above sum using two claims. The first controls the interactions of faces in $\Clust{(\bV)}\cup \bH$.

\begin{claim}\label{clm:wall-cluster-W-to-else}
There exists a universal $\bar C$ such that  
\begin{align*}
    \sum_{f\in \Clust{(\bV)} \cup \bH}\sum_{g\in \sF(\mathbb Z^3)}  e^{- \bar c d(f,g)}\leq \bar C\fm(\cI;\cJ)\,.
\end{align*}
\end{claim}
\begin{proof}
The claim follows immediately by integrating the exponential tails, to bound the left-hand side by $C  |\Clust{(\bV)} \cup \bH|$ for some $C$, which by~\eqref{eq:excess-area-wall-relations} is at most $4C \fm(\Clust{(\bV)})= 4 C \fm(\cI;\cJ)$.
\end{proof}

The next claim controls interactions between walls nested in some $\cC_i$, that were not deleted (i.e., faces in $\bB_i$), with those outside $\cC_i$ (as these may have undergone different vertical shifts). 
\begin{claim}\label{clm:wall-cluster-ceiling-to-else}
There exists a universal $\bar C$ such that for every $i\leq r$,  
\begin{align*}
    \sum_{f\in \bB_i \cup \theta_\udarrow \bB_i} \sum_{g\in \sF(\mathbb Z^3): \rho(g)\notin \rho(\hull {\cC}_i)}  e^{ - \bar c d(f,g)}   \leq \bar C | \rho(\partial \hull{\cC}_i)|\,.
\end{align*}
\end{claim}
\begin{proof}
Clearly, by vertical translation invariance of the index set of the latter sum, it suffices to prove this for a sum over $f\in \bB_i$, say.  First, we write for some universal $C$, 
\begin{align*}
\sum_{f\in \bB_i} \sum_{g\in \sF(\mathbb Z^3): \rho(g)\notin \rho(\hull\cC_i)} e^{- \bar c d(f,g)} & \leq \sum_{f\in \bB_i} \sum_{u\in \rho(\partial \hull \cC_i)} C e^{ - \bar c d(\rho(f), u)} \\
& \leq \sum_{f\in \rho(\hull{\cC}_i)} \sum_{u\in \rho(\partial \hull{\cC}_i)} C e^{ - \bar c d(f,u)}  +  \sum_{u\in \rho(\partial \hull \cC_i)} \sum_{W: W\subset \bB_i} C |W| e^{ - \bar cd(\rho(W),u)}\,,
\end{align*} 
where the first term accounts for the ceiling faces of $\bB_i$ and the latter sum runs over all walls of $\cI$ that are a subset of $\bB_i$. The first term above is clearly at most $C|\partial \hull{\cC}_i|$ for some other $C$. 
Because $W$ is nested in a wall of $\Clust{(\bV)}$ while not being in $\Clust{(\bV)}$, $W$ is not in the ceiling cluster of $\cC_i$ and therefore $d(\rho(W),\rho(\partial \hull \cC_i))>\fm(W)$. Thus, using~\eqref{eq:excess-area-wall-relations}, the second term above is at most 
\begin{align*}
    \sum_{u\in \rho(\partial \hull \cC_i)} \sum_{W\subset \bB_i} 2C \fm(W) e^{-\frac 12 \bar c [d(\rho(W),u)+\fm(W)]} \leq \sum_{u\in \rho(\partial \hull \cC_i)} \sum_{W\subset \bB_i} C' e^{ - \frac 12\bar c d(\rho(W),u)}\,.
\end{align*}
Since disjoint walls have disjoint projections, summing out the inner sum, we see that this term is altogether at most $C |\rho (\partial \hull \cC_i)|$ for some other $C$. Combining the two bounds yields the desired. 
\end{proof}

With the above claims in hand, we proceed with bounding each of the terms in~\eqref{eq:wall-cluster-exp-tail-splitting}. The first and second terms in~\eqref{eq:wall-cluster-exp-tail-splitting} are easily seen to each be bounded by $\bar C\bar K \fm(\cI;\cJ)$ by~\eqref{eq:g-uniform-bound} and Claim~\ref{clm:wall-cluster-W-to-else}. 

The third term of~\eqref{eq:wall-cluster-exp-tail-splitting} is controlled as follows. For every face $f\in \bF_{\ext}$, the radius $\br(f,\cI;f,\cJ)$ must be attained by a wall face, either in $\Clust{(\bV)}$ or nested in some wall of $\Clust{(\bV)}$. As such, by~\eqref{eq:g-exponential-decay}, we can bound the third term of~\eqref{eq:wall-cluster-exp-tail-splitting} by
\begin{align*}
   \sum_{f\in \bF_{\ext}} \bar K e^{ - \bar c \br(f,\cI;f,\cJ)} \leq  \sum_{f\in \bF_{\ext}} \sum_{g\in \Clust{(\bV)}} \bar K e^{ - \bar c d(f,g)}+ \sum_{f\in \bF_{\ext}}\sum_i \sum_{g\in \bB_i} \bar K e^{ - \bar c d(f,g)}\,.
\end{align*}
The first sum is clearly at most $\bar C \bar K \fm(\cI;\cJ)$ by Claim~\ref{clm:wall-cluster-W-to-else}, while the second is clearly at most $\bar C\bar K \sum_i |\rho(\partial \hull {\cC}_i)|$ by Claim~\ref{clm:wall-cluster-ceiling-to-else}. This is then easily seen to be at most $2 \bar C \bar K |\Clust{(\bV)}| \leq 4 \bar C \bar K \fm(\cI;\cJ)$. 

The fourth term of~\eqref{eq:wall-cluster-exp-tail-splitting} is bounded as follows. Fix any $i$. For every $f\in \bB_i$, the radius $\br(f,\cI; \theta_\udarrow f,\cJ)$ is attained either by a face  projecting into $\rho(\hull \cC_i)^c$, a wall face of $\Clust{(\bV)}$ whose innermost nesting ceiling is $\cC_i$, or a wall face of $\bB_j$ for some $j$ satisfying $\rho(\cC_i) \subset \rho(\hull \cC_j)^c$ (i.e., $\cC_j$ nested in $\cC_i$). Thus, we can write 
\begin{align*}
    \sum_{i=1}^s \sum_{f\in \bB_i} |\g(f,\cI)- \g(\theta_{\udarrow}f,\cJ)| & \leq \sum_i \sum_{f\in \bB_i} \sum_{g\in \sF(\mathbb Z^3): \rho(g)\notin \rho(\hull \cC_i)} \bar K e^{ - \bar c d(f,g)} + \sum_i \sum_{f\in \bB_i} \sum_{g\in \Clust{(\bV)}} \bar K e^{ - \bar c d(f,g)} \\ 
    & \quad + \sum_{i} \sum_{f\in \bB_i} \sum_{j:\rho(\cC_i) \subset \rho(\hull \cC_j)^c} \sum_{g\in \bB_j} \bar K e^{-\bar c d(f,g)}\,.
\end{align*}

In the first case, the sum is bounded, via Claim~\ref{clm:wall-cluster-ceiling-to-else} by $ \sum_i \bar C \bar K |\rho (\partial \hull \cC_i)|$; in the second case, the sum is bounded, via Claim~\ref{clm:wall-cluster-W-to-else}, by $\bar C \bar K \fm(\cI;\cJ)$. The last sum can be rewritten as 
\[
\sum_{j} \sum_{g\in \bB_j} \sum_{i:\rho(\cC_i) \subset \rho(\hull \cC_j)^c} \sum_{f\in \bB_i} \bar K e^{- \bar c d(f,g)} \leq \sum_j \sum_{g\in \bB_j} \sum_{f\in \sF(\mathbb Z^3): \rho(f)\notin \rho(\hull \cC_i)} \bar K e^{ - \bar cd(f,g)} \leq \sum_j \bar C \bar K |\rho (\partial \hull \cC_j)|\,,
\]
by Claim~\ref{clm:wall-cluster-ceiling-to-else}. Combining these three and again using the fact that $\sum_i |\partial \rho(\hull \cC_i)|\leq 2\fm(\cI;\cJ)$, we conclude that the right hand side of~\eqref{eq:wall-cluster-exp-tail-splitting} is at most $C \fm(\cI;\cJ)$ for some universal $C$ as desired. 
\end{proof}

It now remains to analyze the multiplicity of the map $\Phi_{\bV}$ on the set of interfaces $\cI\in \bI_{\bW}\cap \bI_\bV$.

\begin{lemma}\label{lem:wall-cluster-multiplicity}
There exists a universal constant $C$ such that for every $M\ge 1$, for every $\cJ\in \Phi_{(W_i)}(\bI_{\bW}\cap \bI_\bV)$
\[
\big|\{\cI \in \Phi_{\bV}^{-1}(\cJ)\,:\, \fm(\cI;\cJ) = M\}\big|\leq C^M\,.
\]
\end{lemma}
\begin{proof}
Observe that given $\cJ$, the wall collection $\Clust({\bV})$ serves as a \emph{witness} to $\cI$, so that it suffices, for fixed $\bV$, to bound the number of possible $\Clust{(\bV)}$ compatible with $\bV$ and the standard wall representation of $\cJ$. For ease of notation, let $\cC_{1},\ldots,\cC_s$ be the ceilings of $\cI_{\Clust{(\bV)}}$ and partition the walls of $\Clust{(\bV)}$ into $\bV$ and sets $(\bF_i)_{i\le s}$, partitioning the walls of $\Clust{(\bV)}\setminus \bV$ by their innermost nesting ceiling among $(\cC_i)_{i\le s}$. 

We enumerate over all such choices by identifying each choice of $\Clust{(\bV)}$ with the following witness, consisting of $r$ $*$-connected subsets of $\sF(\Z^3)$ (one for each of $W_1,\ldots,W_r$), with each face decorated by a  color among $\{\red,\blue\}$. We construct the witness as follows: 
\begin{enumerate}
    \item Include all faces in $\theta_{\textsc{st}} \Clust{(\bV)}$ and color them $\blue$,
    \item For each wall $W\in \bF_j$, add the shortest path of faces in $\cL_{0,n}$ between $W$ and $ \rho(\partial \hull{\cC}_j)$ in $\red$.
\end{enumerate}
For every $\Clust{(\bV)}$ having $\fm(\Clust{(\bV)})=M$, by Definitions~\ref{def:ceiling-cluster}--\ref{def:wall-cluster}, the number of faces of the witness is  
$$|\Clust{(\bV)}|+ \sum_{j=1}^s \sum_{W \in \bF_j} d_\rho(W,\rho( \partial  \hull{\cC}_j)) \le |\Clust{(\bV)}|+ \sum_{j=1}^s \sum_{W \in \bF_j} \fm(W)  \leq |\Clust{(\bV)}|+ |\Clust{(\bV)} \setminus \bV|\leq 4M\,.$$ 
The number of possible pre-images of $\cJ$ under $\Phi_{\bV}$ having excess area $M$ is at most the number of witnesses arising from those pre-images; this is in turn at most the number of possible $\{\red,\blue\}$ decorated face-sets of $\sF(\mathbb Z^3)$ having at most $4M$ faces in total, consisting of $r$ $*$-connected face subsets containing $\Theta_{\textsc{st}}W_1,\ldots,\Theta_{\textsc{st}}W_r$ respectively. By first partitioning $CM$ into the number of faces that get allocated to each of the connected components of $W_1,\ldots, W_r$, then enumerating over decorated face sets of that size for each of these connected components, by Fact~\ref{fact:number-of-walls}, we find that this is at most $C^M$ for some universal $C>0$.
\end{proof}

\begin{proof}[\textbf{\emph{Proof of Theorem~\ref{thm:rigidity-inside-wall}}}]
It remains to combine Lemmas~\ref{lem:wall-cluster-weights} and \ref{lem:wall-cluster-multiplicity}. Observe first of all that for every $\cI \in \bI_{\bW}\cap \bI_{\bV}$, we have $\fm(\cI;\cJ) \ge \sum_{i\le r}\fm(W_i) = \fm(\bV)$. We can express the probability $\mu_n^{\mp}(\bI_{\bW},\bI_{\bV})$ as 
\begin{align*}
    \sum_{\cI \in \bI_{\bW}\cap \bI_\bV} \mu_n^\mp(\cI) &  = \sum_{\cJ \in \Phi_{\bV}(\bI_{\bW}\cap \bI_\bV)} \mu_n^{\mp}(\cJ)  \sum_{k\ge \fm(\bV)} \sum_{\cI \in \Phi_{\bV}^{-1}(\cJ): \fm(\cI;\cJ) = k} \frac{\mu_n^{\mp}(\cI)}{\mu_n^{\mp}(\cJ)} \\ 
    & \le \sum_{\cJ \in \Phi_{\bV}(\bI_{\bW}\cap \bI_{\bV})} \mu_n^{\mp}(\cJ) \sum_{k\ge \fm(\bV)} C^{k} e^{ - (\beta - C)k} \\ 
    & \le C' \mu_{n}^{\mp} (\bI_{\bW}) \exp\big( - (\beta - C')\fm(\bV)\big)\,,
\end{align*}
where in the second line we used Lemma~\ref{lem:wall-cluster-weights} and Lemma~\ref{lem:wall-cluster-multiplicity}, and in the third line, we used the fact that $\Phi_{\bV}(\bI_{\bW}\cap \bI_{\bV}) \subset \bI_\bW$. Dividing both sides above by $\mu_n^{\mp}(\bI_\bW)$ then concludes the proof.
\end{proof}

\subsection{Consequences of Theorem~\ref{thm:rigidity-inside-wall}} In this section, we prove Corollaries~\ref{cor:nested-sequence-inside-a-ceiling}--\ref{cor:pillar-inside-a-ceiling}, translating the exponential tail of Theorem~\ref{thm:rigidity-inside-wall} to results on the height profile of the interface and its pillars. For ease of notation, we drop the subscripts $n$ from $S_n, \bW_n$. 

\begin{proof}[\textbf{\emph{Proof of Corollary~\ref{cor:nested-sequence-inside-a-ceiling}}}]
Begin by expressing
\begin{align*}
    \mu_{n}^{\mp}(\fm(\fW_{x,S})\ge r & \mid \bI_{\bW})  \leq \sum_{k \geq r} \sum_{\fW_{x,S}: \fm(\fW_{x,S}) = k} \mu_n^{\mp} (\fW_{x,S} \mid \bI_\bW)
\end{align*}
Observe that for any nested sequence of walls $\fW_{x,S} = (W_1,\ldots,W_r)$ compatible with $\bW= (W_z)_{z\notin S}$, since $\cL_0\setminus S$ is connected, $\rho(\bW)\cap\rho(\bigcup_{i\le r}\hull{W}_i)=\emptyset$. 
Thus we can apply Theorem~\ref{thm:rigidity-inside-wall} to find that this is at most 
\begin{align*}
    \sum_{k \geq r} \sum_{\fW_{x,S}: \fm(\fW_{x,S}) = k} C \exp ( - (\beta - C) k)\,.
\end{align*}
It suffices for us to understand how many choices there are for $\fW_{x,S}$ compatible with $\bW$ and having $\fm(\fW_{x,S}) = k$. Evidently, this is in turn bounded by the number of nested sequences of standard walls $\fW_x$ (with no constraint on admissibility with $\bW$), which we recall from Fact~\ref{fact:number-of-walls}, is at most $s^k$ for some universal constant $s>0$. Summing out the exponential tail above, we obtain the desired bound. 
\end{proof}

\begin{proof}[\textbf{\emph{Proof of Corollary~\ref{cor:pillar-inside-a-ceiling}}}]
Let us begin with the upper bound. The proof is analogous that of Theorem 2.26 of~\cite{GL19a}, replacing groups of walls with wall clusters. In order for $\cP_{x,S}$ to have $\hgt(\cP_{x,S})\ge h$, it must be the case that there exists $y$ nested in a wall of $\fW_{x,S}$, such that $\fW_{y,S}\cup \fW_{x,S}$ have excess area at least $4h$. 

In particular, we can union bound over the maximum height attained by $\fW_{x,S}$ beyond $\hgt(\cC_\bW)$ and note that if that height is $0<h_1<h$, then there must exist $y$ nested in $\fW_{x,S}$ such that $\fW_{y,S}$ attains a further height $h - h_1$. Namely, we can bound the probability $\mu_n^{\mp}(\hgt(\cP_{x,S})\ge h \mid \bI_{\bW})$ by 
\begin{align*}
    \sum_{r\ge 0} \sum_{h_1 \leq h} \mu_n^{\mp} (\fm(\fW_{x,S}) >r+4h_1 \mid \bI_{\bW}) \sum_{\substack{y\in \cL_0 \\ d(x,y)\leq r}} \sup_{\substack{\fW_{x,S} \\ y\in \hull{\fW}_{x,S}}} \mu_n^{\mp}(\fm(\fW_{y,S} \setminus\fW_{x,S})> 4(h-h_1) \mid \fW_{x,S},\bI_\bW)\,. 
\end{align*}
(Here we used that the cumulative excess area of $\fW_{x,S}\cup \fW_{y,S}$ must exceed $4h$ if $\hgt(\cP_{x,S})\ge h$.)
By Corollary~\ref{cor:nested-sequence-inside-a-ceiling} applied to the first term, this is in turn at most 
\begin{align*}
        \sum_{r\ge 0} \sum_{h_1 \leq h} e^{ - (\beta - C)(r+ 4h_1)} r^2 \sup_{\substack{\fW_{x,S} \\ y\in \rho(\hull{\fW}_{x,S})}} \mu_n^{\mp}(\fm(\fW_{y,S} \setminus\fW_{x,S})> 4(h-h_1) \mid \bI_{\fW_{x,S}}\cap \bI_\bW)\,. 
\end{align*}
To  apply Corollary~\ref{cor:nested-sequence-inside-a-ceiling} to the probability above, we set $S'$ to be the hull of the ceiling of $\fW_{x,S}$ nesting $y$ (so that $\cL_0 \setminus S'$ is evidently connected) and condition further on all walls of $(S')^c$ admissible with $\fW_{x,S},\bW$. Since Corollary~\ref{cor:nested-sequence-inside-a-ceiling} applies uniformly over such conditioning (the walls we condition on will necessarily have projections contained in  $\rho(\fW_{y,S}\setminus \fW_{x,S})^c$), we find that the above is at most 
\begin{align*}
    \sum_{r\ge 0}  e^{ - (\beta- C)r} \sum_{h_1\le h} r^2 e^{ - 4(\beta - C)h_1}  e^{ - 4(\beta - C)(h-h_1)} \le C' h e^{ - 4(\beta - C)h} \le C' e^{ - 4(\beta - C')h}\,,
\end{align*}
as claimed.

We now turn to the lower bound. Let $W_{x,\parallel}^h$ be a standard wall consisting of the bounding vertical faces of a column of $h$ vertically consecutive sites above $x$ (i.e., centered at $x + (0,0, \frac i2)_{i  \le h}$). The following claim which we isolate, lower bounds the ratio of  $\mu_n^{\mp}(\hgt(\cP_{x,S})\ge h, \bI_\bW)$ to $\mu_n^{\mp}(\fm(\mathfrak W_{x,S}) = 0, \bI_{\bW})$; this latter quantity is at least $(1-\epsilon_\beta)\mu_n^{\mp}(\bI_\bW)$ by Corollary~\ref{cor:nested-sequence-inside-a-ceiling}. 
To avoid boundary-case issues, let $\bar x = \{f: f\sim^* x\} \cup \{x\}$, and let $\mathfrak W_{\bar x,S}= \bigcup_{f\in \bar x}\mathfrak W_{f,S}$. 

\begin{claim}\label{clm:lower-bound-forcing}
For every $S\subset \cL_{0,n}$ such that $\cL_0 \setminus S$ is connected, every $x\in S_n$ and every $\bW= \{W_z: z\notin S_n\}$, we have for every $h\ge 1$, 
\begin{align*}
    \frac{\mu_n^\mp(\{\Theta_{\textsc{st}}\fW_{x,S} = W_{x,\parallel}^h\},\bI_{\bW})}{\mu_n^\mp(\{\fW_{\bar x,S}= \trivincr\}, \bI_{\bW})} \ge \exp ( - 4(\beta + C)h)\,.
\end{align*}
\end{claim}
\noindent Let us first complete the proof of Corollary~\ref{cor:pillar-inside-a-ceiling}, before proving Claim~\ref{clm:lower-bound-forcing}. By inclusion, and Claim~\ref{clm:lower-bound-forcing},
\begin{align*}
    \mu_n^\mp(\hgt(\cP_{x,S})\ge h, \bI_\bW) & \ge \mu_n^\mp(\{\Theta_{\textsc{st}}\fW_{x,S} = W_{x,\parallel}^h\},\bI_{\bW}) \\
    &  \ge e^{ - 4(\beta + C)h} \mu_n^\mp(\{\fW_{\bar x,S}= \trivincr\}\cup \{\Theta_{\textsc{st}}\fW_{\bar x,S} = W_{x,\parallel}^h\}, \bI_{\bW})
\end{align*}
which by Corollary~\ref{cor:nested-sequence-inside-a-ceiling} with $r=1$, and a union bound over the faces in $\bar x$, is at least $(1-\epsilon_\beta) e^{- 4(\beta + C)h} \mu_n^\mp(\bI_{\bW})$
for some $\epsilon_\beta \downarrow 0$ as $\beta \uparrow \infty$. 
Dividing out both sides by $\mu_n^\mp(\bI_\bW)$ then implies the desired lower bound.  
\end{proof}

\begin{proof}[\textbf{\emph{Proof of Claim~\ref{clm:lower-bound-forcing}}}]
Consider any interface $\cI$ in $\bI_{\bW}$, having $\fW_{\bar x,S} = \trivincr$, 
and define $\Psi_{x,h}(\cI)$ as the interface whose standard wall representation additionally has the standard wall $W_{x,\parallel}^h$ consisting of the bounding vertical faces of the column of $h$ sites $(x+(0,0,\frac{i}{2}))_{1\le i\le h}$. This results in an admissible collection of standard walls by the assumption that $\cI$ has $\fW_{\bar x,S} = \trivincr$, and the wall representation of $\Psi_{x,h}$ will have $\Theta_{\textsc{st}}\fW_{x,S}= W_{x,\parallel}^h$. Denoting by $A\oplus B$, the symmetric difference of the two sets, by Theorem~\ref{thm:cluster-expansion}, 
\begin{align*}
    \Big|\log\frac{\mu_n^{\mp}(\Psi_{x,h}(\cI))}{\mu_n^{\mp}(\cI)} + \beta \fm(\Psi_{x,h}(\cI);\cI)\Big| & \le \sum_{f\in \cI\oplus \Psi_{x,h}(\cI)} \bar K + \sum_{f\in \cI\cap \Psi_{x,h}(\cI)} |\g(f;\cI) - \g(f;\Psi_{x,h}(\cI))| \\ 
    & \le \bar K |\cI\oplus \Psi_{x,h}(\cI)| + \sum_{f\in \cI \cap \Psi_{x,h}(\cI)}\sum_{g\in \cI \oplus \Psi_{x,h}(\cI)} \bar K e^{ - \bar c d(f,g)} \\
    & \le C |\cI \oplus \Psi_{x,h}(\cI)|\,.
\end{align*}
Noticing that $\fm(\Psi_{x,h}(\cI);\cI) = 4h$, and $|\cI \oplus \Psi_{x,h}(\cI)| = 4h + 2$, and using the fact that $\Psi_{x,h}$ is an injection from $\{\cI\in \bI_{\bW}: \fW_{\bar x,S} = \trivincr\}$ to $\{\cI \in \bI_{\bW}: \Theta_{\textsc{st}}\fW_{x,S} = W_{x,\parallel}^h\}$, we deduce that 
\begin{align*}
    \mu_n^{\mp}(\{\Theta_{\textsc{st}}\fW_{\bar x,S} = W_{x,\parallel}^h\},\bI_{\bW}) & = \sum_{\cJ \in \bI_{\bW}: \Theta_{\textsc{st}}\fW_{\bar x,S}^{\cJ} = W_{x,\parallel}^h} \mu_n^{\mp}(\cJ) \\
    & \ge \!\! \sum_{\cJ \in \Psi_{x,h}(\{\cI\in \bI_\bW:\fW_{\bar x,S} = \trivincr\})} \!\!\! \mu_n^{\mp}(\Psi_{x,h}^{-1}(\cJ)) e^{ - (\beta + C)\fm(\cJ;\Psi_{x,h}^{-1}(\cJ))} \\
    & \ge e^{ - 4(\beta  + C)h} \mu_n^{\mp}(\{\fW_{\bar x,S} =\trivincr\},\bI_{\bW})\,,
\end{align*}
concluding the proof.
\end{proof}

\section{Tall pillars are typically isolated}\label{sec:tall-pillar-shape}
Our aim in this section is to show a shape theorem for tall pillars, uniformly over the conditioning on $\bW_n$. While one could in principle prove the more refined shape theorems of~\cite[Theorem 4]{GL19a} (e.g., proving asymptotic stationarity and mixing properties of the increment sequence), we prove an analogue of~\cite[Theorem 4]{GL19b}, which suffices for showing the tight asymptotics of $M_S$. Namely, we focus on showing that typical tall pillars have exponential tails on their base; beyond that, we add a new property of the shape showing that they are \emph{isolated} from their nearby environments with $1-\epsilon_\beta$ probability. 
This notion of \emph{isolated pillars}, which we next define, will be enough to decorrelate the pillar from its surrounding environment, enabling us to couple its law to infinite volume and obtain Theorem~\ref{thm:cond-uncond-simple} in the next section.

Recall the definition of the pillar $\cP_{x}$ with its spine $\cS_{x}$ and base $\sB_{x}$ from \S\ref{sec:pillar-def}. Label the cut-points of $\cS_{x}$ as $v_1,\ldots v_\sT,v_{\sT+1}:= v_{>\sT}$, and call its increments $\sX_{1},\ldots, \sX_{\sT+1}:= \sX_{>\sT}$. Let $\cI \setminus \cP_x$ be the \emph{truncated interface} which, informally, is the result of deleting the pillar $\cP_x$ from $\cI$: namely, it is the interface obtained from the spin configuration $\sigma'$ which is the result of starting from $\sigma(\cI)$ and flipping all the spins in $\sigma(\cP_x)$ to minus. One can similarly define a truncated interface $\cI \setminus \cS_x$.

\begin{definition}\label{def:isolated-pillars}
Let $x\in \cL_{0,n}$. We say an interface $\cI$ has \emph{$(L,h)$-isolated} pillar $\cP_x$, if it satisfies the following.
\begin{enumerate}
    \item The pillar $\cP_x$ has empty base $\sB_x = \emptyset$ (i.e., $v_1 = x+(0,0, \frac 12)$), and increment sequence satisfying
    \begin{align*}
    \fm(\sX_t) \le \begin{cases}
    0 & \mbox{ if $t\le L^3$} \\
    t & \mbox{ if $t >L^3$}
    \end{cases}\,,
    \end{align*}
    as well as spine whose face-set $\sF(\cS_x) = \bigcup_{j\ge 1} \sF(\sX_t)$ satisfies
    \begin{align*}
    |\sF(\cS_x)| \le 10h\,.
    \end{align*} \label{item:iso-pillar-increments}
    \item The walls $(\tilde W_y)_{y\in \cL_{0,n}}$ of $\cI\setminus \cP_x$ satisfy 
    \begin{align*}
\fm(\tilde W_y) \le \begin{cases}
0 &  \mbox{ if $d(y,x)\le L$} \\
\log [d(y,x)] & \mbox{ if $L < d(y,x) < L^3 h$}
\end{cases}\,.
    \end{align*}\label{item:iso-pillar-walls}
\end{enumerate}
We write $\cI \in \Iso_{x,L,h}$ to denote that $\cI$ has $(L,h)$-isolated pillar $\cP_x$.
\end{definition}

The goal of this section will be to show the following shape theorem for tall pillars (possibly inside ceilings). Recall that we set $o := (- \frac 12, -\frac 12, 0)\in \cL_{0,n}$ to be a fixed origin face.

\begin{theorem}\label{thm:shape-inside-ceiling}
There exist constants $L_\beta$ and $\epsilon_\beta$  (with $L_\beta \uparrow \infty$ and $\epsilon_\beta \downarrow 0$ as $\beta \to \infty$) such that the following holds for all $\beta>\beta_0$. For every $h = h_n \ge 1$ and $x = x_n \in \cL_0$, and any set $S_n$ such that $\cL_0\setminus S_n$ is connected, for every $\bW_n= (W_z)_{z\notin S_n}$ such that  $\rho(\bW_n) \subset S_n^c$, and $h = o(d(x_n,S_n))$, we have for all $0 \le h'\le h$, 
\begin{align}\label{eq:Ph-subset-Lambda}
    \mu_{n}^\mp \big(\cI\restriction_{S_n} \in \Iso_{x,L,h} \mid \hgt(\cP_{x,S_n})\ge h', \bI_{\bW_n}\big) \ge 1-\epsilon_\beta\,. 
\end{align}
By taking a limit as $n\uparrow \infty$ with $h$ fixed and $S_n = \cL_{0,n}$, we obtain 
\begin{equation}\label{eq:Ph-subset-Z3}
\mu_{\Z^3}^\mp(\cI \in \Iso_{o,L,h} \mid \hgt(\cP_o)\geq h') \geq 1-\epsilon_\beta \,.
\end{equation}
\end{theorem}

\begin{remark}\label{rem:exp-tails-inside-ceilings}
Theorem~\ref{thm:shape-inside-ceiling} can be modified to also show that the $t$'th increment of $\cP_{x,S}$ has uniformly exponential tails on its excess area. Thus in addition to showing that with probability $1-\epsilon_\beta$, a pillar is $(L,h)$-isolated, it serves to show exponential tails on the excess areas of its base and increments. See e.g.,~\cite[Algorithm 1 and Theorem 4.1]{GL19b} for this modification to additionally get an exponential tail on~$\fm(\sX_t)$. 
\end{remark}

\begin{figure}
    \centering
    \includegraphics[width=0.7\textwidth]{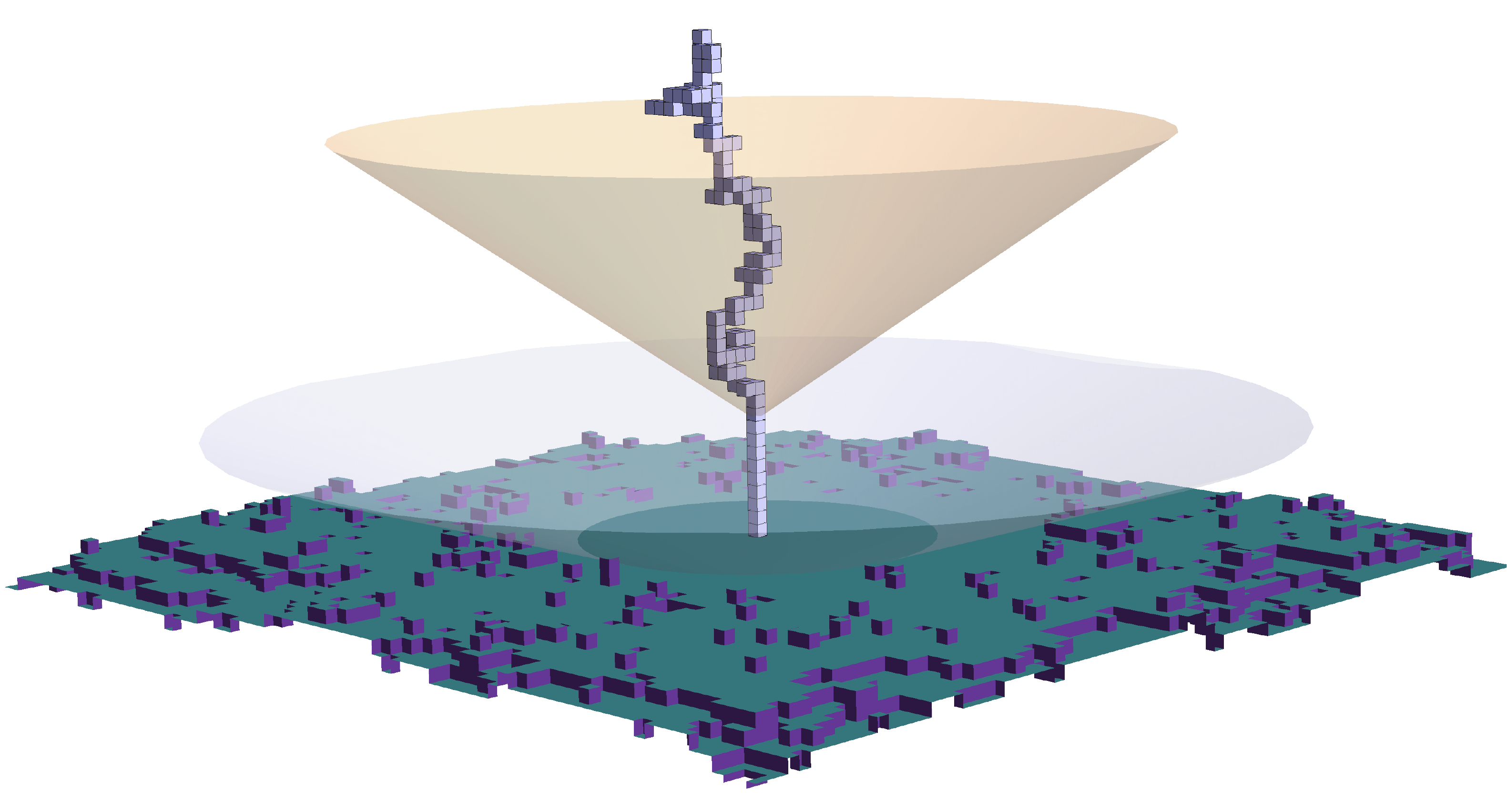};
    \caption{An interface having isolated pillar $\cP_x$ (light purple) and truncated interface $\cI\setminus \cP_x$ (green and dark purple). The sets $\Cone_{x}$ and $\overline \Cone_{x}$ are the regions above the orange cone, and below the light gray cone respectively. 
    This decomposition of isolated pillars ensures that the interactions between the pillar and nearby walls through $\g$ are bounded.}
    \label{fig:isolated-pillar}
\end{figure}

\subsection{Decomposition of interfaces with isolated pillars}
Interfaces with isolated pillars have the nice property that their pillar $\cP_x$ and truncated interface $\cI\setminus \cP_x$ are subsets of two well-separated subsets of $\sF(\mathbb Z^3)$; this property motivates calling their pillar isolated from its surrounding environment, and is important to coupling the laws of isolated pillars under two distinct environments, as we will do in Section~\ref{sec:pillar-couplings}. 
To formalize this notion, let us define the following two sets for $L,h$ implicit from the context (see Fig.~\ref{fig:isolated-pillar}): 
\begin{align}
    \Cone_{x,\bW} & := \{z= (z_1,z_2, z_3)\in \sF(\cL_{>\hgt(\cC_\bW)+ L^3)}): d(\rho(z), x)\le (z_3-\hgt(\cC_\bW))^2 \wedge 10 h \} \label{eq:pillar-cone-def}\,,\\ 
    \overline{\Cone}_{x, \bW} &  := \{z = (z_1,z_2, z_3)\in \sF(\Z^3): d(\rho(z),x)\ge L\,,\, z_3 \le \hgt(\cC_\bW)+ (\log d(\rho(z),x))^2 \}\,, \label{eq:complement-cone-def}
\end{align}
where, we recall, $\cC_\bW$ is the ceiling over $x$ in the interface $\cI_{\bW}$. Let $\cP_{x,\bW,\parallel}^{L^3}$ be the vertical bounding faces of $L^3$ vertically consecutive sites above $x + (0,0,\hgt(\cC_\bW))$, and recall that $\Cyl_{x,r} := \{z\in \mathscr F(\Z^3): d_\rho(z,x) \le r\}$.

\begin{claim}\label{clm:iso-pillar-containments}
Fix any $L$ large and any $h$. Any interface $\cI \in \bI_{\bW}$ having $\cI\restriction_{S} \in \Iso_{x,L,h}$ satisfies
\begin{align}\label{eq:interface-containment-definitions}
    \cI \subset \underbrace{(\Cone_{x,\bW}\cap\cL_{<(\hgt(\cC_\bW) + 10 h)})}_{\bF_{\triangledown}} \cup \underbrace{\cP_{x,\bW,\parallel}^{L^3}}_{\bF_{\parallel}} \cup  \underbrace{(\cL_{\hgt(\cC_\bW)}\cap \Cyl_{x,L})}_{\bF_{-}} \cup \underbrace{\overline{\Cone}_{x,\bW}}_{\bF_{\curlyvee}} \cup \underbrace{\Cyl^c_{x,L^3 h}}_{\bF_{\ext}}\,.
\end{align}
For $\bF_{\triangledown}(\bW),\bF_{\parallel}(\bW), \bF_{-}(\bW), \bF_{\curlyvee}(\bW),\bF_{\ext}$ defined as above, the right-hand side is a disjoint union, 
\begin{align*}
    (\bF_{\triangledown} \cup \bF_{\parallel})\cap (\bF_{-} \cup \bF_{\curlyvee}\cup \bF_{\ext})= \emptyset
\end{align*}
and the pillar $\cP_{x,S}$ is a subset of the first two sets above, while $\cI \setminus \cP_{x,S}$ is a subset of the latter three sets. 
\end{claim}

\begin{proof}
Clearly, it suffices to show the latter claim on the containments of $\cP_{x,S}$ and $\cI\setminus \cP_{x,S}$, to obtain the former about the containment of $\cI$. The disjointness of the two sets is by construction. 
 
 To see that $\cP_{x,S}$ is a subset of $\bF_{\triangledown}\cup \bF_{\parallel}$, we first notice by item~\eqref{item:iso-pillar-increments} of Definition~\ref{def:isolated-pillars}, the bounding faces of the first $L^3$ increments of $\cP_{x,S}$ are exactly the set $\bF_{\parallel}$ (if $h\ge L^3$, otherwise, they are a subset of $\bF_{\parallel}$), as they all have $\fm(\sX_t) =0$ for $t\le L^3$. Moreover, the remaining increments $t\ge L^3$ each satisfy $\fm(\sX_t)\le  t$. Now notice that the maximal horizontal displacement of the pillar at height $z_3-  \hgt(\cC_\bW)$ is at most 
    \begin{align*}
        \sum_{t\le \tau_{z_3}} \fm(\sX_t) \le \sum_{L^3 \le t\le z_3- \hgt(\cC_\bW)} \fm(\sX_t)\le (z_3 - \hgt(\cC_\bW))^2\,,
    \end{align*}
    where $\tau_{z_3}$ is the increment number intersecting height $z_3$. The maximal horizontal displacement is also at most $10 h$ using item ~\eqref{item:iso-pillar-increments} of Definition~\ref{def:isolated-pillars}. It remains to show that $\cP_{x,S}\subset \cL_{<(\hgt(\cC_\bW) + 10 h)}$, which also follows from item~\eqref{item:iso-pillar-increments} of Definition~\ref{def:isolated-pillars}. 

 To see that $\cI \setminus \cP_{x,S}$ is a subset of $(\bF_{-} \cup \bF_{\curlyvee}\cup \bF_{\ext})$, we first of all notice that as  $\rho(\bW)\subset S^c$ and $h= o(d(x,S^c))$, all walls of $\cI$ intersecting $\Cyl_{x,L^3 h}$ are nested in $\cC_\bW$, and are vertical shifts by $\hgt(\cC_\bW)$ of the faces of $\cI\restriction_{S}$ projecting into $\rho(\hull{\cC}_\bW)$. It thus suffices to show that 
 \begin{enumerate}
     \item for all $y\in \cL_{0,n}$ having $d(y,x)\le L$, the interface $(\cI\restriction_{S} \setminus \cP_{x,S}) = (\cI \setminus \cP_{x,S})\restriction_S$ has height zero above $y$,
     \item for all $y\in \cL_{0,n}$ having $L \le d(y,x)\le L^3 h$, the maximal height of $(\cI\restriction_{S} \setminus \cP_{x,S})$ above $y$ is $(\log d(y,x))^2$. 
 \end{enumerate} 
 These are proved as follows. 
    \begin{enumerate}
        \item Notice by item~\eqref{item:iso-pillar-walls} of Definition~\ref{def:isolated-pillars}, the interface $(\cI\restriction_{S}\setminus \cP_{x,S})$ has no walls indexed by sites in $\Cyl_{x,L}$: as such, they must all be at $\hgt(\cC_\bW)$ (using the fact that the base of $\cP_{x,S}$, i.e., $\sB_{x,S}$,  is empty). 
        \item By Observation~\ref{obs:height-nested-sequence-of-walls}, the height of an interface above $y$ is bounded by the sum of the excess areas of all walls that nest it; each such wall must go through an index point with distance at most $d(y,x)$ to $x$ (as no wall of $\cI\restriction_{S}\setminus \cP_{x,S}$ nests $x$), and therefore, by item~\eqref{item:iso-pillar-walls}, there are at most $\log d(y,x)$ many walls of $\cI\restriction_{S}\setminus \cP_{x,S}$ nesting $y$, each having excess area at most $\log d(y,x)$, yielding the desired. 
    \end{enumerate}
Combining the above we conclude the desired inclusions on $\cP_{x, S}$ and $\cI\setminus \cP_{x,S}$. 
\end{proof}

The following lemma then controls the interactions between the pillar and the surrounding environment when the pillar is isolated. Recall from Theorem~\ref{thm:cluster-expansion} that when applying maps on the interfaces, the interactions through the bubbles decays through an effective radius of congruence $\br$; the next lemma shows this contribution is uniformly bounded in $h$ and decaying in $L$ when the radius is attained by an interaction between a pillar and non-pillar face of an interface in $\Iso_{x,L,h}$. 

\begin{lemma}\label{lem:iso-pillar-interactions}
There exists $\bar C$ such that for every $\bW$, every $L$ large, and every $h\ge 1$,
\begin{align*}
    \sum_{f\in \bF_{\triangledown}\cup \bF_{\parallel}}\sum_{g\in \bF_{\curlyvee}\cup \bF_{\ext}} e^{ - \bar c d(f,g)} \le \bar C e^{ - \bar c L}\,\qquad \mbox{and}\qquad \sum_{f\in \bF_{\triangledown}}\sum_{g\in \bF_{-} \cup \bF_{\curlyvee}\cup \bF_{\ext}} e^{ - \bar c d(f,g)} \le \bar C e^{ - \bar c L}\,.
\end{align*}
\end{lemma}
\begin{proof}
We consider three sums that together imply the two inequalities above for some $\bar C$. By summing out $g\in  \bF_{-}\cup \bF_{\curlyvee}$ we can bound the summands with $f\in \bF_{\triangledown}$ and such $g$ by 
\begin{align}\label{eq:ntb-1-third-term}
    \sum_{f\in \bF_{\triangledown}} K e^{ - \bar cd(f,\bF_{\curlyvee}\cup \bF_{-})}\le \sum_{k\ge 1} \sum_{f\in \bF_{\triangledown}: f\cdot e_3 = k} K e^{ - \bar c d(f,\bF_{\curlyvee}\cup \bF_{-})} \le \sum_{L^3 \le k\le 10 h} k^4 e^{ - \bar c \frac{\sqrt k}{(\log k)^2}} \le e^{ - \bar c L}\,,
\end{align}
for all $L$ sufficiently large, using definitions~\eqref{eq:pillar-cone-def}--\eqref{eq:complement-cone-def}. 
Sums over $f\in \bF_{\triangledown}\cup \bF_{\parallel}$ and $g\in \bF_{\ext}$ are at most  
\begin{align}\label{eq:ntb-1-second-term}
    \sum_{f\in \bF_{\triangledown}\cup  \bF_{\parallel}} K e^{ - \bar c d(f,\bF_{\ext})} \le K |\bF_{\triangledown}\cup\bF_{\parallel}| e^{ - \bar c d(\bF_{\triangledown},\bF_{\ext})} \le K \cdot (L^3 + 10 h) \cdot (h^4) \cdot e^{ - \bar c L^3 h/2}\,.
\end{align}
Finally, the sum over $f\in \bF_{\parallel}$ and $g\in \bF_{\curlyvee}$ is at most $L^3 e^{ - \bar c L} \le C e^{ - \bar c L}$. Together these imply the desired for $L$ a large enough constant. 
\end{proof}

\subsection{Construction and properties of
\texorpdfstring{$\Phi_{\iso}$}{Phi\_iso}}
We begin by constructing the map $\Phi_{\iso} = \Phi_{\iso}(x,L,h)$. For the remainder of this section, fix a simply-connected set $S= S_n$, fix a sequence $x = x_n \in S_n$, and let $h = h_n$ be such that $h_n = o(d(x,S_n^c))$.  Further, fix any family of walls $\bW = \bW_n = (W_z)_{z\notin S_n}$ having $\rho(\bW_n)\subset S_n^c$, and let $\cC_\bW$ denote the ceiling of $\cI_{\bW}$ whose projection contains $S_n$. For ease of notation, for a wall $\tilde W$ of an interface $\tilde \cI$, let $\lceil \tilde W\rceil$ be the set of interior ceilings of $\tilde W$ in $\tilde \cI$.

\begin{definition}\label{def:iso-map}
Fix $x,S,\bW$ as above. For every $L$ and every $h = o(d(x,S))$, let  $\Phi_{\iso}= \Phi_{\iso}(x,L,h)$ be the map on $\bI_{\bW}$ constructed in Algorithm~\ref{alg:le-big-map} below.
\end{definition}

\begin{remark}\label{rem:differences-in-the-algorithm}
 We remark on some of the differences between the map $\Phi_{\iso}$ and the matching Algorithm 1 of~\cite{GL19b} used for showing tight exponential tail bounds on the base and increments of the pillar.  
\begin{enumerate}
    \item We add the deletion criterion of step 6 in $\Phi_{\iso}$ to ensure that after application of the map, the resulting interface is well defined, and has isolated pillar at $x$ (see Lemma~\ref{lem:iso-map-well-defined}). This step additionally allows us to remove a more complicated step ({\tt{A2}}) of the algorithm in~\cite{GL19b}, which played a similar role of ensuring well-definedness of the interface therein, and removed atypically large walls near the pillar from the standard wall representation.  
    \item More generally, when processing the spine, the algorithm of~\cite{GL19b} evaluated various criteria of proximity of the increment $\sX_t$ to walls in its surrounding environment, in an iterative manner. Those criteria changed as the algorithm worked its way up the increment sequence, with the distance to a nearby wall evaluated not just 
     on the initial interface, but also on the \emph{prospective} output pillar of the algorithm were it to trivialize all increments $(\sX_s)_{s\le t}$. This was to ensure that the interactions between the pillar and non-deleted walls of the exterior were well-controlled in both $\cI$ and the resulting $\Phi(\cI)$. For us, the interactions in $\Phi_{\iso}(\cI)$ between its pillar and surrounding environment are easily shown to be uniformly bounded by Lemma~\ref{lem:iso-pillar-interactions} because $\Phi_{\iso}(\cI)$ has isolated pillar.
\end{enumerate}
 Beyond these two simplifications allowed by the notion of isolated pillars, the remainder of the map is qualitatively the same as in~\cite{GL19b}, up to the changes of considering the restricted interface $\cI \restriction_{S}$ rather than the full interface, and deleting wall-clusters rather than groups-of-walls of marked sites from the standard wall representation. 
\end{remark}

\begin{figure}
\begin{algorithm2e}[H]
\DontPrintSemicolon
\nl
Let $\{\tilde W_y : y\in \cL_{0,n}\}$ be the walls of $\cI\setminus \cS_{x,S}$. Let $(\sX_i)_{i\geq 1}$ be the increments of $\cS_{x,S}$.
\;

\BlankLine
\tcp{Base modification}
\nl
Mark $\bar x:= \{f\in \cL_{0,n}: f\sim^* x\}\cup \{x\}$ and $\rho(v_1)$ for deletion.\; 

\nl 
\If{the interface with standard wall representation $\Theta_{\textsc{st}}\tilde \fW_{v_1,S}$ has a cut-height}
{Let $h^\dagger$ be the height of the highest such cut-height.\;
Let $y^\dagger$ be the index of a wall that intersects $(\cP_{x,S} \setminus \tilde \fW_{v_1,S}) \cap \cL_{h^\dagger}$ and mark $y^\dagger$ for deletion.\;
}
\BlankLine
\tcp{Spine modification}

\nl
\For{$j=1$ \KwTo $\sT+1$}{     
     \If(\hfill\tcp*[h]{(A1)}){
\quad$ \fm(\sX_j) \geq \begin{cases} 0 & \mbox{if }j\le L^3 \\ j-1 & \mbox{if }j>L^3\end{cases}$ \quad\mbox{}}{Let $\fs \gets j$.} 
      \If(\hfill\tcp*[h]{(A2)}){\quad $d( \lceil \tilde W_y\rceil,\sX_j)\le (j-1)/2$ \qquad\quad for some $y\in S_n$\quad\mbox{}}{
      Let $\fs\gets j$ and let $y^*$ be the minimal index $y\in S_n$ for which~({\tt A2}) holds.}
}
 Let $j^* \gets \fs$ and mark $y^*$ for deletion.\;

 \nl 
 \If(\hfill\tcp*[h]{(A3)}){$|\sF(\cS_{x,S})| >5h$}{let $\fs \gets \sT+1$ and $j^* \gets \fs$.}
 
 \BlankLine
\tcp{Environment modification}
 \nl
 \For {$y\in \cL_{0}\cap \Cyl_{L^3 h}(x)$}{
    \If{\quad$\fm(\tilde W_y) \ge \begin{cases}0 & \mbox{if }d(y,x)\le L \\  \log [d(y,x)] & \mbox{else}\end{cases}$}{Mark $y$ for deletion}
    }
 
\BlankLine
\tcp{Reconstructing the interface}
\BlankLine

\nl 
\lForEach{$y\in\cL_{0,n}$ marked for deletion}
{remove $\Theta_{\textsc{st}}\Clust(\tilde \fW_{y,S})$ from $(\Theta_{\textsc{st}}\tilde W_y)_{y\in \cL_{0,n}}$.}

 \nl Add the standard wall $\Theta_{\textsc{st}}W_{x,\parallel}^{\mathbf h}$ consisting of the bounding vertical faces of $(x+ (0,0,\frac i2))_{i=1}^{\mathbf h}$ where $\mathbf h := (\hgt(v_1)- \frac 12) - \hgt(\cC_\bW)$.

\nl Let $\cK$ be the interface with the resulting standard wall representation.

\nl Let
\[ \cS \gets 
\begin{cases}\big(\underbrace{X_\trivincr,\ldots,X_\trivincr}_{\hgt(v_{j^*+1})-\hgt(v_1)},\sX_{j^*+1},\ldots,\sX_{\sT},\sX_{>\sT}\big) &\mbox{ if \mbox{({\tt{A3}}) is not violated}},\\
\noalign{\medskip}
 \big( \underbrace{X_\trivincr,\ldots,X_\trivincr}_{h- \mathbf h}\big) &\mbox{ if \mbox{({\tt{A3}}) is violated}}\,.\end{cases}\,.
 \]
\nl Obtain $\Phi_{\iso}(\cI)$ by appending the spine with increments $\cS$ to $\cK$ at $x+ (0,0,\hgt(\cC_\bW)+ \mathbf h)$.\;

\caption{The map $\Phi_{\iso}= \Phi_{\iso}(x,L,h)$}
\label{alg:le-big-map}
\end{algorithm2e}
\end{figure}

In what follows, for all $h'$, let $E_{x,S}^{h'}$ be the event $\{\hgt(\cP_{x,S})\ge h'\}$. 

\begin{lemma}[Well-definedness of $\Phi_\iso$]\label{lem:iso-map-well-defined}
For every $L$ large and every $h'\le h$,  for every $\cI \in \bI_{\bW}\cap E_{x,S}^{h'}$, the resulting $\cJ = \Phi_{\Iso}(\cI)$ is a well-defined interface in $E_{x,S}^{h'}\cap \bI_{\bW}$ and $\cJ\restriction_S\in \Iso_{x,L,h}$. 
\end{lemma}

\begin{proof}
First of all, we claim that the standard wall representation of $\cK$ (defined in step 9) is admissible; this follows from the facts that its standard wall representation is a subset of that of $\cI\setminus \cS_{x,S}$, together with the fact that prior to the addition of $W_{x,\parallel}^{\mathbf h}$, there were no walls incident to the face $x$ (as $\bar x$ is marked for deletion, and therefore $\tilde \fW_{\bar x,S} := \bigcup_{f\in \bar x} \tilde \fW_{f,S}$ is deleted). 

For validity of the interface and thus well-definedness, it remains to show that the pillar we add in step~11 does not intersect any part of the interface $\cK$. This follows from Claim~\ref{clm:iso-pillar-containments}, showing that the spine we generated after step 10 of $\Phi_{\iso}$ and therefore satisfying criterion~\eqref{item:iso-pillar-increments} of the definition of $\Iso_{x,L,h}$ is confined to $\bF_{\parallel}\cup \bF_{\triangledown}\cap \Cyl_{x, 10 h}$. At the same time, all faces of $\cK$ not that are not at height  $\hgt(\cC_\bW)$ are confined to $\bF_{\curlyvee}\cup \Cyl_{x, L^3 h}^c$. As such, the two sets are not incident one another except at $x+ (0,0,\hgt(\cC_\bW))$ as desired, and the interface generated by appending $\cS$ to $\cK$ at $x + (0,0,\hgt(\cC_\bW) + \mathbf h)$ forms a valid interface. 

In order to see that this interface is in $\bI_\bW$, notice first that $\cK$ is in $\bI_\bW$ as it only consisted of the deletion of standard walls, and the addition of the single standard wall $W_{x,\parallel}^{\mathbf h}$ whose projection is interior to $S$. Noting further that the projection $\rho(\cS + x + (0,0,\hgt(\cC_\bW) + \mathbf h))$ is a subset of $\rho(\Cyl_{x,10h})\subset S$, we deduce that the walls of $\bW$ are not projected onto, and remain unchanged in the wall representation of $\cJ$ as desired. 

To see that $\cJ\restriction_{S}$ is in $\Iso_{x,L,h}$, notice from the above, that the pillar $\cP_{x,S}^{\cJ}$ (the pillar of $\cJ\restriction_S$ at $x$) has empty base, and steps 4 and 5 of $\Phi_\iso$ ensure that its increment sequence satisfies item~\eqref{item:iso-pillar-increments} in Definition~\ref{def:isolated-pillars}. Step 6 ensures that the walls of $\cJ\restriction_S \setminus \cP_{x,S}^\cJ$ satisfy item~\eqref{item:iso-pillar-walls} in Definition~\ref{def:isolated-pillars}. Finally, steps 8 and 10 together imply that $\cJ\in E_{x,S}^{h'}$ as claimed. 
\end{proof}

We next consider the change in energy under the map $\Phi_{\iso}$ showing a series of important bounds on the quantity. For ease of notation, let $\bY$ be the set of $y\in \cL_{0,n}$ marked for deletion, and let $\bD$ be a representative set of $y\in \cL_{0,n}$ of all deleted standard walls. First of all, note that the change in energy between $\cI$ and $\cJ := \Phi_{\iso}(\cI)$ is given by
\begin{align}\label{eq:m(I;J)}
\fm(\cI;\cJ)= \begin{cases}
\sum_{z\in \bD} \fm(\tilde W_z) + \sum_{j=1}^{j^*} \fm(\sX_j) - |W_{x,\parallel}^{\mathbf h}|\,, & \mbox{({\tt{A3}}) is not violated}  \\ 
\sum_{z\in \bD} \fm(\tilde W_z) + \sum_{j=1}^{\sT+1} \fm(\sX_j) + 4(\hgt(v_{\sT+1}) - h)  - |W_{x,\parallel}^{\mathbf h}|\,, & \mbox{({\tt{A3}}) is violated}.
\end{cases}
\end{align}

The following inequalities regarding $\fm(\cI; \cJ)$ will be used repeatedly. In particular, these will imply that if $\cI\restriction_S \notin \Iso_{x,L,h}$ then $\fm(\cI; \Phi_{\iso}(\cI))\ge 1$ giving us the energy gain we rely on to prove Theorem~\ref{thm:shape-inside-ceiling}.

\begin{claim}\label{clm:m(W)}For every $L$ large, and every $\cI \in \bI_{\bW}$, denoting by $\cJ = \Phi_\iso(\cI)$, we have
\begin{align}\label{eq:|W_x^J|}
|W_{x,\parallel}^{\mathbf h}| \leq \frac23 \fm(\tilde \fW_{v_1,S} \cup \tilde \fW_{y^\dagger,S})\,,\qquad \mbox{and thus} \qquad \fm(\cI;\cJ) \ge \frac 13 \fm(\bigcup_{y\in \bD} \tilde W_y) + \sum_{j=1}^{j^*}\fm(\sX_j)\,.
\end{align}
In particular
\begin{equation}\label{eq:m(I;J)-WxJ} |W_{x,\parallel}^{\mathbf h}|\leq 2 \fm(\cI;\cJ)\,,\quad\mbox{and}\quad \fm\Big(\bigcup_{y\in \bD} \tilde W_y\Big) \leq 3 \fm(\cI;\cJ)\,,
\end{equation}
and 
\begin{equation}\label{eq:m(I;J)-XAJ}
\begin{cases}
j^* - 1 \leq (2\vee L^3) \fm(\cI;\cJ) & \quad \mbox{if ({\tt{A3}}) is not violated} \\ 
h- \mathbf h \le \fm(\cI;\cJ) & \quad \mbox{if ({\tt{A3}}) is violated}
\end{cases}\,.
\end{equation}
\end{claim}

\begin{proof}
The proof goes as the proof of Claim 4.7 of~\cite{GL19b}. Let $\bV = \bigcup_{y\in \bD} \tilde W_y$ be the set of deleted walls. We first observe from definitions of walls and ceilings---a detailed proof is given in Corollary 11 of~\cite{GL19b}---that there are no cut-points in the interface with standard wall representation $\tilde \fW_{v_1,S} \cup \tilde \fW_{y^\dagger,S}$, and therefore, 
\begin{align*}
    |W_{x,\parallel}^{\mathbf h}|  = 4(\hgt(v_1) - \tfrac 12 -\hgt(\cC_\bW)) \le \tfrac 23 (\fm(\tilde \fW_{v_1,S}\cup \tilde \fW_{y^\dagger, S}))\,. 
\end{align*}
Thus, $|W_{x,\parallel}^{\mathbf h}|\le \frac 23 \fm(\bV)$ and the two inequalities of~\eqref{eq:m(I;J)-WxJ} follow from the fact that $\fm(\cI;\cJ)\ge \fm(\bV)- |W_{x,\parallel}^{\mathbf h}|$. If ({\tt{A3}}) is violated, then the replacement of $\cS_{x,S}$ with $\cS$ induces an excess area of $5h - 4(h- \mathbf h)$ which is obviously at least $h - \mathbf h$. 
Otherwise, it remains to bound $j^* - 1$ by $6L^3 \fm(\cI;\cJ)$ via a case analysis of the violation attaining $j^*$; without loss, we assume $j^* >1$ so that this is nontrivial---in that case, $j^*$ was set for the last time due to one of the three possible violations: 
\begin{enumerate}[({A}1)]
    \item If criterion {(\tt{A1})} is violated, then either $j_* \le L^3$ or $j^* \le \fm(\sX_j)+1$. 
    \item In this case, $d( \lceil \tilde W_y\rceil,\sX_{j^*}) \leq (j-1)/2$ for $y  = y^*$.
   By the simple geometric observation (see Fact 2.8 from~\cite{GL19b}) that for every $j, y$, we have
    \begin{align}\label{eq:useful-inequality-from-GL19b}
    j-1 \leq d( \lceil \tilde W_y\rceil,\sX_j) + \fm(\tilde \fW_{y,S})\,,
    \end{align}
    applied to $j^*$ and $y^*$, we have
    \[ j^*-1 \leq d( \lceil \tilde W_{y^*}\rceil,\sX_{j^*}) + \fm(\tilde \fW_{y^*,S}) \leq (j^*-1)/2 + \fm(\tilde \fW_{y^*,S})\,, \qquad \mbox{so} \qquad j^*-1 \leq 2 \fm(\tilde \fW_{y^*,S})\,.\]
    \item In this case, $j^*  = \sT+1$ and $\fm(\cI;\cJ) \ge \sum_{j} \fm(\sX_j) + 4(\hgt(v_{\sT+1}) - h) \ge h \vee 4(\sT+1 - h)$. 
\end{enumerate}
In all of the above situations, we have $j^* -1 \le (2\vee L^3) \fm(\cI;\cJ)$.
\end{proof}

\subsection{Analysis of \texorpdfstring{$\Phi_{\iso}$}{Phi\_iso}}
The two main steps to the proof of Theorem~\ref{thm:shape-inside-ceiling} given $\Phi_{\iso}$ and Lemma~\ref{lem:iso-map-well-defined}, are the analysis of the weight gain and multiplicity of the map. These are captured by the following two propositions.

\begin{proposition}\label{prop:partition-function-contribution}
There exists $C>0$  and such that for all $\beta>\beta_0$, all $L$ large, and every $\cI \in \bI_\bW$, 
\begin{align*}
    \Big|\log \frac{\mu_n^{\mp}(\cI)}{\mu_n^{\mp}(\Phi_{\iso}(\cI))} + \beta \fm(\cI;\Phi_{\iso}(\cI))\Big| \leq C L^3 \fm(\cI;\Phi_{\iso}(\cI))\,.
\end{align*}
\end{proposition}

\begin{proposition}\label{prop:multiplicity}
There exists $C_\Phi$ such that 
for every $L$ large and every $M\geq 1$ and $0\le h'\le h$,
\begin{align*}
\max_{\cJ\in\Phi_{\iso}(\bI_{\bW_n}\cap E_{x,S}^{h'})}\left|\{\cI\in\Phi_{\iso}^{-1}(\cJ) \,:\; \fm(\cI;\cJ)=M\}\right| \leq C_\Phi^{L^3 M}\,.
\end{align*}
\end{proposition}

We defer the proofs of these two lemmas to Section~\ref{sec:deferred-proofs} as the proofs are quite involved, and the key ideas in them (the manner in which the interaction term and multiplicity are controlled) are quite similar to those of the more involved map in~\cite{GL19b}. The additional ingredients required to replace groups of walls by wall clusters, so that no walls of $\bW$ are deleted, already appeared in the conditional rigidity result of Section~\ref{sec:wall-clusters}.

\begin{proof}[\textbf{\emph{Proof of Theorem~\ref{thm:shape-inside-ceiling}}}]
Observe first of all that if $\cI\restriction_{S} \notin \Iso_{x,L,h}$, then $\fm(\cI;\Phi_\iso(\cI)) \ge 1$ necessarily. This is because if $\cI\restriction_S \notin \Iso_{x,L,h}$, either it has $\bD \ne \emptyset$, in which case this follows from~\eqref{eq:m(I;J)-WxJ}, or it has $\fm(\sX_t) >0$ for some $t\le j^*$ in which case this follows directly from~\eqref{eq:m(I;J)}.

It thus suffices to prove that for every $r\ge 1$,
  \begin{align*}
       \mu_{n}^{\mp} \big(\fm(\cI; \Phi_{\iso}(\cI)) \geq r\mid \bI_{\bW}, E^{h'}_{x,S}\big) \leq C\exp\big[-  (\beta- CL^3) r)\big]\,.
  \end{align*} 
  and take $L = L_\beta = \beta^{1/4}$, say.  
For every $r\ge 1$, 
\begin{align*}
    \sum_{M\geq r}\,\, \sum_{\substack{\cI\in \bI_\bW\cap E^{h'}_{x,S}\\ \fm(\cI;\Phi_{\iso}(\cI)) = M}} \mu_n^\mp (\cI) & \leq \sum_{M\geq r} \sum_{\substack{\cI\in \bI_\bW\cap E^{h'}_{x,S}\\ \fm(\cI;\Phi_\iso(\cI)) = M}} e^{ - (\beta- C)M} \mu_n^\mp(\Phi_{\iso}(\cI))  \\
    & = \sum_{M\geq r}\,\, \sum_{\cJ\in \Phi_{\iso}(\bI_{\bW}\cap E^{h'}_{x,S})}  \mu_n^\mp(\cJ) \,\, \sum_{\substack{\cI\in \Phi_{\iso}^{-1}(\cJ) \\ m(\cI; \Phi_{\iso}(\cI))=M}} e^{-(\beta - CL^3 )M} \\
    & \leq \sum_{M\geq r} C_\Phi^M e^{- (\beta- C L^3)M} \mu^\mp_n(\Phi_{\iso}(\bI_{\bW}\cap E^{h'}_{x,S}))\,.
\end{align*}
In the first inequality, we used  Proposition~\ref{prop:partition-function-contribution} and in the second inequality, we used Proposition~\ref{prop:multiplicity}. 

Now, noting by Lemma~\ref{lem:iso-map-well-defined} that $\Phi_{\iso}(\bI_{\bW}\cap E_{x,S}^{h'}) \subset ( \bI_{\bW}\cap E^{h'}_{x,S})$, 
we deduce that
\[
\mu_n^\mp(\fm(\cI; \Phi_{\iso}(\cI)) \geq r,  \bI_\bW\cap E^{h'}_{x,S})\leq C e^{-(\beta - CL^3 -\log C_\Phi) r}\mu_n^\mp(\bI_\bW\cap E^{h'}_{x,S})\,.
\]
Dividing through by $\mu_n^\mp(\bI_\bW\cap E^{h'}_{x,S})$ then yields the desired conditional bound. \end{proof}

\section{Independence of pillars from the surrounding environment}\label{sec:pillar-couplings}
In this section, we prove the main decorrelation result, showing that the typical pillars in $S_n$ conditionally on $\bW_n$, are close in law, to those in infinite volume, even given that the pillars attain some high height~$h$. 
More precisely, we show that on the event that the pillars are isolated, the Radon--Nikodym derivative between pillars in $\mathbb Z^3$ and pillars conditionally on some external environment $\bW_n$ is bounded by $1\pm \epsilon_\beta$.

\begin{theorem}\label{thm:cond-uncond}
For every $\beta>\beta_0$ there exist $L_\beta \uparrow \infty$ and $\epsilon_\beta \downarrow 0$ such that the following hold. Let $S_n \subset \cL_{0,n}$ be such that $\cL_0\setminus S_n$ is connected, let $\bW_n = (W_z)_{z\in S_n}$ be an admissible collection of walls such that $\rho(\bW_n) \subset S_n^c$, and let $x_n, h_n$ be such that $h_n = o(d(x_n, S_n^c))$. 
Then
\begin{align}
    \label{eq:cond-uncond}
     1-\epsilon_\beta \leq \frac{\mu_{n}^\mp\left(\cP_{x,S_n}\in A\,,\, \cI\restriction_{S_n}\in\Iso_{x,L,h} \mid \bI_{\bW_n}\right)}{
     \mu_{\Z^3}^\mp\left(\cP_o\in A\,,\, \cI \in\Iso_{o,L,h}\right)} \leq 1+\epsilon_\beta\,, 
\end{align}
for every subset $A$ of pillars $\cP$ compatible with $\cI$ being $(L,h)$-isolated at $\cP$. 
\end{theorem}

Before proving this result, we combine it with Theorem~\ref{thm:shape-inside-ceiling} to deduce Theorem~\ref{thm:cond-uncond-simple}.
\begin{proof}[\emph{\textbf{Proof of Theorem~\ref{thm:cond-uncond-simple}}}]
For ease of notation, define the following distributions on pillars:
\begin{align*} \phi(\cdot) &:= \mu_{n}^\mp \left(\cP_{x,S_n}\in\cdot \mid \bI_{\bW_n}\right)\,, 
&\psi(\cdot) := \mu_{\Z^3}^\mp \left(\cP_o\in\cdot \right)\,, \\
 \widetilde\phi(\cdot) &:= \mu_{n}^\mp \left(\cP_{x,S_n}\in\cdot\,,\,\cI\restriction_{S_n}\in\Iso_{x,L,h} \mid \bI_{\bW_n}\right)\,, 
&\widetilde\psi(\cdot) := \mu_{\Z^3}^\mp \left(\cP_o\in\cdot\,,\, \cI\in\Iso_{o,L,h}\right)\,.
\end{align*}
Define the set $E^h$ as the set of pillars (modulo horizontal translations) that attain height $h$, that is, $\cP_{x,S}\in E^h$ if $\cI\restriction_S \in E_x^h$. Let $\fP_h$ be the subset of all pillars in $E^h$ that are compatible with $\Iso_{x,L,h}$ i.e., appear as $\cP_x$ in at least one interface $\cI\in \Iso_{x,L,h}$.
With these notations, we have by~\eqref{eq:Ph-subset-Lambda} and~\eqref{eq:Ph-subset-Z3} (with $h'=h$) that
\begin{align*} &\widetilde\phi( \fP_h) \geq (1-\epsilon_\beta) \phi(E^h)\,, 
 & \widetilde\psi(\fP_h) \geq (1-\epsilon_\beta)\psi(E_h)\,,
\end{align*}
whereas the inequalities in~\eqref{eq:cond-uncond} (with $A=\fP_h$) imply
\[ (1-\epsilon_\beta)\widetilde\psi(E^h) \leq \widetilde{\phi}(E^h)\leq (1+\epsilon_\beta)\widetilde\psi(E^h)\,.\]
Showing~\eqref{eq:cond-uncond-coupling-1} amounts to proving that 
\[ (1-\epsilon_\beta)\psi(E^h) \leq \phi(E^h) \leq (1+\epsilon_\beta)\psi(E^h)\,,\]
which immediately follows from the preceding inequalities and the facts  $\widetilde\phi \leq \phi $, $\widetilde\psi \leq\psi$ and $\fP_h\subset E^h$, as
\begin{align*}
    \phi(E^h)&\geq \widetilde\phi(\fP_h) \geq  (1-\epsilon_\beta)\widetilde\psi(\fP_h)\geq (1-\epsilon_\beta)^2\psi(\fP_h)\,,\qquad\mbox{ and}\\ \phi(E^h) &\leq \frac1{1-\epsilon_\beta}\widetilde\phi(\fP_h) \leq \frac{1+\epsilon_\beta}{1-\epsilon_\beta} \widetilde\psi(\fP_h)\leq \frac{1+\epsilon_\beta}{1-\epsilon_\beta} \psi(E^{h})\,.
\end{align*}
It remains to show~\eqref{eq:cond-uncond-coupling-2}. Let $\bar A\subset E^h$, and write $A = \bar A \cap \fP_h$. We have
\begin{align*} \psi(\bar A) &= \widetilde\psi(A) + \mu_{\Z^3}^\mp(\cP_o\in \bar A\,,\,\cI\notin \Iso_{h}) \leq \widetilde\psi(A) + \mu_{\Z^3}^\mp(\cP_o\in E^h\,,\,\cI\notin \Iso_{h})\,,
\end{align*}
and in particular,
\[\frac{\widetilde\psi(A)}{\psi(E^h)}\leq \psi_h(\bar A \mid E^h) \leq \frac{\widetilde\psi(A)}{\psi(E^h)} + \epsilon_\beta  \,.\]
Similarly, 
by the same argument applied to $\phi$, and using the above inequalities relating $\widetilde\phi$ and $\widetilde\psi$, as well as $\phi(E^h)$ and $\psi(E^h)$,
 we find that 
\[
\frac{(1-\epsilon_\beta)\widetilde\psi(A)}{(1+\epsilon_\beta)\psi(\overline \fP_h)}\leq
\frac{\widetilde\phi(A)}{\phi(\overline\fP_h)}\leq  \phi(\bar A \mid \overline\fP_h) \leq \frac{\widetilde\phi( A)}{\phi(\overline \fP_h)}+ \epsilon_\beta
\leq
\frac{(1+\epsilon_\beta)\widetilde\psi(A)}{(1-\epsilon_\beta)\psi(\overline\fP_h)} + \epsilon_\beta
\,.\]
Combining the last two displays yields
\[ 
\left|\psi(\bar A \mid E^h) - \phi(\bar A \mid E^h) \right|\leq \epsilon_\beta + \frac{2\epsilon_\beta}{1-\epsilon_\beta} \frac{\widetilde\psi(A)}{\psi(E^h)} \leq \epsilon_\beta + \frac{2\epsilon_\beta}{1-\epsilon_\beta}\,,
\]
which gives the desired for some other sequence $\epsilon_\beta$. 
\end{proof}

The majority of this section (Sections~\ref{subsec:swap-map-construction}--\ref{subsec:swap-main-proof}) is now devoted to the proof of Theorem~\ref{thm:cond-uncond}. In Section~\ref{subsec:crude-pillar-correlations}, we end with a cruder decorrelation estimate used on interactions between pillars at $o(h)$ distances.

\subsection{Constructing a swapping map}\label{subsec:swap-map-construction}
We will use a 2-to-2 map, similar to those used in Section~7 of~\cite{GL19a} (for proving stationarity of the increment sequence of a tall pillar) to swap $\cP_{x,S_n}$ in an interface in $\Lambda_n$ having $\cI\in \bI_{\bW_n}$, with  $\cP_{x'}$ in an interface in some $\Lambda_m$. The fact that the pair of pillars $\cP_{x'}$ and $\cP_{x,S_n}$ are isolated per Definition~\ref{def:isolated-pillars}, ensures that their weights, in their different environments, are close to one another. 

\begin{definition}
Recall Definition~\ref{def:isolated-pillars} and for $L$, $n,m$, $h = h_n$ and $x= x_n \in \cL_{0,n}$ and $x' = x_m'\in \cL_{0,m}$, let 
\begin{align*}
    \Iso_{h,S}^{(n)} := \{\cI: \cI\restriction_{S_n}\in \Iso_{x,L,h}\}\,, \qquad \mbox{and}\qquad \Iso_{h}^{(m)} := \{\cI': \cI' \in \Iso_{x',L,h}\}\,.
\end{align*}
\end{definition}

Interfaces in $\Iso_{x,L,h}$ are such that the pillar $\cP_x$ is sufficiently isolated from its surrounding walls that its interactions (as far as the sub-critical bubbles, measured through the term $\g$ in Theorem~\ref{thm:cluster-expansion} are concerned) with any two different ``environments" $\cI \setminus \cP_o$ are close. This lets us perform a swapping operation that interchanges an isolated pillar under $\mu_n^\mp(\cdot \mid \bI_{\bW_n})$ with one drawn from $\mu_{\mathbb Z^3}^\mp$ and show that this swapping operation preserves weights on the events $\cI \in \Iso_{h,S}^{(n)}$ and $\cI' \in \Iso_h^{(m)}$. We formalize this as follows.

For the next few subsections (through Section~\ref{subsec:swap-main-proof}), we fix ourselves in the setting of Theorem~\ref{thm:cond-uncond} by taking sequences $h=h_n$ such that $1\le h\le n \le m$, $S= S_n \subset \cL_{0,n}$ such that $\cL_0\setminus S_n$ is connected, $x= x_n \in S_n$ such that $h_n = o(d(x_n, \partial S_n))$ and $x' = x_m'$ such that $h_n = o(d(x'_m,\partial \cL_{0,m})$. Further, fix $\bW = \bW_n = (W_z)_{z\notin S_n}$ having $\rho(\bW_n) \subset S_n^c$ and let $\cC_{\bW}$ be the ceiling in the interface $\cI_{\bW}$ whose projection includes $x$. 

\begin{figure}
\vspace{-0.21in}
    \centering
    \begin{tikzpicture}
    \node at (-4.75,1.5) {\includegraphics[width  = .4\textwidth]{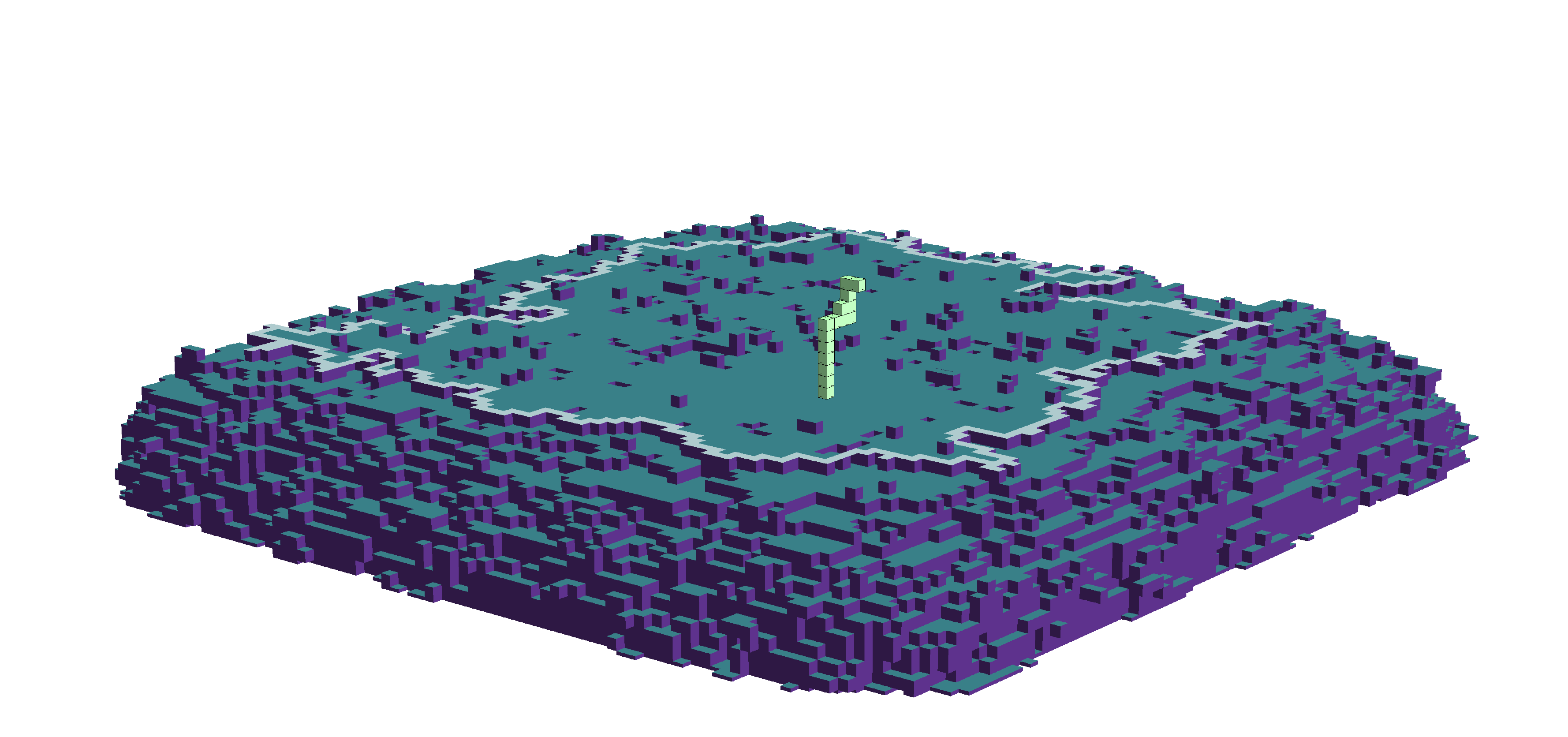}};
    \node at (-4.75,-1.5) {\includegraphics[width  = .42\textwidth]{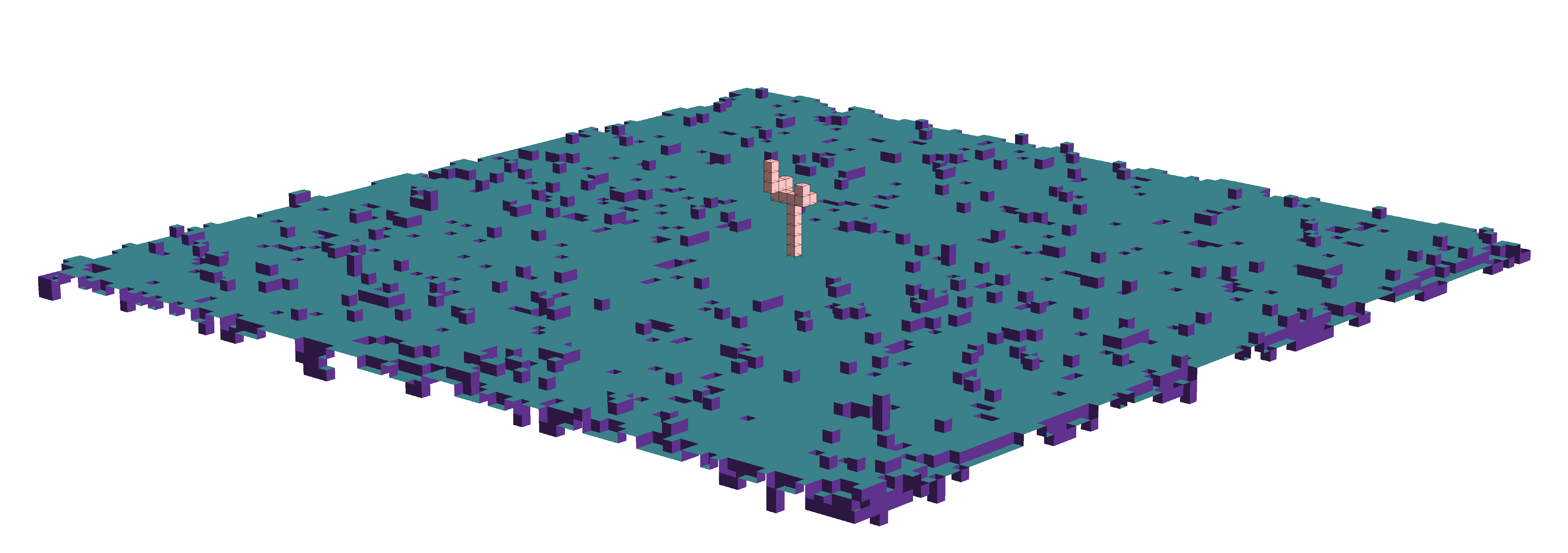}};
     \node at (4.75,1.5) {\includegraphics[width  = .42\textwidth]{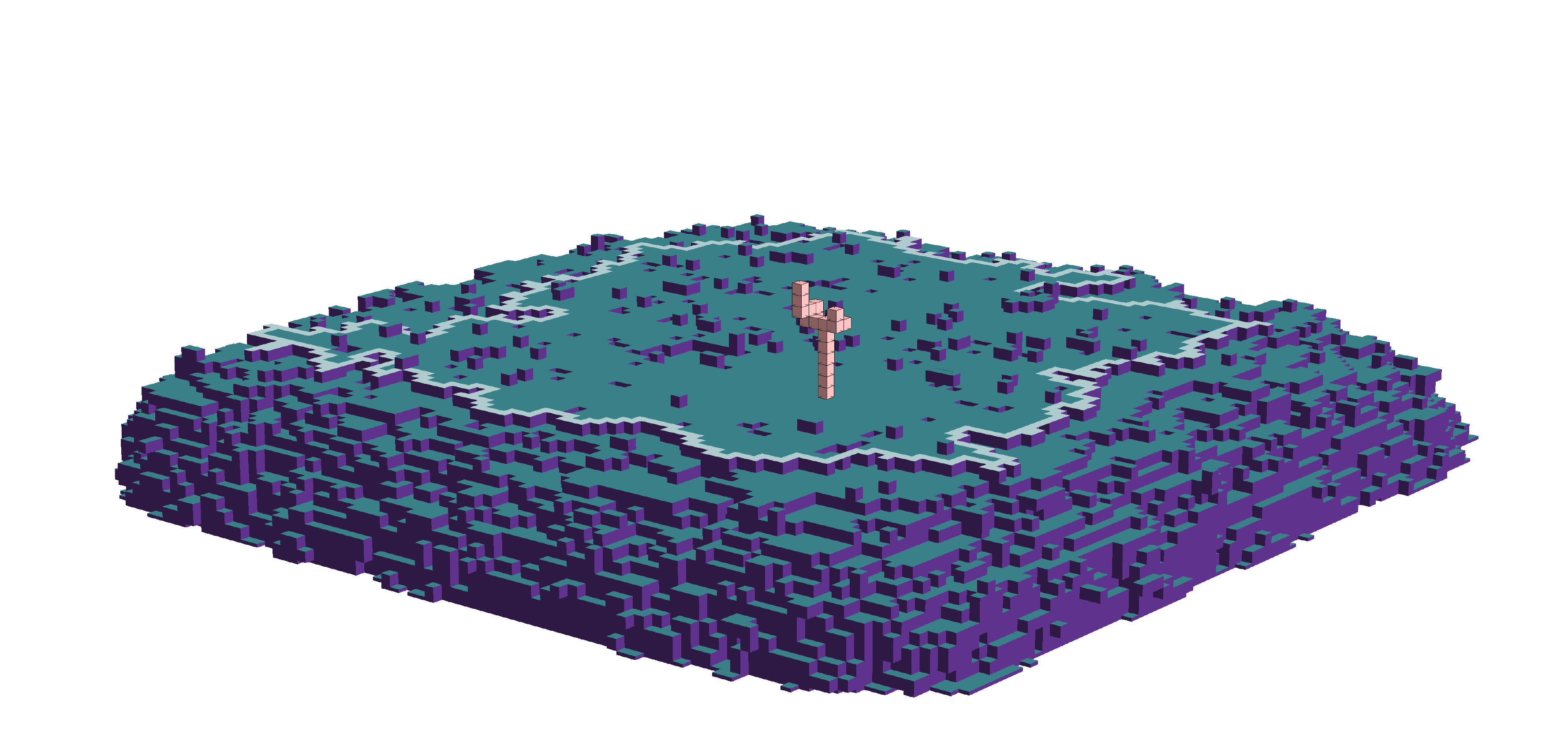}};
    \node at (4.75,-1.5) {\includegraphics[width  = .4\textwidth]{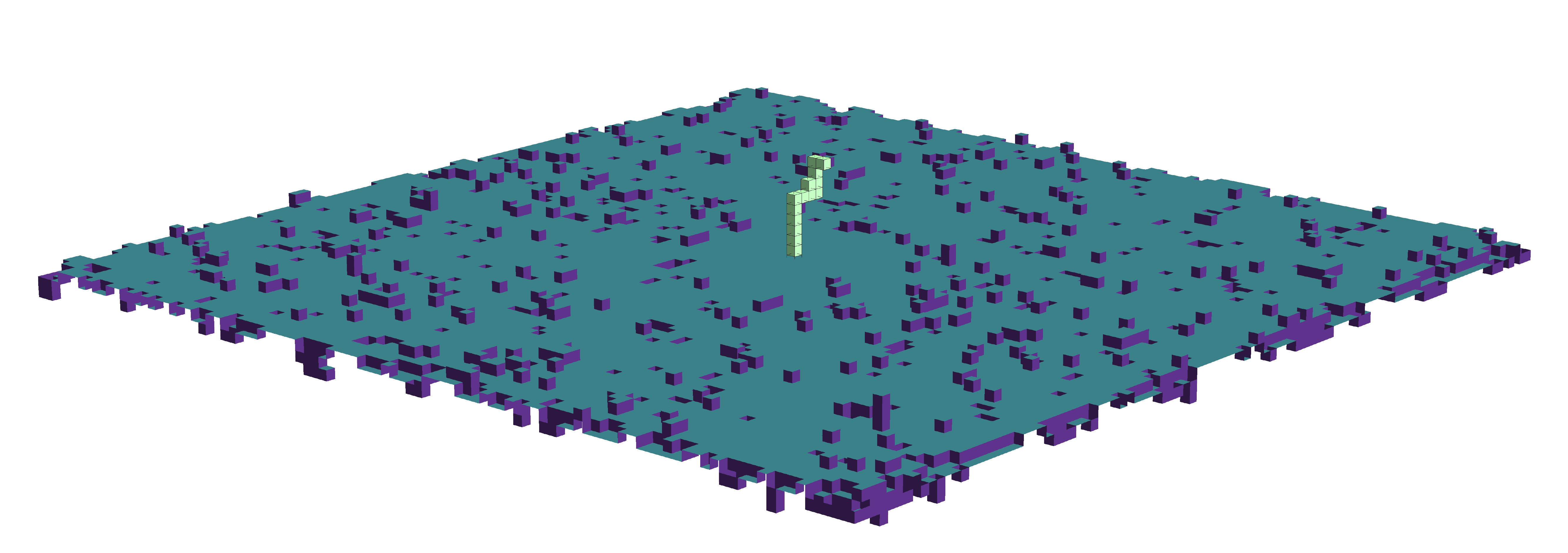}};
    
    \node[circle,draw,opacity=0.3,fill=gray!10,inner sep=0mm,minimum width=20pt] (green) at (-4.5,1.6) {
    };
    
    \node[circle,draw,color=gray!50,opacity=0.3,fill=gray!10,inner sep=0mm,minimum width=20pt] (red) at (-4.7,-1.2) {
    };
    
    \path (green.south west) edge [<->, color=gray!50, thick, bend right=20] (red.north west);

    \draw[->, thick] (-1.25,0)--(1.25,0);
    \node[font  = \large] at (0,.25) {$\Phi_\swap$};
    \end{tikzpicture}
    \caption{The  map $\Phi_\swap$ takes two interfaces $(\cI,\cI')$ having isolated pillars $\cP_{x,S}$ and $\cP_{x'}'$ respectively (left), and swaps their pillars to obtain two interfaces $(\cJ,\cJ')$ (right).}
    \label{fig:swap-map}
\end{figure}

\begin{definition}\label{def:swap-map}
Take $L = L_\beta$ to be chosen later, and consider two interfaces
$$ \cI \in \Iso_{h,S_n}^{(n)}\,, \qquad \mbox{and} \qquad \cI'\in \Iso_{h}^{(m)}\,.$$
The map $\Phi_{\swap} = \Phi_{\swap}(x,S,\bW)$ acts on such pairs of interfaces, and from $(\cI,\cI')$ constructs the pair $(\Phi_{\swap}^1(\cI,\cI'), \Phi_{\swap}^2(\cI,\cI'))$ as follows: 
\begin{enumerate}
    \item Construct $\Phi_{\swap}^1(\cI,\cI')$ by taking the interface $\cI$ and replacing $\cP_{x,S}$ with $\cP'_{x_m'}$ (in $\sigma(\cI)$, flipping all sites in $\sigma(\cP_{x,S})$ to minus, then flipping sites in the shift $\sigma(\cP_{x'}') - x'+x+ (0,0,\hgt(\cC_\bW))$ to plus). 
    \item Construct $\Phi_{\swap}^2(\cI,\cI')$ by taking the interface $\cI'$ and similarly replacing $\cP'_{x'}$ with $\cP_{x,S}$. 
\end{enumerate}
\end{definition}

\noindent See Figure~\ref{fig:swap-map} for a depiction of $\Phi_\swap$.

\subsection{Properties of \texorpdfstring{$\Phi_{\swap}$}{Phi\_swap}}
We show first that the fact that $\cI\restriction_{S}\in \Iso_h^{(n)}$ and $\cI' \in \Iso_h^{(m)}$ ensures that the map of Definition~\ref{def:swap-map} is a well-defined bijection on the right sets of interfaces. Towards this, recall the definitions of $\Cone_{x,\bW}$ and $\overline\Cone_{x,\bW}$ from~\eqref{eq:pillar-cone-def}--\eqref{eq:complement-cone-def}, and the sets $\bF_{\triangledown},\bF_{\parallel}, \bF_{-}$, $\bF_{\curlyvee},\bF_{\ext}$ depending on $x,L,h,\bW$, which were shown to cover every interface $\cI \in \Iso_{h,S}^{(h)}\cap \bI_{\bW}$ in~\eqref{eq:interface-containment-definitions}.

\begin{lemma}\label{lem:swap-map-well-defined}
The map $\Phi_{\swap}$ is a bijection from 
\begin{align}\label{eq:swap-map-domain}
  (\Iso_{h,S}^{(n)}\cap \bI_{\bW} \cap \{\cI:\cP_{x,S} \in A\})\times (\Iso_{h}^{(m)} \cap \{\cI': \cP'_{x'}\in A'\})\,,
  \end{align}
  to
  \begin{align}\label{eq:swap-map-range}
  (\Iso_{h,S}^{(n)}\cap \bI_{\bW}\cap \{\cI:\cP_{x,S}\in A'\}) \times (\Iso_{h}^{(m)} \cap \{\cI':\cP'_{x'}\in A\})\,.
\end{align}
In particular, the map $\Phi_{\swap}$ is a bijection from $$(\Iso_{h,S}^{(n)}\cap \bI_{\bW})\times \Iso_{h}^{(m)} \longrightarrow (\Iso_{h,S}^{(n)} \cap \bI_{\bW}) \times \Iso_{h}^{(m)}\,.$$
\end{lemma}

\begin{proof}
The second claim in the lemma is a special case of the first, so it suffices to prove the former. Consider two interfaces, $\cI \in \Iso_{h,S}^{(n)} \cap \bI_{\bW} \cap \{\cP_{x,S} \in A\}$ and $\cI'\in \Iso_h^{(m)}\cap \{\cP_{x'}\in A'\}$. We begin with the following observation from Claim~\ref{clm:iso-pillar-containments}.  
\begin{enumerate}
    \item Every face $f$ in $\cI\setminus \cP_{x,S}$ having $d(f,x)\le L$ is a ceiling face at height $\hgt(\cC_\bW)$.
    \item Every face $f$ in $\cI'\setminus \cP_{x'}'$ having $d(f,x)\le L$ is a ceiling face at height zero.
\end{enumerate}

It follows from this that if no plus site in $\sigma(\cI\setminus \cP_{x,S})\cap \cL_{>\hgt(\cC_\bW)}$ is $*$-adjacent to a site of 
$$\theta \cP' := \cP'_{x'} -x' + x + (0,0, \hgt(\cC_\bW))\,,$$ 
then $\cJ= \Phi_\swap^1(\cI,\cI')$
is a valid interface with $\cP_{x,S}(\cJ) = \theta\cP'_{x'}$, and with \begin{align}\label{eq:truncated-swaps-equivalent}
    \cJ\setminus \cP_{x,S}(\cJ) = \cI \setminus \cP_{x,S}(\cI)\,, \quad \mbox{and thus} \quad \cJ\restriction_S\setminus \cP_{x,S}(\cJ) = \cI\restriction_S \setminus \cP_{x,S}(\cI)
\end{align} This non-adjacency follows from Claim~\ref{clm:iso-pillar-containments}, along with the observations that $d(\bF_{\curlyvee},\bF_{\triangledown}\cup \bF_{\parallel})\ge L>0$, and $d(\bF_{\ext},\bF_{\triangledown}  \cup \bF_{\parallel})\ge L^3 h/2>0$.  
By analogous reasoning, $\cJ' = \Phi^2_\swap(\cI,\cI')$ is also a well-defined interface, and in turn has $\theta\cP'_{x'}(\cJ') = \cP_{x,S}(\cI)$, and $\cJ'\setminus \cP_{x'}'(\cJ) = \cI' \setminus \cP_{x'}'(\cI)$. 

Let us now see that $\cJ$ is in $\Iso_{h,S}^{(n)} \cap \bI_{\bW} \cap \{\cP_{x,S}\in A'\}$. To see that $\cJ \in \bI_{\bW}$, it suffices for us to show that $\rho(\cP_{x'}')\cap \rho(\bW) = \emptyset$, from which we would deduce that  $\cJ\in \bI_{\bW}$ via~\eqref{eq:truncated-swaps-equivalent}. This follows immediately from the pillar containment of Claim~\ref{clm:iso-pillar-containments} and the fact that $h = o(d(x,S^c))$. It follows that $\cJ$ is in $\Iso_{h,S}^{(n)}$ as its restricted pillar $\cP_{x,S}(\cJ) = \theta\cP'$ which satisfies item~\eqref{item:iso-pillar-increments} in Definition~\ref{def:isolated-pillars}, and the remainder of its interface $\cJ\restriction_S\setminus \cP_{x,S}(\cJ)$ agrees with $\cI\restriction_{S}\setminus\cP_{x,S}(\cI)$, which satisfied item~\eqref{item:iso-pillar-walls} in Definition~\ref{def:isolated-pillars}. Finally, the pillar $\cP_{x,S}(\cJ)$ is in $A'$ as it is a horizontal shift of $\cP'_{x'}\in A'$. 
The fact that $\Phi^2_{\swap}(\cI,\cI')\in \Iso_{h}^{(m)} \cap \{\cP_{x'}\in A\}$ follows from analogous reasoning. 

To conclude that the map is a bijection, observe that $\Phi_{\swap}$ is its own inverse and sends the set of~\eqref{eq:swap-map-domain} back to~\eqref{eq:swap-map-range} by swapping the roles of $A$ and $A'$.
\end{proof}

We next construct a bijection from the pair of face-sets $(\cI,\cI')$ that encodes the action of $\Phi_{\swap}$. Let $\bF_{\triangledown}, \bF_{\parallel}, \bF_{-}, \bF_{\curlyvee}, \bF_{\ext}$ and $\bF'_{\triangledown}, \bF'_{\parallel}, \bF'_{-}, \bF'_{\curlyvee}, \bF'_{\ext}$ be the sets defined by~\eqref{eq:interface-containment-definitions} with respect to $(x, \bW, \cC_\bW)$ and $(x',\emptyset,\cL_{0,m})$ respectively.

\begin{definition}\label{def:swap-face-bijection}
For every $\cI\in \bI_{\bW}\cap \Iso_{h,S}^{(n)}$ and $\cI'\in \Iso_{h}^{(m)}$, we construct the following map $\Theta_{\swap}$ from the pair of face-sets $(\cI,\cI')$, viewed as subsets of $\sF(\Lambda_n)\times \sF(\Lambda_m)$, to pairs of face subsets of $\sF(\Lambda_n)\times \sF(\Lambda_m)$. 
\begin{enumerate}
    \item Map $\cI \cap (\bF_{\triangledown}\cup \bF_{\parallel} \cup \Cyl_{x,L/2})$ to their shift by $-x + x' -(0,0,\hgt(\cC_\bW))$, in the second face subset. 
    \item Map $\cI'\cap (\bF'_{\triangledown}\cup \bF'_{\parallel}\cup \Cyl_{x',L/2})$ to their shift by $-x' + x +(0,0,\hgt(\cC_\bW))$, in the first face subset. 
    \item Identify the remaining faces of $\cI$ with themselves (in the first face subset).
    \item Identify the remaining faces of $\cI'$ with themselves (in the second face subset). 
\end{enumerate}
\end{definition}

\begin{lemma}\label{lem:swap-map-properties}
For every $(\cI,\cI') \in  (\{\cI\restriction_{S}\in \Iso_{h}^{(n)}\}\cap \{\cP_{x,S} \in A\}\cap \bI_{\bW})\times \Iso_{h}^{(m)}$, the map $\Theta_{\swap}$ is a bijection from the sets of faces $(\cI,\cI')$ to the sets of faces $(\Phi_{\swap}^1(\cI,\cI'), \Phi_{\swap}^2(\cI,\cI'))$. In particular, 
\begin{align}\label{eq:swap-map-excess-area}
    |\cI|+ |\cI'| = |\Phi_{\swap}^1(\cI, \cI')| + |\Phi_{\swap}^2(\cI,\cI')|\,.
\end{align}
Moreover, $\Theta_{\swap}$ is such that 
\begin{align}\label{eq:swap-map-face-bijection}
    \br(f,\cI; \Theta_{\swap}f,\cJ) & \ge d(f,\bF_{\triangledown}\cup \bF_{\parallel}) \qquad\mbox{for all }f\in \cI\setminus (\bF_{\triangledown}\cup \bF_{\parallel}\cup \Cyl_{x,L/2})\,,  \\
    \br(f',\cI'; \Theta_{\swap}f',\cJ') & \ge d(f',\bF'_{\triangledown}\cup \bF'_{\parallel}) \qquad \mbox{for all }f'\in \cI'\setminus (\bF_{\triangledown}'\cup \bF_{\parallel}'\cup\Cyl_{x',L/2})\,. \label{eq:swap-map-face-bijection-2} \\ 
    \br(f,\cI; \Theta_{\swap}f,\cJ') & \ge d(f,\bF_{\curlyvee}\cup \bF_{\ext}) \qquad \mbox{for all }f\in \cI\cap (\bF_{\triangledown}\cup \bF_{\parallel}\cup \Cyl_{x,L/2})\,, \label{eq:swap-map-face-bijection-3} \\ 
    \br(f', \Theta_{\swap}f') & \ge d(f',\bF'_{\curlyvee}\cup \bF'_{\ext}) \qquad \mbox{for all }f'\in \cI'\cap (\bF'_{\triangledown}\cup \bF'_{\parallel}\cup \Cyl_{x',L/2})\,. \label{eq:swap-map-face-bijection-4}
    \end{align}
\end{lemma}
\begin{proof}
Evidently the existence of such a bijection implies~\eqref{eq:swap-map-excess-area}, so it suffices to show that $\Theta_{\swap}$ is a bijection. 

Recall from Clam~\ref{clm:iso-pillar-containments} that the pillars $\cP_{x,S}$ and $\cP_{x'}'$ are contained in the respectively defined $\bF_{\triangledown} \cup \bF_{\parallel}$ and $\bF_{\triangledown}', \bF_{\parallel}'$. The definition of the map $\Phi_{\swap}$ and the fact that the intersection of $\cI\setminus \cP_{x,S}$ with $\Cyl_{x,L/2}$ is just a shift of the intersection of $\cI'\setminus \cP'_{x'}$ with $\Cyl_{x',L/2}$ by $x-x' + (0,0,\hgt(\cC_\bW))$ implies that this map sends the faces in $\cI$ and  the faces in $\cI'$ to the faces in $\cJ$ and the faces in $\cJ'$. The fact that $\Theta_\swap$ is a bijection is evident as it is its own inverse. 

To deduce the bounds~\eqref{eq:swap-map-face-bijection}--\eqref{eq:swap-map-face-bijection-4} on the radii of congruence, first consider $f\in \cI\setminus (\bF_{\triangledown}\cup \bF_{\parallel}\cup \Cyl_{x,L/2})$. Observe that for such $f$, $\Theta_{\swap}f = f$. By construction of the bijection and the inclusion of Claim~\ref{clm:iso-pillar-containments}, the radius $\br(f,\cI; \Theta_{\swap} f,\cJ)$ is at least $d(f,\bF_{\triangledown}\cup \bF_{\parallel}\cup \Cyl_{x,L/2})$. However, by definition and Claim~\ref{clm:iso-pillar-containments}, the only faces of $\cI$ in $\Cyl_{x,L/2}\setminus (\bF_{\triangledown}\cup\bF_{\parallel})$ are exactly those of $\cJ$ in $\Cyl_{x,L/2}\setminus (\bF_{\triangledown}\cup\bF_{\parallel})$; thus the radius $\br(f,\cI; \Theta_{\swap} f,\cJ)$ must be attained by a face in $\bF_{\triangledown}\cup \bF_{\parallel}$, implying~\eqref{eq:swap-map-face-bijection}. The bound~\eqref{eq:swap-map-face-bijection-2} is deduced analogously. 

Next, consider $f\in \cI \cap (\bF_{\triangledown}\cup \bF_{\parallel}\cup \Cyl_{x,L/2})$. It must be the case that the radius $\br(f,\cI; \Theta_{\swap}f, \cJ')$ is either attained by a face of $\cI\setminus (\bF_{\triangledown}\cup \bF_{\parallel} \cup \Cyl_{x,L/2})$, or by a face of $\cJ'\setminus (\bF_{\triangledown}' \cup\bF_{\parallel}'\cup \Cyl_{x',L/2})$. By Claim~\ref{clm:iso-pillar-containments}, these sets are contained in $\bF_{\curlyvee}\cup \bF_{\ext}\cup (\Cyl_{x,L}\setminus \Cyl_{x,L/2})$, or $\bF'_{\curlyvee}\cup\bF'_{\ext}\cup (\Cyl_{x',L}\setminus \Cyl_{x',L/2})$ respectively.  Notice that the faces of $\cI \setminus (\bF_{\triangledown}\cup \bF_{\parallel})$ in $\Cyl_{x,L}$ are exactly those of $\Cyl_{x,L}\cap \cL_{\hgt(\cC_\bW)}$, and the faces of $\cJ'\setminus (\bF'_{\triangledown}\cup \bF'_{\parallel})$ in $\mathsf{Cyl}_{x',L}$ are exactly the shift of that set by $-x + x'-(0,0,\hgt(\cC_\bW))$, which equals the shift of $f$ to $\Theta_{\swap}f$. Thus, we deduce that the radius is either attained from $f$ by a face in $\bF_{\curlyvee} \cup \bF_{\ext}$, or from $\Theta_{\swap} f$ by a face in $\bF'_\curlyvee \cup \bF'_{\ext}$. This implies~\eqref{eq:swap-map-face-bijection-3}, and~\eqref{eq:swap-map-face-bijection-4} follows analogously.  
 \end{proof}

\subsection{Analysis of \texorpdfstring{$\Phi_{\swap}$}{Phi\_swap}}
The key to proving Theorem~\ref{thm:cond-uncond} will be the following bound on the weight distortion of pairs of interfaces under the application of the map $\Phi_{\swap}$. By~\eqref{eq:swap-map-excess-area}, the excess area of the map is zero, and our goal is to show that the ratio of weights of the pair $(\cI,\cI')$ to $(\cJ,\cJ')$ is almost one.  

\begin{proposition}\label{prop:swap-map}
Fix $L$ in the definition of $\Iso_{L,h}$ sufficiently large. For all $\cI\in \bI_{\bW}\cap \Iso_{h,S}^{(n)}$ and $\cI'\in \Iso_{h}^{(m)}$,  
\begin{align*}
     \Big|\frac{\mu_n^\mp(\cI)\mu_m^\mp(\cI')}{\mu_n^\mp(\Phi^1_{\swap}(\cI,\cI'))\mu_m^\mp(\Phi^2_{\swap}(\cI,\cI'))} - 1\Big| \le \bar K e^{ - L/C}\,.
\end{align*}
\end{proposition}

\begin{proof}
As before, let $\cJ = \Phi^1_{\swap}(\cI,\cI')$ and $\cJ' = \Phi^2_{\swap}(\cI,\cI')$. We express the ratio on the left-hand side using Theorem~\ref{thm:cluster-expansion}. By~\eqref{eq:swap-map-excess-area}, 
$$\fm(\cI;\cJ) + \fm(\cI';\cJ')  = 0\,,$$
so that the ratio can be written as 
\begin{align}\label{eq:swap-map-g-terms}
    \exp \bigg[ \sum_{f\in \cI}\g(f;\cI) + \sum_{f'\in \cI'} \g(f';\cI) - \sum_{f\in \cJ} \g(f;\cJ) - \sum_{f'\in \cJ'} \g(f';\cJ')\bigg]\,.
\end{align}
We rearrange the summands according to the bijection $\Theta_{\swap}$ of Definition~\ref{def:swap-face-bijection}, sending the faces in $(\cI,\cI')$ to those in $(\cJ,\cJ')$. Let 
\begin{align*}
    \mathbf P:= \cI \cap (\bF_{\triangledown}\cup \bF_{\parallel} \cup \Cyl_{x,L/2})\,,\qquad \mbox{and similarly,} \qquad \mathbf P' := \cI'\cap (\bF'_{\triangledown}\cup \bF'_{\parallel} \cup \Cyl_{x',L/2})\,.
\end{align*}
We can then rearrange~\eqref{eq:swap-map-g-terms} to obtain 
\begin{align}\label{eq:swap-map-splitting}
   \bigg|\log \frac{\mu_n^\mp(\cI)\mu_m^\mp(\cI')}{\mu_n^\mp(\cJ)\mu_m^\mp(\cJ')}\bigg| &    \le \sum_{f\in \bP} \big|\g(f;\cI) - \g(\Theta_{\swap}f;\cJ')\big| + \sum_{f'\in \bP'} \big|\g(f',\cI')  - \g(\Theta_{\swap}f,\cJ)\big| \\ 
   &  \qquad +  \sum_{f\in \cI\setminus \bP} \big|\g(f;\cI) - \g(\Theta_{\swap}f;\cJ)\big| + \sum_{f'\in \cI'\setminus \bP'}\big| \g(f';\cI') - \g(\Theta_{\swap}f',\cJ')\big|\,. \nonumber
\end{align}
We consider the summands above individually.
By~\eqref{eq:g-exponential-decay} and~\eqref{eq:swap-map-face-bijection-3}, the first term is bounded as 
\begin{align*}
     \sum_{f \in \bP} \big|\g(f;\cI) - \g(\Theta_{\swap} f;\cJ')\big| \le \sum_{f\in \bP} \bar K e^{ - \bar c \br(f,\cI; \Theta_{\swap} f,\cJ')} \le \sum_{f\in \bP} K e^{ - \bar c d(f,\bF_{\curlyvee}\cup \Cyl_{x,L^3 h}^c)}\,.
\end{align*}
Let us now consider this latter sum. We can bound it by 
\begin{align}\label{eq:need-to-bound-1}
   \sum_{f\in \bF_{\parallel}\cup \Cyl_{x,L/2}}  K e^{ - \bar c d(f,\Cyl_{x,L}^c)} + \sum_{f\in \bF_{\triangledown}}  K e^{-  \bar c d(f,\bF_{\curlyvee})}+ \sum_{f\in \bF_{\triangledown}} K e^{ - \bar c d(f,\bF_{\ext})}    \,.
\end{align}
The first sum in~\eqref{eq:need-to-bound-1} can be bounded as 
\begin{align}\label{eq:ntb-1-first-term}
     \sum_{f\in \bF_{\parallel}\cup \Cyl_{x,L/2}}  K e^{ - \bar c d(f,\Cyl_{x,L}^c)}\le K |\bF_{\parallel} \cup \Cyl_{x,L/2}| e^{ - \bar c d(\bF_{\parallel}\cup \Cyl_{x,L/2},\Cyl_{x,L}^c)}\le K (4 L^3 + L^2) e^{ - \bar c L/2}\,.
\end{align}
The second and third terms in~\eqref{eq:need-to-bound-1} are together at most $\bar C e^{ - \bar c L}$ as long as $L$ is large, by Claim~\ref{clm:iso-pillar-containments}. Combining the three terms above, we easily find that as long as $L$ is large,~\eqref{eq:need-to-bound-1} is at most $K e^{ -  L/C}$.  

Now consider the third term from~\eqref{eq:swap-map-splitting}, for which by~\eqref{eq:g-exponential-decay} we have
\begin{align*}
     \sum_{f\in \cI\setminus \bP} \big|\g(f;\cI) - \g(f;\cJ)\big| \le \sum_{f\in \cI \setminus \bP} \bar K e^{ - \bar c \br(f, \cI; f, \cJ)}\,.
\end{align*}
By~\eqref{eq:swap-map-face-bijection}, this is at most 
\begin{align*}
    \sum_{f\in \cI\setminus \bP} \sum_{g\in \bF_{\triangledown}\cup \bF_{\parallel}} e^{ - \bar c d(f,g)} &  \le \sum_{f\in  \Cyl_{x,L/2}^c} \sum_{g\in \bF_{\parallel}} e^{ - \bar c d(f,g)} + \sum_{f\in \bF_{\curlyvee}\cup \bF_{\ext}} \sum_{g\in \bF_\triangledown} e^{ - \bar c d(f,g)} + \sum_{f\in \cI \cap \cL_{\hgt(\cC_\bW)}}\sum_{g\in \bF_{\triangledown}} e^{ - \bar c d(f,g)}\,.
\end{align*}
The first term here, is at most $4 K L^3 e^{ - \bar c L/2}$, by summing out in $f$. The second term is, after summing out in $f$, bounded by $K e^{ - L/C}$ by~\eqref{eq:ntb-1-third-term}--\eqref{eq:ntb-1-second-term}. The third term can be bounded as follows: 
\begin{align*}
    \sum_{f\in \cI \cap \cL_{\hgt(\cC_\bW)}}\sum_{g\in \bF_{\triangledown}} e^{ - \bar c d(f,g)}\le \sum_{k\ge 1} \sum_{g\in \bF_{\triangledown}: g\cdot e_3 = k} e^{ -\bar c k} \le \sum_{L^3 \le k \le 10 h} k^4 e^{ - \bar c k}\,,
\end{align*}
which is at most some $e^{ - \bar c L^3 /2}$ for $L$ large.  The second and fourth terms in~\eqref{eq:swap-map-splitting} are evidently bounded symmetrically.  Altogether, 
 we obtain the desired for some other $C$, using $|e^{x}-1|\le 2x$ for $x\le 1$. 
\end{proof}

\subsection{Proof of Theorem~\ref{thm:cond-uncond}}\label{subsec:swap-main-proof}
For ease of notation, let
    $\nu_n^\mp  = \mu_n^\mp ( \cdot \mid \bI_{\bW})$. 
Recall from Lemma~\ref{lem:swap-map-well-defined}, that for every $\cI\in \bI_{\bW}\cap \Iso_{h,S}^{(n)}$,and every $\cI'\in \Iso_h^{(m)}$, we have $\Phi^1_\swap(\cI,\cI') \in \bI_{\bW}$. Therefore, 
dividing and multiplying the ratio in Proposition~\ref{prop:swap-map} by $\mu_n^\mp(\bI_{\bW})$, we obtain the following: for every $\cI \in \bI_{\bW} \cap \Iso_{h,S}^{(n)}$, and every $\cI'\in \Iso_h^{(m)}$, for $L$ large,
\begin{align*}
         \Big|\frac{\nu_n^\mp(\cI)\mu_m^\mp(\cI')}{\nu_n^\mp(\Phi^1_{\swap}(\cI,\cI'))\mu_m^\mp(\Phi^2_{\swap}(\cI,\cI'))} - 1\Big| \le K e^{ - c L}\,.
\end{align*}

It will suffice for us to show that for every set of pillars $A$, we have 
\begin{align}\label{eq:wts-pillar-decorrelation-1}
    \nu_n^\mp (\Iso_{h,S}^{(n)}, \cP_{x,S} \in A)\le (1+\epsilon_\beta)\mu_{m}^\mp(\Iso_{h}^{(m)}, \cP_{x'}\in A)\,,
    \end{align}
    and
    \begin{align}\label{eq:wts-pillar-decorrelation-2}
\mu_{m}^\mp(\Iso_h^{(m)}, \cP_{x'}\in A) \le (1+\epsilon_\beta) \nu_n^\mp (\Iso_{h,S}^{(n)}, \cP_{x,S} \in A)\,,
\end{align}
and then let $x_m' = o$ identically and take a limit as $m\to \infty$ for $h=h_n$ fixed w.r.t.\ $m$. 
We will prove the first of these inequalities, as the proof of the second is mutatis mutandis. Multiplying and dividing the left-hand side of~\eqref{eq:wts-pillar-decorrelation-1} by $\mu_m^\mp(\cI \in \Iso_h^{(m)})$, we get by Proposition~\ref{prop:swap-map},  
\begin{align*}
    \frac{1}{\mu_m^\mp(\cI \in \Iso_h^{(m)})} & \sum_{\cI\in \bI_{\bW}\cap \Iso_{h,S}^{(n)}, \cP_{x,S} \in A}  \,\sum_{\cI' \in \Iso_h^{(m)}}\nu_n^\mp(\cI)  \mu_m^\mp (\cI') \\
    & \le \frac{1+K e^{ - c L}}{\mu_m^\mp(\cI \in \Iso_h^{(m)})} \sum_{\cI\in \bI_{\bW}\cap \Iso_{h,S}^{(n)}, \cP_{x,S} \in A} \,\sum_{\cI' \in \Iso_h^{(m)}} \nu_n^\mp(\Phi^1_{\swap}(\cI,\cI')) \mu_m^\mp (\Phi^2_{\swap}(\cI,\cI'))\,.
\end{align*}
Now recall from Lemma~\ref{lem:swap-map-properties} that the map $\Phi_{\swap}$ is a bijection from 
\begin{align*}
    \{\cI\in \Iso_{h,S}^{(n)} \cap \bI_{\bW}: \cP_{x,S_n} \in A\} \times \{\cI'\in \Iso_h^{(m)}\}\,,\qquad \mbox{to}\qquad  \{\cI\in \Iso_{h,S}^{(n)}\cap \bI_{\bW}\}\times \{\cI'\in \Iso_h^{(m)}: \cP_o'\in A\}\,.
\end{align*} 
We can therefore rewrite the right-hand side of the above inequality as 
\begin{align*}
\frac{1+ K e^{ - c L}}{\mu_m^\mp (\Iso_h^{(m)})} \sum_{\cJ\in \bI_{\bW} \cap \Iso_{h,S_n}^{(n)}} &  \sum_{\cJ'\in \Iso_h^{(m)}: \cP_o'\in A} \nu_n^\mp(\cJ) \mu_m^\mp(\cJ') \\
& \le \frac{(1+ K e^{ - c L})\nu_n^\mp(\Iso_h^{(n)})}{\mu_m^\mp(\Iso_h^{(m)})} \mu_m^\mp(\cJ' \in \Iso_h^{(m)}, \cP'_o\in A)\,.
\end{align*}
Recall from Theorem~\ref{thm:shape-inside-ceiling}, that there exists a sequence $L_\beta\uparrow \infty$ and $\epsilon_\beta\downarrow 0$  such that $\mu_m^\mp(\Iso_h^{(m)})\ge 1-\epsilon_\beta$. With that choice of $L_\beta$, we see that $K e^{ - c L} \le \epsilon'_\beta$, for some other sequence $\epsilon'_\beta \downarrow 0$. As such, for a new $\epsilon_\beta \downarrow 0$, the right-hand side above is at most $(1+\epsilon_\beta)\mu_m^\mp (\cJ'\in \Iso_h^{(m)}, \cP_o'\in A)$ as desired in~\eqref{eq:wts-pillar-decorrelation-1}.  

Taking $m$ to infinity in the inequalities~\eqref{eq:wts-pillar-decorrelation-1}--\eqref{eq:wts-pillar-decorrelation-2} implies the desired. \qed

\subsection{Cruder bound on nearby pillar correlations}\label{subsec:crude-pillar-correlations}
A consequence of Theorem~\ref{thm:cond-uncond-simple} is that the probability of $\cP_{x,S}$ reaching height $h$ decays like $e^{ - \alpha_h}$ even conditionally on $\bW$. For that, however, it is important that the walls $\bW$, are at distance much greater than $h$ from $x$. 
We conclude this section with a crude but simple bound that allows us to control correlations of pillars to nearby walls (distances that are $O(h)$, even $O(1)$), and as such are not covered by  Corollary~\ref{cor:cond-uncond-k-simple}.

Of course two pillars are not independent if they are the same, by sharing a base, and therefore we restrict attention to pillars with an empty base (which Section~\ref{sec:tall-pillar-shape} showed to happen with probability $1-\epsilon_\beta$), even conditionally on $\bI_\bW$. For ease of notation, for $h\ge 1$, henceforth let 
$$\overline E_{x,S}^h := \{\hgt(\cP_{x,S})\ge h\}\cap \{\sB_{x,S} \ne \emptyset\}\,.$$

\begin{proposition}\label{prop:two-pillars}
For every $S\subset \cL_{0,n}$ every $\bW =(W_z)_{z\notin S}$ such that $\rho(\bW)\subset S^c$, and every distinct $x,y\in S_n$, for every $h$, 
\[ \mu_{n}^\mp\left(\overline E_{x,S}^h \cap \overline E_{y,S}^h\mid \bI_{\bW}\right) \leq e^{ - (4\beta - C)h} \mu_n^{\mp}(\overline E_{x,S}^{h}\mid \bI_{\bW})\,.\]
\end{proposition}

Before proceeding to the proof, we comment briefly on the relation of Proposition~\ref{prop:two-pillars} to the analogous Claim 6.3 of~\cite{GL19b}, where we proved a similar decorrelation estimate between nearby (but distinct, due to the empty base criterion) pillars. 
That argument relied crucially on the FKG inequality; conditionally on $\bI_\bW$ (a non-monotone and potentially exponentially unlikely set), we cannot apply the FKG inequality and must prove this by hand, resulting in the loss in the rate $\alpha_h \to (4\beta - C)h$.
Because we only require this coarser bound, the proof goes via a straightforward map that deletes the entire pillar $\cP_{y,S}$. 

\begin{definition}\label{def:two-pillars-map}
Let $\Psi_{y,S}$ be the following map on interfaces $\cI\in \bI_{\bW}\cap \overline E_{x,S}^{h} \cap \overline E_{y,S}^h$. From the interface $\cI$, delete all bounding faces of the pillar $\cP_{y,S}$, and extract from that the resulting interface $\Psi_{y,S}(\cI) = \cI\setminus \cP_{y,S}$ (possibly by adding back the horizontal face at $y + (0,0,\hgt(\cC_\bW))$).
\end{definition}

\begin{lemma}\label{lem:two-pillars-properties}
The map $\Psi_{y,S}$ is well defined and if $\cI\in \bI_\bW$ and $\cI \in \overline E_{x,S}^h \cap \overline E_{y,S}^h$, it satisfies the following: $\cJ:= \Psi_{y,S}(\cI)$ is in $\bI_{\bW}\cap \overline E_{x,S}^h$, and 
\begin{align*}
    \fm(\cI; \cJ) \ge |\sF(\cP_{y,S_n})|- 1 \ge |\cI \oplus \cJ| - 2 \ge 4h -1\,. 
\end{align*}
\end{lemma}

\begin{proof}
That $\Psi_{y,S}$ is well-defined follows from the observation that the vertical faces of $\cI$ determine the entirety of $\cI$, and the map removes all vertical bounding faces of a $*$-connected component of $+$-spins above $\hgt(\cC_\bW)$. Indeed this forms a (maximal) connected component of the vertical wall faces of $W_y$; since $\rho(W_y)\cap \rho (\bW)=\emptyset$, this leaves $\bW$ unchanged in the standard wall representation of $\cJ$, and since $\sB_{y,S} = \emptyset$, $\cP_{y,S} \cap \cP_{x,S} = \emptyset$, so that the latter pillar is unchanged in $\cJ$, and $\overline E_{x,S}^h$ is still satisfied. 

In particular, the set $\cI\oplus \cJ$ consists of the bounding faces of $\cP_{y,S}$, possibly together with a single horizontal face at $y + (0,0,\hgt(\cC_\bW))$ to ``fill in" the interface from the vertical faces. This is at most a single horizontal face because $\sB_{y,S}  =\emptyset$. With that, we find that $|\cI \oplus\cJ| \le |\sF(\cP_{y,S})|+ 1$. Since $\hgt(\cP_{y,S})\ge h$, it must be the case that $|\sF(\cP_{y,S})|\ge 4h$. Finally, since at most one face of $\cI \oplus \cJ$ belongs to $\cJ$, clearly $\fm(\cI ;\cJ) = |\cI|-|\cJ|\ge |\cI \oplus \cJ|- 2$. 
\end{proof}

We next turn to analyzing the map $\Psi_{y,S}$ of Definition~\ref{def:two-pillars-map}; this goes via the usual analysis of the change of weights under $\Psi_{y,S}$ and the multiplicity of the map, presented in the following two lemmas. 

\begin{lemma}\label{lem:two-pillars-weights}
There exists $C>0$ such that for all $\beta>\beta_0$, for every $\cI \in \bI_{\bW} \cap \overline E_{x,S}^h \cap \overline E_{y,S}^h$, we have 
\begin{align*}
    \Big|\log\frac{\mu_n^\mp(\cI)}{\mu_n^\mp(\Psi_{y,S}(\cI))} + \beta \fm(\cI; \Psi_{y,S}(\cI))\Big| \le C\fm(\cI; \Psi_{y,S}(\cI))\,.
\end{align*}
\end{lemma}

\begin{lemma}\label{lem:two-pillars-multiplicity}
There exists $C>0$ such that for every $\cJ$ in the range $\Psi_{y,S}(\{\cI\in \bI_\bW \cap \overline E_{x,S}^{h} \cap \overline E_{y,S}^h\})$, for every $k\ge 0$, 
\begin{align*}
    |\{\cI\in \overline E_{x,S}^h\cap \overline E_{y,S}^h: \Psi_{y,S}^{-1}(\cI)= \cJ\,,\, \fm(\cI; \cJ) = k\}| \le C^k\,.  
\end{align*}
\end{lemma}

\begin{proof}[\emph{\textbf{Proof of Lemma~\ref{lem:two-pillars-weights}}}]
Since the excess area of the map is comparable to the number of faces in $\cI \oplus \Psi_{y,S}(\cI)$, the analysis of the map is quite straightforward and does not require the wall/ceiling machinery of~\cite{Dobrushin68}. Fix an $\cI\in \overline E_{x,S}^h \cap \overline E_{y,S}^h$ having $\cI\in \bI_{\bW}$, and let $\cJ = \Psi_{y, S}(\cI)$. 

Denote by $\bP$ the faces of $\cI \oplus \cJ$, which we recall are exactly the bounding faces of $\cP_{y,S}$ that were deleted by $\Psi_{y,S}$ together with the at most one horizontal face that was added at $y+ (0,0,\hgt(\cC_\bW))$. By Theorem~\ref{thm:cluster-expansion}, we can expand the left-hand side of the lemma as 
\begin{align*}
     \Big|\log\frac{\mu_n^\mp(\cI)}{\mu_n^\mp(\Psi_{y,S}(\cI))} + \beta \fm(\cI; \Psi_{y,S}(\cI))\Big| \le \sum_{f\in \bP\cap \cI} |\g(f;\cI)| + \sum_{f\in \bP \cap \cJ} |\g(f;\cJ)| + \sum_{\cI \setminus \bP} |\g(f;\cI)- \g(f;\cJ)|\,.
\end{align*}
By~\eqref{eq:g-uniform-bound}, the first two terms above are at most $\bar K |\cI\oplus \cJ|\le 3 \bar K \fm(\cI;\cJ)$ by Lemma~\ref{lem:two-pillars-properties}. Turning to the last term, by~\eqref{eq:g-exponential-decay}, 
\begin{align*}
    \sum_{\cI\setminus \bP} |\g(f;\cI)  - \g(f;\cJ)| \le \sum_{\cI\setminus \bP} \bar Ke^{ - \bar c\br(f,\cI; f,\cJ)} \le \sum_{f\in \cI \setminus \bP}  \sum_{g\in \bP} \bar K e^{ - \bar c d(f,g)}\le  3\bar K^2 \fm(\cI,\cJ)\,.
\end{align*}
Combining the above bounds yields the desired.
\end{proof}

\begin{proof}[\emph{\textbf{Proof of Lemma~\ref{lem:two-pillars-multiplicity}}}]
Fix any $\cJ$ in the range of $\Psi_{y,S}$ applied to the set of interfaces in $\bI_{\bW} \cap \overline E_{x,S}^h\cap \overline E_{y,S}^h$. As noted in Lemma~\ref{lem:two-pillars-properties}, any such $\cJ$ has $\cP_{y,S}(\cJ) = \emptyset$ and from that interface, together with the set of faces $\cI \oplus \cJ$, we can reconstruct the interface $\cI$. The face-set $\cI \oplus \cJ$ is rooted at $y+ (0,0,\hgt(\cC_\bW))$ (the height $\hgt(\cC_\bW)$ can be read off from $\cJ$), and therefore, it suffices to enumerate over the number of connected face-sets of size $|\cI \oplus\cJ| \le 3\fm(\cI;\cJ)= 3k$ rooted at a fixed face: by Fact~\ref{fact:number-of-walls}, the total number of such rooted connected is at most $s^{3k}$ for some universal $s>0$, concluding the proof. 
\end{proof}

We are now in position to combine Lemmas~\ref{lem:two-pillars-weights}--\ref{lem:two-pillars-multiplicity}, to deduce Proposition~\ref{prop:two-pillars}.

\begin{proof}[\textbf{\emph{Proof of Proposition~\ref{prop:two-pillars}}}]
We can express the probability $\mu_{n}^{\mp}(\overline E_{x,S}^h \cap \overline E_{y,S}^h \cap \bI_{\bW})$ as follows:
\begin{align*}
    \sum_{\cI \in \overline E_{x,S}^h\cap \overline E_{y,S}^h\cap \bI_{\bW}} \mu_n^\mp(\cI) & = \sum_{\cJ \in \Psi_{y,S}(\overline E_{x,S}^h\cap \overline E_{y,S}^h\cap \bI_{\bW})} \sum_{k\ge 4h-1} \sum_{\substack{\cI \in \Psi_{y,S}^{-1}(\cJ) \\ \fm(\cI;\cJ) = k}}\mu_n^{\mp}(\cI) \\
    & \leq \sum_{\cJ \in \Psi_{y,S}(\overline E_{x,S}^h\cap \overline E_{y,S}^h\cap \bI_{\bW})} \sum_{k\ge 4h-1} C^{k} e^{- (\beta - C)k}\mu_n^{\mp}(\cJ)\,,
\end{align*}
where the inequality used Lemmas~\ref{lem:two-pillars-weights}--\ref{lem:two-pillars-multiplicity}. By Lemma~\ref{lem:two-pillars-properties}, we have $\Psi_{y,S}(\overline E_{x,S}^h \cap \overline E_{y,S}^h\cap \bI_{\bW})\subset \overline E_{x,S}^h \cap \bI_{\bW}$, and therefore the right-hand above is at most 
\begin{align*}
    \sum_{\cJ \in \overline E_{x,S}^h\cap \bI_{\bW}} C e^{ - 4(\beta - C)h} \mu_n^{\mp}(\cJ)\le Ce^{ - 4(\beta - C)h}\mu_n^{\mp} (\overline E_{x,S}^h\cap \bI_{\bW})\,.
\end{align*}
Dividing out both sides by $\mu_n^{\mp}(\bI_{\bW})$ then yields the desired bound of the proposition.
\end{proof}

\section{Maximal height fluctuations within a ceiling}\label{sec:fluctuations-in-ceiling}
Our goal in this section is to bound, conditionally on the walls $\bW$ outside a given set $S$ giving rise to a ceiling~$\cC_\bW$ of $\cI_\bW$ with $S \subset \rho(\hull\cC_\bW)$, the tails of the shifted maximum height $\bar M_{S}$ of the interface $\cI$ above~$S$. Recall that $\bar M_S = M_S - \hgt(\cC_\bW)$ is given by the maximum height of $\cI \restriction_S$. 
When analyzing the event $\bar M_{S}\geq h$, the cost of a localized pillar $\cP_{x,S}$ in the interface reaching height $h$ will be dominated, for the bulk of the index points $x\in S$, by the large deviation rate $\alpha_h$. To control the effect of sites near $\partial S$, as well as situations where distant pillars interact, we introduce the following events for any $A\subset S$:
\begin{equation*}%\label{eq:G-event} 
\cG^{\fm}_A(r)= \bigcap_{x\in A}\left\{ \fm(\fW_{x,S})< r \right\}\qquad,\qquad 
\cG^{\mathfrak{D}}_A(r)= \bigcap_{x\in A}\left\{ \diam(\fW_{x,S})< r \right\}
\,.
\end{equation*}
First, an application of the conditional rigidity estimates of~\S\ref{sec:wall-clusters} will yield the next result for fluctuations of the recentered maximum within arbitrarily shaped $\hgt(\cC_\bW)$-level curves of the interface.

\begin{proposition}\label{prop:max-generic}
There exist $\beta_0,C>0$ so the following holds for all $\beta>\beta_0$. Let $S_n\subset\cL_{0,n}$ be a sequence of simply-connected sets such that $|S_n|\to\infty$ with $n$, and let $\bW_n = \{ W_z : z\notin S_n\}$ be such that $\rho(\bW) \subset S_n^c$. Then for every $h = h_n \ge 1$,
\begin{align} 
\mu^\mp_{n}\left( \bar M_{S_n}< h \given \bI_{\bW_n}\right) \geq 
\mu^\mp_{n}\left( \cG^{\fm}_{S_n}(4h) \given \bI_{\bW_n}\right) &\geq \exp\big(- |S_n| e^{-(4\beta -C)h}\big)\,, \label{eq:max-lower-crude}
\end{align}
and
\begin{align}\label{eq:max-upper-crude}
\mu^\mp_{n}\left(\bar M_{S_n} < h \mid \bI_{\bW_n}\right) &\leq  \exp\left(- |S_n| e^{-(4\beta+C)h}\right) \,.
\end{align}
\end{proposition}

For sets $S_n$ with a uniformly bounded isoperimetric dimension and $h$ that is at least poly-logarithmic (concretely, say, $h\geq \sqrt{\log |S_n|}$), we can apply the results of Theorem~\ref{thm:cond-uncond-simple} and obtain
 much finer estimates.

\begin{proposition}\label{prop:max-thick} 
There exist $\beta_0,\kappa_0>0$ such that the following holds for every fixed $\beta>\beta_0$. Let $S_n\subset\cL_{0,n}$ be a sequence of simply-connected sets such that $\isodim(S_n)\leq \sqrt \beta$ and 
 $\lim_{n\to\infty}|S_n|=\infty$. Let $\bW_n = (W_z)_{z\notin S_n}$ be such that $\rho(\bW_n) \subset S_n^c$.
For every $ \sqrt{\log |S_n| } \leq h \leq \frac1{\sqrt\beta}\log |S_n|$,
\begin{align}\label{eq:max-thick-lower-large-h} 
\mu^\mp_{n}\left( \bar M_{S_n}< h \,,\,\cG^\fm_{S_n\setminus S_{n,h}^\circ}
(4h)\,,\, \cG^{\fm}_{S^\circ_{n,h}}
(5h) \mid \bI_{\bW_n}\right) &\geq \exp\left(
- (1+\epsilon_\beta)|S_n| e^{-\alpha_h}\right)\,,
\end{align}
and
\begin{align}\label{eq:max-thick-upper-large-h}
\mu^\mp_{n}\left(\bar M_{S_n} < h\,,\,\cG^{\mathfrak D}_{S_{n,h}^\circ}(e^{\kappa_0 h})  \mid \bI_{\bW_n}\right) &\leq  \exp\left(-(1-\epsilon_\beta) |S_n| e^{-\alpha_h}\right)\,,
\end{align}
where $S_{n,h}^\circ := \{x\in S_n : d(x,\partial S_n) \geq e^{2\kappa_0 h}\}$. In particular, for some absolute constant $C>0$, \begin{align}\label{eq:max-thick-combined-large-h}
\exp(- (1+\epsilon_\beta) |S_n| e^{-\alpha_h}) \leq  \mu^\mp_{n}\left( \bar M_{S_n}< h  \mid \bI_{\bW_n}\right) \leq 
\exp(
- (1-\epsilon_\beta) |S_n| e^{-\alpha_h})+\exp(-(\beta-C)e^{\kappa_0 h})
\,.\end{align}	
\end{proposition}

\subsection{Maximum within general ceilings}\label{subsec:max-generic-ceil}
We begin with the proof of Proposition~\ref{prop:max-generic} which only requires the conditional rigidity estimates of Section~\ref{sec:wall-clusters}, as the exponential rates in the exponents of~\eqref{eq:max-lower-crude}--\eqref{eq:max-upper-crude} differ by a constant $C$. 
While not used for Theorem~\ref{thm:gumbel}, we include this as may be of use in settings where one has no control on the geometry of the set $S$, and its proof serves as a warm-up for that of Proposition~\ref{prop:max-thick}. 

\medskip 
\noindent \emph{Proof strategy.} For the lower bound on the probability in~\eqref{eq:max-lower-crude}, we would like to show that the probability of $\bar M_{S_n}<h$ is at least the product of probabilities of $\hgt(\cP_{x,S_n})\le h$ over $x\in S_n$. For height functions (e.g., SOS and DG model~\cite{CLMST14,LMS16}), this comparison follows from the FKG inequality, but our conditioning on $\bI_{\bW_n}$ unfortunately destroys the FKG property of the Ising model. Instead, we prove this kind of bound via a careful revealing procedure of nested collections of walls, one at a time, and iteratively applying Corollary~\ref{cor:nested-sequence-inside-a-ceiling}. 

For the upper bound on the probability in~\eqref{eq:max-upper-crude}, we would like to compare $\one\{\bar M_{S_n}<h\}$ to the maximum of $S_n$ many independent $\ber(e^{ - (4\beta +C)h})$ random variables. By Claim~\ref{clm:lower-bound-forcing} the exponential rate $(4\beta + C)$ we are aiming for is attained by a vertical column of height $h$ uniformly over its environment (cf.\ the sharper bound of Proposition~\ref{prop:max-thick}). Thus, we obtain~\eqref{eq:max-upper-crude} by considering the possible insertion of pillars consisting of straight columns of height $h$ at some mesh of $|S_n|/4$ faces in $S_n$.

\begin{proof}[\textbf{\emph{Proof of Proposition~\ref{prop:max-generic}}}]
Throughout the proof, we will write $S=S_n$ with $s_n = |S_n|$, and $\bW=\bW_n$, and $h= h_n$, for brevity.

\smallskip
\noindent
\emph{Proof of~\eqref{eq:max-lower-crude} (lower bound on the probability of interfaces with $\bar M_{S_n} < h$).}  Define, for $x\in S$ and $A\subset S$, the events
\begin{align}\label{eq-G-H-events} G_{x,A} = \{ \fm(\fW_{x,A})< 4h \}\,,\quad
H_{x,A} = \{ \hgt(\cP_{x,A}) < h \}\,.\end{align}
Denote by $\fG_{x,A}$ the walls nesting $x$ within $A$, as well as every wall nested in them; i.e.,
\begin{equation}\label{eq:fG-def}
\fG_{x,A} = \fW_{x,A} \cup \left\{ W' : W' \Subset W\mbox{ for some }W\in\fW_{x,A}\right\} \,,
\end{equation}
and let 
\begin{equation}\label{eq:hat-Gx-def}\widehat G_{x,A} = \bigcap\{G_{u,A} \,:\; u\in \rho(\hull\fG_{x,A})\}\,.\end{equation}
We will mostly consider $A=S$, in which case we omit this subscript and simply use $G_x$ ,$H_x$, $\fG_{x}$ and $\widehat G_x$.

Noting $\cG^{\fm}_{S}(4h)=\bigcap_{x\in S} \widehat G_x$ and $\{\bar M_S < h\}=\bigcap_{x\in S}H_x$, we claim  $\widehat G_x \subset H_x$, so that $\cG^\fm_S(4h)\subset \{\bar M_S<h\}$. Indeed, under $H_x^c$, there must exist some $y$ such that $W_y \Subset W$ for some $W\in \fW_{x,S}$ and such that the nested sequence of walls $\fW_{y,S}$ reaches height at least~$h$ above $\hgt(\cC)$ (any $y\in S$ such that $\cP_{x,S}$ exceeds height $h$ above $y$ has this property),
and thus has an excess area of at least~$4h$, implying the event 
$\widehat G_x^c$.

We will bound the probability of these events via the exponential tails of $\fm(\fW_{x,S})$ established in~\S\ref{sec:wall-clusters}. Applying those bounds to the entire set of $x\in S_n$ must be done carefully though, as one ``tall'' nested sequence of walls $\fW_{x,S}$ may elevate other walls nested within it.

Take any $x\in S$ and any set $Z \subset S$ such that $(\cL_0\setminus S) \cup Z$ is connected, and let $\overline \bW = \bW \cup \{ W_z : z\in Z\}$ for any collection of walls such that $\rho(\bigcup_{z\in Z} W_z) \subset Z$. For every $r\geq 1$, we deduce from Corollary~\ref{cor:nested-sequence-inside-a-ceiling} that
\[ 
 \mu^\mp_n\left(\diam(\fG_{x,S\setminus Z})\geq r\mid \bI_{\overline \bW}\right) \le \mu^\mp_n\left(\fm(\fW_{x,S\setminus Z})\geq r\mid \bI_{\overline \bW}\right) \leq e^{-(\beta - C)r} \,,\]
(using that $\diam(\fG_{x,S\setminus Z})=\diam(\fW_{x,S\setminus Z})$).
On the event that $\diam(\fG_{x,S\setminus Z}) <r$, a union bound over $\{u: d(x,u)\le r\}$ gives by Corollary~\ref{cor:nested-sequence-inside-a-ceiling}, 
\[
    \mu_{n}^{\mp} \bigg( \widehat G_{x,S\setminus Z}^c\,,\, \diam(\fG_{x,S\setminus Z})<r \mid \bI_{\overline \bW}\bigg)  \le \mu^\mp_n\bigg(\bigcup_{u\in B_r(x)} G_{u,S\setminus Z}^c \given \bI_{\overline\bW}\bigg) \leq r^2 e^{-4 (\beta - C) h}\,.
\]
Combining the last two inequalities with a choice of $r=4h$ we see that, for every  $h\geq 1$,
\begin{equation}\label{eq:hat-Gx-bound} \mu^\mp_n\left(\widehat G_{x,S\setminus Z}^c \mid \bI_{\overline\bW}\right) \leq e^{-4(\beta - C)h} + (4h)^2 e^{-4(\beta - C)h} \leq e^{-4(\beta- C')h}\,,\end{equation}
for some larger absolute constant $C'$. 

Next, label the faces $x\in S$ as $x_1,x_2,\ldots$ in a way such that $x_k$, for every $k\geq 1$, is a closest face to~$\partial S$ among all faces of $S$ not already indexed. 
As $\cL_0 \setminus S$ is connected, this guarantees that the face set $(\cL_0 \setminus S) \cup Z_k$ remains connected for every $k\geq 1$, where $Z_k := \bigcup_{i=1}^k \rho(\hull\fG_{x_i,S})\cup \{x_i\}$. 
Proceed to reveal $\{\fG_{x_i,S}: i\le |S|\}$ iteratively while $\bigcap_{i< k} \widehat G_{x_i}$ holds, as follows. In the $k$'th step:
\begin{enumerate}[(a)]
\item either $x_k\in Z_{k-1}$, in which case $\widehat G_{x_k}$ occurs by our conditioning;
\item or $x_{k}\notin Z_{k-1}$, in which case $ \fG_{x,S} =\fG_{x,{S\setminus Z_{k-1}}}$ by construction of $Z_{k-1}$ and we may apply~\eqref{eq:hat-Gx-bound}. 
\end{enumerate}
 It follows that, if $\cF_k$ is the  filtration associated with this process, 
 then in both cases
\begin{equation}\label{eq:upper-crude-iter} \mu^\mp_{n}\bigg(\widehat G_{x_k} \given  \cF_{k-1}\,,\,\bigcap_{i<k}\widehat G_{x_i}\,,\, \bI_\bW\bigg)  \geq 1 - e^{-4(\beta-C')h}\,,\end{equation}
and we can conclude that, for the event $\cG^{\fm}_{S}(4h)=\bigcap_{x\in S}G_x=\bigcap_{x\in S}\widehat G_x$, we have
\[
\mu^\mp_{n}\left(\cG_{S}(4h) \given \bI_{\bW} \right)=\mu^\mp_n\bigg(\bigcap_{x\in S} \widehat G_x \given \bI_{\bW} \bigg) \geq \left(1 - e^{-4(\beta-C')h}\right)^{s_n}\geq \exp\bigg(-s_n e^{-4(\beta-C'')h}\bigg)\,,\]
using $1-x \geq e^{-x-x^2}\geq e^{-\frac32 x}$ for all $x<\frac12$ (as $h\geq 1$, we may set $\beta_0$ such that $4(\beta-C')h\geq 1$). This establishes the right inequality in~\eqref{eq:max-lower-crude}, and the left inequality in that display follows from the aforementioned observation that $\cG^{\fm}_S(4h)$ implies that $\bar M_S < h$.

\smallskip
\noindent
\emph{Proof of~\eqref{eq:max-upper-crude} (upper bound on the probability of interfaces with $\bar M_{S_n} < h$).} The sought bound will follow from a straightforward forcing argument, iteratively applying  Claim~\ref{clm:lower-bound-forcing}. 

Consider a (maximal) mesh $\bar S$ of the faces of $S_n$ such that for every $x,y\in \bar S$, their distance is at least two and notice that $|\bar S|\ge \frac 18 s_n$. Let $\bar x = \{f\in \cL_{0,n}: f\sim^* x\}\cup\{x\}$ as before, and define   
$$\bar S^+ := \{ x \in \bar S : \fW_{\bar x,S}=\trivincr\}\cup \{x\in \bar S: \fW_{x,S} = \{W_{x,\parallel}^{h}\}\}\,.$$
We first establish that $|\bar S^+|$ is comparable to $s_n$ with very high probability. 
Indeed, if $|\bar S^+|\leq \frac 18 s_n-r$ then $|\bar S^+|\le |\bar S| -r$ and thus there must exist a subset $\{x_i\}\subset \bar S$ with disjoint nested sequences of walls $(\fW_{\bar x_i,S})_i$ such that $\sum \fm(\fW_{\bar x_i,S})\geq r$. 
The number of choices for the standard wall collection $\Theta_{\textsc{st}}\bigcup_{z\in \bar S} \{W\in \fW_{\bar z,S}\}$ with an excess of $r$ is at most
$\binom{|\bar S| + r-1}r s^r$ for an absolute constant $s>0$
(here it is useful to identify to each wall a representative in some predetermined ordering of the vertices of $\bar S$; then, one needs to partition the total excess of $r$ into the faces of $S_n$ as $r=\sum_{\bar x\in |\bar S|} r_{\bar x}$ according to their representative, and enumerate over at most $s^{r_{\bar x}}$ options (by Fact~\ref{fact:number-of-walls}) for each  $\Theta_{\textsc{st}}\fW_{\bar x,S}$). 
By Theorem~\ref{thm:rigidity-inside-wall},
\[ \mu_{n}^\mp\left(|\bar S^+| \leq \tfrac34 |\bar S|  \mid \bI_{\bW} \right) \leq 
\sum_{r\geq |\bar S|/4} \binom{|\bar S|+r}r s^r e^{-(\beta-C)r} \leq \sum_{r\geq |\bar S|/4}\left(5e^{C+1} s e^{-\beta}\right)^r \leq e^{-\frac1{32}(\beta-C')s_n}\,,
\]
using $\binom{a}b \leq (ea/b)^b$ in the second inequality. Having established that, we can upper bound 
\begin{align}\label{eq:lower-bound-reduction-to-S-plus}
    \mu_n^{\mp}(\bar M_{S_n}< h \mid \bI_{\bW}) \le \sup_{\bar S^+\subset \bar S : |\bar S^+|\ge \frac 34 |\bar S|}\mu_n^{\mp}\Big(\bigcap_{x\in \bar S^+} \{\fW_{\bar x,S} = \trivincr\} \mid \bar S^+, \bI_\bW\Big)+ e^{ - (\beta - C')s_n/32}\,.
\end{align}
Consider the supremum, by fixing any subset $\bar S^+ \subset \bar S$ having $|\bar S^+|\ge 3|\bar S|/4$, and revealing all the walls $(W_z)_{z\in S_n \setminus \bar S^+}$ under the distribution. Call $\cF_0$ the filtration generated by the family of walls $(W_z)_{z\in S_n\setminus \bar S^+}$, enumerate the first $\frac 34 |\bar S|$ many faces of $\bar S^+$ as $x_1,x_2,\ldots,x_{3|\bar S|/4}$, and let $\cF_i$ be the filtration generated by $\cF_0$ and $\fW_{\bar x_j,S} \in \{\trivincr, W_{x_j,\parallel}^h\}$ for $j<i$. Recall from Claim~\ref{clm:lower-bound-forcing} (with the choices of $S_n= \{\bar x_i\}$ and $\bW =(W_z)_{z\notin \bar x_i}$ there, where the criteria are satisfied due to the conditioning on $x_i \in \bar S^+$) that for every $i$, 
\begin{align*}
    \mu_n^{\mp}\big(\fW_{\bar x_{i},S} = \{W_{x_i,\parallel}^h\} \mid  \bar S^+, \bI_\bW, \cF_{i-1}\big) \ge \frac 12 \exp( - 4(\beta + C)h)\,.
\end{align*}
Iteratively applying the complementary bound, we obtain 
\begin{align*}
    \mu_n^{\mp}\Big(\bigcap_{x\in \bar S^+} \{\fW_{\bar x,S}  = \trivincr\} \mid \bar S^+, \bI_\bW , \cF_0\Big) & \le \prod_{i=1}^{3|\bar S|/4} \mu_n^{\mp}(\fW_{\bar x_i,S} = \trivincr \mid \bar S^+, \bI_{\bW}, \cF_{i-1}) \\ 
    & \le \big(1- e^{ - 4(\beta + C)h}\big)^{3|\bar S|/4}\,. 
\end{align*}
Using the fact that $|\bar S|\ge \frac 18 s_n$ and $(1-x)\le e^{-x}$, this is at most $\exp( - \frac 1{10} s_n e^{ - 4(\beta + C)h})$. Plugging this in to~\eqref{eq:lower-bound-reduction-to-S-plus} we obtain the desired up to changing the constant $C$. 
\end{proof}

\subsection{Maximum within thick ceilings}\label{subsec:thick-ceils}
For ceilings $S_n$ whose isoperimetric dimension is uniformly bounded ($|\partial S_n|\le |S_n|^{(d-1)/d}$ for some fixed $d$) we are able to use the shape of typical pillars and their coupling to the tall pillars of $\mu_{\Z^3}^{\mp}$ as established in Sections~\ref{sec:tall-pillar-shape}--\ref{sec:pillar-couplings}, to close the gap between the $4\beta \pm C$ rates in the exponents of~\eqref{eq:max-lower-crude}--\eqref{eq:max-upper-crude}. Closing this gap and identifying the rate as the infinite-volume large deviation (LD) rate $\alpha_h$ is critical to sharp asymptotics of the recentered maximum height oscillations $\bar M_{S}$, as desired by Theorem~\ref{thm:gumbel}.

\medskip 
\noindent \emph{Proof strategy.} In order to refine the lower bound on the probability in~\eqref{eq:max-lower-crude}, and replace $(4\beta - C)h$ with $\alpha_h$, the bulk of the steps in the revealing scheme must use pillars whose probability of attaining height $h$ is exactly $\exp ( - \alpha_h)$. 
To achieve this, we coarse-grain $S_n$ into smaller-scale tiles, separated by an $\omega(h)$ distance. The pillars in each of these tiles are then iteratively revealed: conditionally on the pillars in some set of tiles, the LD rate of any pillar in the next tile to be exposed is $\alpha_h$ by Corollary~\ref{cor:cond-pillar-rate} (within each tile a union bound is satisfactory). The boundary regions between the tiles are then of a smaller order (here using our assumption on $\isodim(S_n)$), and are processed as a final step using the cruder $(4\beta - C)h$ rate.  

When refining the upper bound on the probability in~\eqref{eq:max-upper-crude}, there is no mechanism for inserting random pillars attaining LD rate $\alpha_h$; we therefore use a tiling scheme similar to that of the upper bound here. As this is an upper bound on the event under consideration, it suffices to consider the interiors of the tiles (and disregard the boundary regions between tiles), and furthermore focus on the class of nice pillars with empty base which are already sufficiently typical and costly. The proof proceeds by iteratively revealing the pillar profiles in the tiles, but this time it uses the  correlation control of Proposition~\ref{prop:two-pillars} to perform a second-moment method within each tile.  

\begin{proof}[\textbf{\emph{Proof of Proposition~\ref{prop:max-thick}}}]
For ease of notation, throughout this proof we let $S=S_n$ with $s_n = |S_n|$, as well as $\bW=\bW_n$ and $h = h_n$. 
The quantity $\alpha_h$  from~\eqref{eq:alpha-alpha-h-def} satisfies, for every $h\geq 1$,
\[(4\beta - C)h \leq \alpha_h \leq (4\beta+e^{-4\beta})h+C\,,\] with $C>0$ an absolute constant (cf.~\cite[Prop.~2.29 and Eq.~(6.3)]{GL19a} and also~\cite[Cor.~5.2]{GL19b}). 
Let $C_0>0$ be a large absolute constant w.r.t.\ which both this and the statements of Theorem~\ref{thm:rigidity-inside-wall} and Eq.~\eqref{eq:upper-crude-iter}  hold. 
Set $\kappa_0=2C_0$, recalling that the subset $S^\circ = S_{n,h}^\circ= \{x\in S \,:\; d(x,\partial S)\geq e^{4 C_0 h}\}$.
We will soon use that, comparing $\exp(-\alpha_h)$ to its crude estimate $\exp(-(4\beta-C_0)h)$, for large $\beta$ we have
\begin{equation}\label{eq:alpha-h-lower-4beta-h}
e^{-\alpha_h+(4\beta-C_0)h} \geq e^{-(C_0+1)h}\,.
\end{equation}
Tile $\cL_0$ using boxes of side length 
\[L := \lfloor \tfrac12 e^{4 C_0 h} \rfloor\,,\] 
(e.g., $\llb i L,(i+1) L\rrb\times \llb j L,(j+1)L\rrb\times\{0\}$ for all $i,j\in\Z$),  observing that $e^{4C_0 h}\leq s_n^{4C_0/\sqrt\beta} \leq s_n^{\epsilon_\beta}$. Let 
\[ S' = \left\{ x\in S : d(x,\partial S) > 2 e^{4C_0 h}\right\}\,,\]
and let $\{Q_i\}$ be the subset of every such box intersecting $S'$. By definition, $d(Q_i,\partial S)> e^{4C_0 h}$ for every~$i$, and thus $Q_i\subset S^\circ$.
Letting $d = \isodim(S)$, we get
\begin{equation}\label{eq-S-S'-bound} |S\setminus S'| \leq |\partial S| (2e^{4C_0 h})^2 \leq 4 s_n^{(d-1)/d + 8C_0/\sqrt{\beta}}\,,\end{equation}
which is $o(s_n)$ provided $d< \sqrt{\beta}/(8C_0)$.
Further let $Q'_i\subset Q_i$ be the concentric sub-rectangle of side length 
\[ L' = L - \lceil e^{2C_0 h}  \rceil\]
within $Q_i$, and write $Q = \bigcup_i Q_i$ and $Q' = \bigcup_i Q'_i$, 
recalling from the above  definitions (and $h\gg 1$) that
\[|Q'|=(1-o(1))|Q|\quad\mbox{and}\quad s_n\geq |Q| \geq |S'| = (1-o(1))s_n\,.\]

\smallskip
\noindent
\emph{Proof of~\eqref{eq:max-thick-lower-large-h} (lower bound on the probability of interfaces with $\bar M_{S_n} < h$).} 
For every $x\in S$, define the events $G_x$, the quantity $\fG_x$ and the event $\widehat G_x$ as in~\eqref{eq-G-H-events}--\eqref{eq:hat-Gx-def}. Further define, for $A\subset S$,
\begin{equation}
\label{eq:def-H-dagger-lower}
\widehat H_{x,A} = \{ \hgt(\cP_{x,A})<h\}\cap\{\fm(\fW_{x,A})<5h\}\,.
\end{equation}
(As before, referring to the events $G_x$, $\widehat G_x$ and $\widehat H_x$ without a subscript $A$ indicates the default choice $A= S$.)
We will show that for every $h\geq\sqrt{\log s_n}$,
\begin{align}\label{eq:thick-upper-bound}
\mu_{n}^\mp\bigg(\bigcap_{x\in Q'} \widehat H_x \cap \bigcap_{x\in S\setminus Q'} \widehat G_x
\given \bI_{\bW} \bigg) &\geq \exp\left(- (1+o(1))|S\setminus Q'| e^{-(4\beta-C_0)h}-(1+\epsilon_\beta+o(1))s_n e^{-\alpha_h}\right)\,,\end{align}
which will imply~\eqref{eq:max-thick-lower-large-h}. Indeed, $\bigcap_{x\in S\setminus Q'} \widehat G_x \subset \cG^\fm_{S\setminus S^\circ}(4h)$ (as $\widehat G_{x}\subset G_{x}$ and $S\setminus Q'\supset S\setminus S^\circ$) and $\widehat G_x\subset \widehat H_x$ since, on one hand, it clearly implies $\fm(\fW_{x,S})\leq 5h$, and on the other (as explained in the proof of Proposition~\ref{prop:max-generic}),
having $\hgt(\cP_{x,S})\geq h$ would imply that some $W_y$ nested in $\fW_{x,S}$ has $\fm(\fW_{y,S})\geq 4h$, violating~$\widehat G_x$.
Therefore, the event on the left-hand side of~\eqref{eq:thick-upper-bound} is contained in the event $\{\bar M_{S}<h\}\cap \cG^\fm_{S\setminus S^\circ}(4h) \cap \cG_{S^\circ}^{\fm}(5h)$. 
Since $S' \subset Q$ implies that $|S\setminus Q'|\leq |S\setminus S'| + |Q \setminus Q'|$, whereas by ~\eqref{eq:alpha-h-lower-4beta-h} and~\eqref{eq-S-S'-bound}, if $d=\isodim(S)$ then
\[ |S\setminus S'|e^{-(4\beta-C_0)h} \leq 4 s_n^{(d-1)/d+(9C_0+1)/\sqrt{\beta}} e^{-\alpha_h} = o(s_n e^{-\alpha_h})
\]
provided that $d<\sqrt{\beta}/ (9C_0+1)$, and by~\eqref{eq:alpha-h-lower-4beta-h} and the definition of $L$ and $L'$,
\[ |Q\setminus Q'| e^{-(4\beta-C_0)h}
\leq O((L-L')/L) s_n e^{-(4\beta-C_0)h} 
=O(e^{-(C_0-1)h} s_n e^{-\alpha_h})=o(s_n e^{-\alpha_h})\,.
\]
It thus follows that~\eqref{eq:thick-upper-bound} implies the sought inequality~\eqref{eq:max-thick-lower-large-h}.

To show~\eqref{eq:thick-upper-bound}, we first treat the walls in $S\setminus Q$.
By iteratively revealing $\fG_u$ for all $u\in S\setminus Q$, in the exact same manner as was done in~\eqref{eq:upper-crude-iter}, we find that
\begin{align}\label{eq:bdy-upper}
    \mu_n^\mp \Big(\bigcap_{x\in S\setminus Q}  \widehat G_x \mid \bI_{\bW}\Big) \ge \Big(1-e^{ - (4\beta - C_0) h}\Big)^{|S\setminus Q|} \geq \exp\Big( - |S\setminus Q| e^{ - (4\beta - C_0) h}\Big)
    \,.
\end{align}

We next treat $Q$, the bulk of the sites.
Assume w.l.o.g.\ that the ordering of the $Q_i$'s is such that $Q_i$ minimizes $d(Q_k,\partial S\cup \bigcup_{j<i} Q_j)$ among all boxes $Q_k$ for $k\geq i$.  
Via this ordering, if 
 $$ Z_i = \bigcup\Bigl\{ \rho(\hull\fG_x)\cup \{x\}\,:\; x\in (S\setminus Q) \cup \bigcup_{j\le i} Q_j\Bigr\}\,,$$
then the face set $(\cL_0\setminus S)\cup Z_i$
remains connected for every $i$.
Let $(\cF_i)$ denote the filtration generated by $\fG_{u}$ for all $ u\in Z_i$. 
Condition on $\cF_{i-1}$ and the event \[ D_{i-1} =  \left(\bigcap\left\{ \widehat H_x \,:\; x\in \mbox{$\bigcup_{j<i} Q'_j$}\right\} \right)\cap\left( \bigcap\left\{ \widehat G_x\,:\; x\in \mbox{$(S\setminus Q)\cup\bigcup_{j<i} Q_j\setminus Q'_j$}\right\}\right)\,.\]
(Note that $D_{i-1}$ is measurable w.r.t.\ $\cF_{i-1}$.) We next argue that for every $x\in Q_i'$,
\begin{equation}\label{eq:couple-pillar-to-inf-vol}
\mu_{n}^\mp(\hgt(\cP_{x,S})\geq h \mid \cF_{i-1}\,, D_{i-1}) \leq  (1+\epsilon_\beta) e^{-\alpha_h}\,.
\end{equation}
Indeed, the conditioning revealed $\fG_{u}$ for every $u\in \bigcup_{j<i} Q_j$, yet importantly, the event $D_{i-1}$ stipulates that
\begin{enumerate} [(a)]
	\item every wall $W\Subset S$ that was revealed as part of $\fG_{u}$ for $u\in \bigcup_{j<i}(Q_j\setminus Q'_j)$ has $\fm(W)<4h$ (as per the event~$\widehat G_u$), so in particular $d(W,Q'_i)\geq L-L'-4h =  (1-o(1))e^{2C_0 h}$; and 
	\item every wall $W\Subset S$ that was revealed as part of $\fG_{u}$ for $u\in \bigcup_{j<i}Q'_j$ must also have $d(W,Q'_i)\geq L-L' $, or else it must necessarily be part of $\fG_{v}$ for some $v\in Q_j\setminus Q'_j$ while having $\fm(W)\geq(1-o(1))e^{2C_0 h}$, which (recalling that $\sqrt{\log s_n}\leq h=O(\log s_n)$) is in violation of~$\widehat G_v$. 
\end{enumerate} Altogether, every wall revealed as part of $\cF_{i-1}$, on $D_{i-1}$, is at distance at least $(1-o(1))e^{2C_0 h}$ from~$Q'_i$.
Applying Corollary~\ref{cor:cond-pillar-rate}
with $\overline\bW = \bW\cup \{\fG_{u,S} : u \in Z_i\}$, here $h=o(\Delta_n)$ (as $h\gg 1$ and 
 $\Delta_n\geq(1-o(1))e^{2C_0 h}$),  so~\eqref{eq:cond-uncond-coupling-1} in the conclusion of that theorem 
 implies the required bound~\eqref{eq:couple-pillar-to-inf-vol}. 
 
 In addition, as the walls revealed as part of $\cF_{i-1}$, on $D_{i-1}$, do not intersect $Q'_i$, Corollary~\ref{cor:nested-sequence-inside-a-ceiling} implies that
 \[ \mu_n^\mp(\fm(\fW_{x,S}) \geq 5h \mid \cF_{i-1},D_{i-1}) \leq e^{-(\beta-C)5h} = o(e^{-\alpha_h})\,,\]
using that $\alpha_h \geq (4\beta-C_0)h$ from~\eqref{eq:alpha-h-lower-4beta-h}.
 
Combining these two estimates (while absorbing the latter in the term $\epsilon_\beta$ from~\eqref{eq:couple-pillar-to-inf-vol}), we deduce that
\begin{align*} \mu_{n}^\mp\bigg(\bigcup_{x\in Q'_i} \widehat H_x^c \given \cF_{i-1}\,, D_{i-1}, \bI_{\bW}\bigg) &\leq  
 \sum_{x\in Q_i'} \mu_{n}^\mp\left( \widehat  H_x^c \given \cF_{i-1}\,, D_{i-1}, \bI_{\bW}\right) \leq 
(1+\epsilon_\beta)|Q'_i|e^{-\alpha_h} \,.
\end{align*}
For any $x\in Q_i\setminus Q'_i$, if $x\in Z_{i-1}$ then necessarily $\fW_{x,S}\subset \fG_u$ for some $u\in\bigcup_{j<i}(Q_j\setminus Q'_j)$ (as $Q_j\setminus Q'_j$ separates $Q'_j$ from $Q_i$), and thus $\widehat G_x^c$ automatically holds by $D_{i-1}$. Otherwise, by~\eqref{eq:hat-Gx-bound} with the choices $Z= Z_{i-1}$ and $\overline \bW = \bW \cup \{W_z: z\in Z_{i-1}\}$, we have for every $x\in Q_i \setminus Q_i'$, 
\[ \mu_n^\mp\left( \widehat G_x^c \mid \cF_{i-1}\,,D_{i-1}\,, \bI_{\bW}\right) \le \mu_n^{\mp}\left(\widehat G^c_{x,S\setminus Z_{i-1}} \mid \cF_{i-1}\,, D_{i-1}\,, \bI_{ \overline \bW}\right) \leq e^{-(4\beta-C_0)h}\,, 
\]
where we used the fact that $\widehat G_{x,S} = \widehat G_{x,S\setminus Z_{i-1}}$ when $x\notin Z_{i-1}$. 
Thus, 
\[ \mu_{n}^\mp\bigg(\bigcup_{x\in Q_i\setminus Q'_i} \widehat G_x^c \given \cF_{i-1}\,, D_{i-1}\,, \bI_{\bW}\bigg) \leq  
\sum_{x\in Q_i\setminus Q'_i}\mu_{n}^\mp( \widehat G_x^c \given \cF_{i-1}\,, D_{i-1}\,,\bI_{\bW}) \leq 
|Q_i\setminus Q'_i|e^{-(4\beta-C_0) h}\,.
\]
Recalling that $|Q_i|=L^2 \leq e^{8C_0 h}$, along with the fact that $\alpha_h \geq (4\beta-C_0)h$ from~\eqref{eq:alpha-h-lower-4beta-h}, we see that
\[
|Q_i|e^{-\alpha_h} \leq |Q_i|e^{-(4\beta-C_0)h} = (1+o(1))e^{(9C_0 - 4\beta)h} = o(1)
\]
for all $\beta>9C_0$; thus, applying $1-x\geq e^{-x/(1-x)}$ for $0<x<1$, we obtain that
\[ \mu_{n}^\mp\bigg(\Big(\bigcap_{x\in Q'_i} \widehat H_x\Big) \cap\Big(  \bigcap_{x\in Q_i\setminus Q'_i} \widehat G_x \Big)\given \cF_{i-1}\,, D_{i-1}\,,\bI_{\bW}\bigg) \geq e^{-(1+o(1))|Q_i\setminus Q'_i|e^{-(4\beta-C_0)h}
-(1+\epsilon_\beta+o(1))|Q_i'|e^{-\alpha_h} }\,.\]
Iterating this over all $i$ results in a lower bound of
\[ \exp\left(-(1+o(1))|Q\setminus Q'| e^{-(4\beta-C_0)h}-(1+\epsilon_\beta+o(1))|Q'|e^{-\alpha_h} \right)\,,\]
which, when multiplied by the right-hand of~\eqref{eq:bdy-upper}, 
establishes~\eqref{eq:thick-upper-bound}, and hence also~\eqref{eq:max-thick-lower-large-h}.

\smallskip
\noindent
\emph{Proof of ~\eqref{eq:max-thick-upper-large-h} (upper bound on the probability of interfaces with $\bar M_{S_n} < h$).} 
For this part, 
define the following variants of the events $\widehat G_x$ and $\widehat H_x$ from above:
\[ G^\dagger_x = \{\diam(\fW_{x,S})< e^{2 C_0 h}\}\,,\qquad H^\dagger_x = \{ \hgt(\cP_{x,S}) < h \} \cup \{ \sB_{x,S}\neq\emptyset\}\,,\]
where $\sB_{x,S}$ is the base of $\cP_{x,S}$. 
 We will show that
\begin{align}\label{eq:thick-lower-bound}
\mu_n^\mp\bigg(\bigcap_{x\in Q'} H^\dagger_x\cap \bigcap_{x\in Q\setminus Q'}  G^\dagger_x \given \bI_{\bW} \bigg) &\leq \exp\left(-(1-\epsilon_\beta)s_n e^{-\alpha_h}\right)\,,
\end{align}
which will immediately imply the inequality~\eqref{eq:max-thick-upper-large-h}, since  $\{\bar M_\cC < h\} \subset \bigcap_{x\in Q'} H^\dagger_x$ and $\cG^{\mathfrak D}_{S^\circ}(e^{2C_0 h})\subset \bigcap_{x\in Q'}G^\dagger_x$.
To prove~\eqref{eq:thick-lower-bound}, 
recall the definition of the filtration $(\cF_i)$ given above, and let
\[ D'_{i} = \bigcap_{j\leq i}\bigg(\bigcap_{x\in Q'_j} H^\dagger_x \cap   \bigcap_{x\in Q_j\setminus Q_j'} G^\dagger_x\bigg)\,.\]
By the Bonferroni inequalities (inclusion--exclusion),
\begin{align}
 \mu_{n}^\mp\bigg(\bigcup_{x\in Q'_i} (H^\dagger_x)^c \given \cF_{i-1}\,,D'_{i-1}\,, \bI_{\bW} \bigg) &\geq \sum_{x\in Q'_i} \mu_{n}^\mp\left((H^\dagger_x)^c\mid\cF_{i-1}\,,D'_{i-1}\,,\bI_{\bW}\right) \nonumber\\
&- \frac12 \sum_{x\in Q'_i} \sum_{y\in Q'_i}
\mu_{n}^\mp\left((H^\dagger_x)^c\cap(H^\dagger_y)^c\mid\cF_{i-1}\,,D'_{i-1}\,,\bI_{\bW}\right) \,.
\label{eq:bonferroni}
\end{align}
For every $x\in Q'_i$, as argued above, the walls revealed as part of $\cF_{i-1}$ on the event $D'_{i-1}$ are all at distance at least $\Delta_n \geq (1-o(1))e^{2C_0 h} \gg h$ from $x$ as per the event $G^\dagger_x$. Hence, Corollary~\ref{cor:cond-pillar-rate} yields
\[ 
 \mu_{n}^\mp\left(\hgt(\cP_{x,S})\geq h \mid \cF_{i-1}\,,D'_{i-1}\,, \bI_{\bW}\right) \geq 
(1-\epsilon_\beta) e^{-\alpha_h}\,.\]
Furthermore, via Theorem~\ref{thm:shape-inside-ceiling} (and the fact that $\{\cI\restriction_S \in \Iso_{x,L,h}\} \subset\{\sB_{x,S} = \emptyset\}$), 
\[ 
\mu_{n}^\mp(\sB_{x,S}=\emptyset \mid \hgt(\cP_{x,S})\geq h\,,\cF_{i-1}\,,D'_{i-1}, \bI_{\bW}) 
\geq 1-\epsilon_\beta\,.
\]
(Alternatively, one could have deduced this by combining
the coupling bound~\eqref{eq:cond-uncond-coupling-2} in Theorem~\ref{thm:cond-uncond-simple} with the result of~\cite[Theorem 4.1(a)]{GL19b} stating that $\mu_{\Z^3}^\mp(\sB_{o}=\emptyset\mid \hgt(\cP_{o})\ge h)\geq 1-\epsilon_\beta$.)
Combined, it follows that
\[ \mu_{n}^\mp\left( (H^\dagger_x)^c \mid \cF_{i-1}\,, D'_{i-1}\,,\bI_{\bW}\right) \geq (1-\epsilon_\beta) e^{-\alpha_h}\]
for some other $\epsilon_\beta$ going to $0$ as $\beta\to\infty$.
At the same time, Proposition~\ref{prop:two-pillars} 
shows that, for every $x,y\in Q_i'$,
\[\mu_n^{\mp} \left((H^{\dagger}_x))^c \cap (H^\dagger_y)^c \mid \cF_{i-1}, D_{i-1}'\,, \bI_{\bW}\right)\le e^{ - (4\beta - C)h} \mu_n^{\mp} \left((H^\dagger_x)^c\mid \cF_{i-1}, D_{i-1}'\right)\,,
\]
and plugging the last two displays in~\eqref{eq:bonferroni},  while recalling $|Q'_i|e^{-(4\beta-C)h} =o(1)$ as $h\geq\sqrt{\log n}$, shows that
\[  \mu_{n}^\mp\bigg(\bigcup_{x\in Q'_i} (H^\dagger_x)^c \given \cF_{i-1}\,, D'_{i-1}\,,\bI_{\bW}\bigg) \geq  (1-\epsilon_\beta - o(1))|Q_i'| e^{-\alpha_h} \,.
\]
Rearranging this and using $1-x \le e^{ - x}$ we find that
\begin{align*}
 \mu_{n}^\mp\bigg(\bigcap_{x\in Q'_i} H^\dagger_x\cap \!\!\bigcap_{x\in Q_i\setminus Q'_i} \!\!G^\dagger_x \given \cF_{i-1}\,, D'_{i-1}\,,\bI_{\bW}\bigg) & \leq
 \mu_{n}^\mp\bigg(\bigcap_{x\in Q'_i} H^\dagger_x \given \cF_{i-1}\,, D'_{i-1}\,,\bI_{\bW}\bigg) \\
 & \leq \exp\left(-(1-\epsilon_\beta - o(1))|Q_i'| e^{-\alpha_h} \right)\,,
\end{align*}
which yields~\eqref{eq:thick-lower-bound} once we iterate this bound over all $i$.

Finally, we deduce~\eqref{eq:max-thick-combined-large-h} from~\eqref{eq:max-thick-lower-large-h} and~\eqref{eq:max-thick-upper-large-h} 
via a simple union bound over $\{\cG^{\mathfrak D}_x(e^{2C_0 h}) : x\in S_n\}$, as we have that
$\mu_{n}^\mp(\cG^{\mathfrak D}_x(e^{2C_0 h})\mid\bI_{\bW}) \leq \exp(-(\beta-C)e^{2C_0 h} )$ by Corollary~\ref{cor:nested-sequence-inside-a-ceiling}.
\end{proof}

With the tail bounds on $\bar M_{S_n}$ established in Propositions~\ref{prop:max-generic}--\ref{prop:max-thick}, we are able to easily deduce tightness, as well as Gumbel tail behavior of the maximum oscillation within $S_n$ conditionally on $\bI_{\bW_n}$. 

\subsection{Proof of Theorem~\ref{thm:gumbel}}
We will establish the required estimates for every $0 \leq k\leq \frac1{\beta^2}\log s_n$.
Recall that $m_{s_n}^*$ as defined in~\eqref{eq:m*-def} satisfies $m_{s_n}^* \asymp \log s_n$. We further have that
\[ \gamma := s_n \exp(-\alpha_{m_{s_n}^*})\quad\mbox{satisfies}\quad \exp(-2\beta-e^{-4\beta})\leq \gamma <\exp(2\beta)\,, \]
with the upper bound $\gamma < e^{2\beta}$ following by definition of $m^*_{s_n}$, and the lower bound due to the super-additivity
\[ \alpha_{h_1}+\alpha_{h_2}-\epsilon_\beta \leq \alpha_{h_1+h_2} \leq \alpha_{h_1}+\overline \alpha h_2 \qquad\mbox{for every $h_1,h_2\geq 1$}\,,\]
with $\overline\alpha = 4\beta+e^{-4\beta}$, 
as established in~\cite[Corollary~5.2]{GL19b}. Applying the two inequalities in the last display with the pair $(h_1,h_2)$ taking values either $(m_{s_n}^*,k)$ or $(m^*_{s_n}-k,k)$, we respectively deduce  
\begin{align} \gamma \exp\left( -\overline\alpha k \right) &\leq s_n \exp(-\alpha_h) \leq (1+\epsilon_\beta)\gamma \exp\left( -\alpha_k \right)&\mbox{for $h=m_{s_n}^*+k$}\,,\label{eq:alpha-h-ml+k}\\
(1-\epsilon_\beta)\gamma \exp\left( \alpha_k \right) &\leq s_n \exp(-\alpha_h) \leq \gamma \exp\left( \overline\alpha k \right)&\mbox{for $h=m_{s_n}^*-k$}\label{eq:alpha-h-ml-k}\,.
\end{align}
For the lower tail, we invoke Proposition~\ref{prop:max-thick} for $h=m_{s_n}^*-k$, absorbing the term $\exp(-s_n^{C/\beta})$ in~\eqref{eq:max-thick-combined-large-h} into the $\epsilon_\beta$ in the first term appearing in that bound since (taking $\beta_0$ large enough) \[ s_n \exp(-\alpha_h)\leq  \gamma\exp(\overline \alpha k) = o(s_n^{C/\beta})\quad\mbox{for all $k\leq \frac1{\beta^2}\log s_n$} \]
provided that $\beta$ is large enough, and deduce from~\eqref{eq:alpha-h-ml-k} that
\[ 
\exp\left(- (1+\epsilon_\beta)\gamma e^{\overline \alpha k}\right) \leq  \mu^\mp_{n}\left( \bar M_{S_n}< m^*_{s_n} -k  \mid \bI_{\bW_n}\right) \leq 
\exp\left(
- (1-\epsilon_\beta) \gamma e^{\alpha_k}\right)\,.
\]
 For the upper tail, Proposition~\ref{prop:max-thick} (this time, using~\eqref{eq:alpha-h-ml+k}) analogously gives 
\[ 
\exp\left(- (1+\epsilon_\beta)\gamma e^{-\alpha_k }\right) \leq  \mu^\mp_{n}\left( \bar M_{S_n}<m^*_{s_n}+k  \mid \bI_{\bW_n}\right) \leq 
\exp\left(
- (1-\epsilon_\beta) \gamma e^{-\overline \alpha k}\right)\,.
\]
For every $k\geq 1$, we use $1-x\leq e^{-x} \leq 1-x(1-x/2)$, whereby the fact $\gamma e^{-\overline
\alpha k}\leq\gamma e^{- \alpha_k} \leq e^{2\beta-(4\beta-C)k}\leq\epsilon_\beta $ allows us to replace the term corresponding to $(1-x/2)$ by $1-\epsilon_\beta$, to deduce 
\begin{align*} 
(1-\epsilon_\beta) \gamma e^{-\overline \alpha k} \leq 
 \mu_{n}^\mp (\bar M_{S_n} \geq m^*_{s_n} + k \mid \bI_{\bW_n}) \leq (1+\epsilon_\beta)\gamma e^{-\alpha_k} \,.
\end{align*}
This completes the proof.
 \qed

\section{Deferred proofs from Section~\ref{sec:tall-pillar-shape}}\label{sec:deferred-proofs}
In this section we complete the proofs deferred from Section~\ref{sec:tall-pillar-shape}. 
We emphasize that the map $\Phi_\iso$ and its analysis are in fact somewhat simplified from~\cite{GL19b} and the proofs herein resulting from those simplifications can be of interest for pedagogical reasons when compared to the corresponding proofs in~\cite{GL19b}.

\subsection{Decomposition of the interfaces}\label{subsec:decomposition-of-interface}
Fix any interface $\cI \in \bI_{\bW_n}$ and for ease of notation, let $\cJ = \Phi_{\iso}(\cI)$. We begin by partitioning the faces of $\cI$ and $\cJ$ into their constituent parts as dictated by the map $\Phi_{\iso}$. This partioning will govern the pairings of $\g(f,\cI)$ with $\g(f',\cJ)$ when applying~\eqref{eq:g-exponential-decay}. 

Recall that $\cP_{x,S}\subset \cI$ has spine $\cS_{x,S}$ and denote its increments by $\sX_{j}$. Let $\bY\subset \cL_0$ be the set of indices of walls in $\cI \setminus \cS_{x,S}$ that were marked for deletion. Let ${\bD}\subset \cL_0$ be the indices of walls that were deleted (i.e., walls in $\bigcup_{y\in \bY} \Clust(\tilde \fW_{y,S})$). 
Split up the faces of $\cI$ as follows:
{\renewcommand{\arraystretch}{1.3}
\[
\begin{tabular}{m{0.05\textwidth}m{0.3\textwidth}m{0.5\textwidth}}
\toprule
\midrule
 $\bX_A^\cI$ & $\bigcup_{1\leq j \leq j^*} \sF(\sX_j)$ & Increments between $v_1$ and $v_{j^*+1}$  \\ 
 $\bX_{B}^{\cI}$ & $\bigcup_{j^* + 1\leq j \leq \sT+1}\sF(\sX_j)$ & Increments above $v_{j^*+1}$ \\
 $\bV$ & $\bigcup_{z\in \bD} \tilde W_z$ & All walls that were deleted \\
 $\bB$ & $\cI\setminus(\bX_A^\cI\cup\bX_1^\cI\cup\bV)$
 & The remaining set of faces in $\cI$ \\
 \bottomrule
 \end{tabular}
\]
}
where $\bB$ splits further into
{\renewcommand{\arraystretch}{1.25}
\[
\begin{tabular}{m{0.05\textwidth}m{0.3\textwidth}m{0.5\textwidth}}
\toprule
\midrule
    $\bC_1$ & 
    $\bigcup_{y\in \bY} \lceil \tilde W_y \rceil$
& Interior ceilings of walls marked for deletion \\
 $\bC_2$& $\bigcup_{z\in \bD\setminus \bY} \lceil \tilde W_z\rceil \cup \bigcup_{z\notin \bD} \tilde W_z \cup \lceil \tilde W_z \rceil $& 
Interior ceilings of walls that were not marked, along with all non-deleted walls and their interior ceilings\\
$\mathbf{Fl}$ & $\bB \setminus (\bC_1 \cup \bC_2)$
& All remaining faces of $\bB$ (ceiling faces in $S_n^c\times \{0\})$.  \\
\bottomrule
 \end{tabular}
\]
}

We next partition the faces of $\cJ $. Let us first introduce a few pieces of notation. Denote by $\cP_{x,S}^\cJ$ the pillar of $x$ in $\cJ\restriction_{S_n} = \Phi_{\iso}(\cI)\restriction_{S_n}$; observe that by construction, its base $\sB_{x,S}^{\cJ}$ is empty, and its spine $\cS_{x,S}^\cJ$ is all of $\cP_{x,S}^\cJ$. For a given $\cI,x,t$ define the shift map $\lrvec \theta$ as the horizontal shift on the increments $\sX_{j^* + 1},\ldots, \sX_{\sT+1}$ induced by $\Phi_{\iso}$. Namely, for $f\in \bX_B^\cI$, let
\begin{align}\label{eq-theta-lr-def}
        \lrvec \theta f =     
         f + \rho(x- v_{j^*+1}) & \qquad \mbox{if $f\in \bX_{B}^{\cI}$}\,.
\end{align}
We can also define the map $\theta_{\udarrow}$ on faces in $\bB$, that vertically shifts faces of $\bB$ to obtain corresponding faces of $\cJ$ as dictated by the bijection Lemma~\ref{lem:interface-reconstruction} and the removal of the walls in $\bV$. With these, let:  

{\renewcommand{\arraystretch}{1.3}
\[
\begin{tabular}{m{0.05\textwidth}m{0.33\textwidth}m{0.51\textwidth}}
\toprule
\midrule
 $W_{x,\parallel}^{\mathbf h}$ & $\sF(\{x\}\times [\hgt(\cC_\bW),\hgt(\cC_\bW) + \mathbf h])$ & New faces added in step 8 of $\Phi_{\iso}$ \\
 $\bX_A^\cJ$ & $\sF(\{x\} \times [\hgt(\cC_\bW) + \mathbf h, \hgt(v_{j^*+1})+\frac 12))$  $\sF(\{x\} \times [\hgt(\cC_\bW) + \mathbf h, \hgt(\cC_\bW) + h))$,  & Increments of $\cJ$ between $\hgt(v_1)$ and $\hgt(v_{j^*+1})$ or $h$, depending on whether ({\tt{A3}}) is violated or not. \\ 
 $\bX_{B}^{\cJ}$ & $\lrvec \theta \bX_{B}^{\cI}$ & Increments of $\cJ$ above $\hgt(v_{j^*+1})$ \\
 $\theta_{\udarrow} \bB$ & $\theta_{\udarrow}\bC_1 \cup \theta_\udarrow \bC_2 \cup \mathbf{Fl}$ & Vertical translations of $\bB$ due to the deletion of walls~$\bV$ \\
 $\bH$ & $\bigcup \{f\in \cJ\setminus \cP_{x,S}^{\cJ}: \rho(f) \in \sF(\rho(\bV))\}$ & Faces added to ``fill in" the rest of the interface  \\
 \bottomrule
 \end{tabular}
\]
}

\subsection{Preliminary bounds on interface interactions}
We first prove a series of preliminary estimates to which we will reduce Proposition~\ref{prop:partition-function-contribution} by pairing faces together according to the decomposition of $\cI$ and $\cJ$ from \S\ref{subsec:decomposition-of-interface}. Several of these are direct analogues of similar claims/lemmas in~\cite{GL19b}. For these we may omit some details of the proofs when they are straightforward, and repeated from~\cite{GL19b}. The first first two claims of these are immediate by summability of exponential tails in $\Z^d$ together with Claim~\ref{clm:m(W)}.

\begin{claim}[Analogue of {\cite[Claim 4.13]{GL19b}}]\label{clm:W-and-H-vs-everything}There exists $\bar C$ such that 
\[ 
\sum_{f\in \sF(\Z^3)} \sum_{g\in \bV\cup\bH}   e^{-\bar c d(f,g)} \leq \bar C \fm(\bV)\,.
 \]
\end{claim}

\begin{claim}[Analogue of {\cite[Claim 4.14]{GL19b}}]\label{clm:X1-and-X3-vs-everything}There exists $\bar C$ such that 
\[ 
\sum_{f\in \sF(\Z^3)} \sum_{g\in \bX_A^\cI\cup \bX_A^{\cJ}}  e^{-\bar c d(f,g)} \leq \bar C \fm(\cI;\cJ)\,,
\]
\end{claim}

The next lemma similarly controls interactions between the pillars and faces in $\bW$.

\begin{lemma}[Analogue of {\cite[Lemma 4.16]{GL19b}}]\label{lem:F-vs-X2-and-X4}
There exists $\bar C$ such that 
\[ 
\sum_{f\in \bX_B^\cI} \sum_{g\in \cL_{\hgt(\cC_\bW)}} e^{-\bar c d(f,g)} \leq \bar C\,, \qquad \mbox{and}\qquad \sum_{f\in \bX_B^\cI} \sum_{g: \rho(g)\notin S_n} e^{-\bar c d(f,g)} \le \bar C\,.
\]
\end{lemma}
\begin{proof}
Let us begin with the first inequality; decomposing the faces of $\bX_{B}^\cI$ according to the increment number they belong to, and noting that $j>j^*$ implies $\fm(\sX_j)<j-1$ (as ({\tt{A1}}) was not violated), we have 
\begin{align*}
    \sum_{f\in \bX_B^\cI} \sum_{g\in \cL_{\hgt(\cC_\bW)}} e^{-\bar c d(f,g)} \le \sum_{j>j^*} \bar K |\sF(\sX_j)| e^{ - \bar c j}\le \sum_{j>j^*} \bar K (4 + 3j) e^{ - \bar c j} \le \bar C e^{ - \bar c j^*}\,.
\end{align*}
Turning to the proof of the second inequality, observe that $|\bX_{B}^{\cI}|\le |\sF(\cS_{x,S})|\le 5h$ and $d(\rho(\bX_{B}^\cI), S_n^c) \ge h$, by summability of exponential tails, the double sum is at most $C h e^{ - \bar c h}\le \bar C$. 
\end{proof}

The next lemma is used to control the interactions between the vertical shifts in $\bB$ incurred by the deletions of walls in $\cI\setminus \cS_{x,S}$; this is essentially identical to Claim~\ref{clm:wall-cluster-ceiling-to-else}, and replaces the group-of-wall based analogue Lemma 4.17 of~\cite{GL19b}. Following the notation of Claim~\ref{clm:wall-cluster-ceiling-to-else}, let $\cC_1,\ldots,\cC_s$ be the collection of ceilings in the interface whose standard wall representation consists of $\bV$. We can then
decompose the faces of $\bB$ into $(\bB^{i})_{i=1,\ldots,s}$ which are indexed by the innermost nesting ceiling among $\cC_1,\ldots,\cC_s$.  

\begin{lemma}\label{lem:B-vs-B}
There exists $\bar C$ such that  
\begin{align*}
    \sum_{i=1}^{s} \sum_{f\in \bB^i \cup \theta_\udarrow \bB^i} \sum_{g\in \sF(\mathbb Z^3): \rho(g)\notin \rho(\hull {\cC}_i)}  e^{ - \bar c d(f,g)}   \leq \bar C \sum_{i=1}^{s} |\partial \rho(\hull{\cC}_i)| \le \bar C \fm(\bV)\,.
\end{align*}
\end{lemma}

\begin{proof}
Consider the $i$'th summand: clearly, by vertical translation invariance of the index set of the latter sum, it suffices to consider just the sum over $f\in \bB^i$, say.  As in the proof of Claim~\ref{clm:wall-cluster-ceiling-to-else}, we write 
\begin{align*}
\sum_{f\in \bB^i} \sum_{g\in \sF(\mathbb Z^3): \rho(g)\notin \rho(\hull \cC_i)} e^{- \bar c d(f,g)} & \leq \sum_{f\in \rho(\hull{\cC}_i)}\sum_{u\in \rho(\partial \hull{\cC}_i)} C e^{ - \bar c d(f,u)}  +  \sum_{u\in \rho(\partial \hull \cC_i)} \sum_{W: W\subset \bB^i} C |W| e^{ - \bar cd(\rho(W),u)}\,,
\end{align*} 
where the first term on the right-hand side accounts for the ceiling faces of $\bB^i$ and the second term accounts for the wall faces, with the sum running over all walls of $\cI$ that are a subset of $\bB^i$. The first term above is clearly at most $C|\partial \rho(\hull{\cC}_i)|$ for some other $C$. 
Using~\eqref{eq:excess-area-wall-relations} and the fact that $W$ is nested in $\cC_i$ while not being in the ceiling cluster of $\cC_i$, so that $d(\rho(W),\rho(\partial \hull \cC_i))>\fm(W)$, the second term above is at most 
\begin{align*}
    \sum_{u\in \rho(\partial \hull \cC_i)} \sum_{W\subset \bB^i} 2C \fm(\bW) e^{-\frac 12 \bar c [d(\rho(W),u)+\fm(W)]} \leq \sum_{u\in \rho(\partial \hull \cC_i)} \sum_{W\subset \bB_i} C' e^{ - \frac 12\bar c d(\rho(W),u)}\,.
\end{align*}
Using the fact that disjoint walls have disjoint projections, and summing out the inner sum over $W: W\subset \bB^i$, we see that this term is altogether at most $C |\partial \rho (\hull \cC_i)|$ for some other $C$. Combining the two bounds and summing over $i = 1,\ldots,s$ yields the desired inequality.  
\end{proof}

The next two lemmas are more involved and are the core of the analysis of the weight gain; it is here that the specific choices made in the construction of the map, namely criteria ({\tt{A1}})--({\tt{A2}}) appear most explicitly. 

\begin{lemma}[Analogue of {\cite[Lemma 4.19]{GL19b}}]\label{lem:C1-vs-X2-and-X4} There exists $\bar C$ such that 
\[ \sum_{f\in \bC_1} \sum_{g\in \bX_B^\cI} e^{-\bar c d(f,g)} \leq \bar C \fm(\cI;\cJ)\,.\]
\end{lemma}
\begin{proof}
We follow the proof of Lemma 4.19 of~\cite{GL19b}.
First of all, there exists $C>0$ such that
\begin{align*}
 \sum_{f\in \bC_1} \sum_{g\in \bX_B^\cI} e^{-\bar c d(f,g)} &\leq C\sum_{i> j^*} \sum_{y\in \bY} |\sF(\sX_i)|e^{- \bar cd(\lceil \tilde W_y\rceil, \sX_i)}\,.
 \end{align*}
Further, since $\sX_i$ was not deleted, by~({\tt A1}) and ({\tt A2}),
\[ \fm(\sX_i) < (i-1)\qquad\mbox{and}\qquad d(\lceil\tilde W_y\rceil,\sX_i) > (i-1)/2\,, \]
so that for every $i>j^*$ and every $y\in \bY$, $\fm(\sX_i) < 2d(\tilde W_y,\sX_i)$.  
Using $|\sF(\sX_i)| \leq 3\fm(\sX_i)+4$, we have that the above sum is at most
\[ 
  8C \sum_{i>j^*} \sum_{y\in \bY}  d(\lceil \tilde W_y\rceil,\sX_i) e^{-\bar cd(\lceil \tilde W_y\rceil ,\sX_i)}\,.
  \]
  Now, let 
\[ \bar\jmath=\min\Big\{j : \hgt(v_{j+1}) > \max_{y\in\bY} \max_{f\in \lceil\tilde W_y\rceil}\hgt(f) \Big\}\,,\]
and denote by $\bar y$ the index of the wall attaining this height. Then, using the fact that the projections $\rho(\lceil \tilde W_y\rceil)$ and $\rho(\lceil \tilde W_{y'}\rceil)$ are disjoint, 
\[ \sum_{j^*<i\leq \bar\jmath} \sum_{y\in \bY}  d(\lceil \tilde W_y\rceil ,\sX_i) e^{-\bar c d(\lceil \tilde W_y\rceil ,\sX_i) )} \leq C \bar\jmath \leq C \fm(\tilde \fW_{\bar y,S})\,.
\]
For the remaining increments, for every $y\in \bY$, let 
\[ d_\rho(y,i) := d\big(\rho(\lceil \tilde W_y\rceil),\rho(\sX_i) \big)\,,\]
and let $ \ell_1 = \bar\jmath + 1 < \ell_2 < \ldots < \ell_r$ be the renewal times of the function $d_\rho(y,\cdot)$, i.e.,  
\begin{align*} 
&d_\rho(y,\ell_j) < d_\rho(y,i)\quad\mbox{for all $\bar\jmath < i<\ell_j$}\,,\qquad \mbox{and}\qquad d_\rho(y,\ell_r) = \min\{ d_{\rho}(y,i) : \bar\jmath < i < t\}\,.
\end{align*}

Let $\ell_{r+1} := \infty$ and observe that for every $j= 1,\ldots,r$ and every $\ell_j \leq i < \ell_{j+1}$,
\[ d(\lceil\tilde W_y\rceil,\sX_i) \geq \sqrt{d_\rho(y,\ell_j)^2 + (\hgt(v_i)-\hgt(v_{\bar\jmath+1}))^2} \geq \frac{d_\rho(y,\ell_j) + i-\ell_{j}}{\sqrt2}\,,\]
using the definition of $\bar\jmath$ and that it satisfies $\bar\jmath < \ell_j$.
In particular, there exists $C, C'>0$ such that
\[ \sum_{\ell_j \leq i < \ell_{j+1}} d(\lceil \tilde W_y\rceil,\sX_i) e^{-\bar c d(\lceil \tilde W_y\rceil,\sX_i)} \leq \sum_{\ell_j \leq i < \ell_{j+1}}
C e^{-\frac{\bar c}{\sqrt2} (d_{\rho}(y,\ell_j)+i-\ell_j)} \leq C' e^{-\frac{\bar c}{\sqrt2} d_{\rho}(y,\ell_j)}\,.
\]
Summing over $1\leq j \leq r$, and noticing that $r \leq d_\rho(y,\ell_1)$,
\[
 \sum_{j=1}^r e^{-\frac{\bar c}{\sqrt{2}} d_{\rho}(y,\ell_j)} = \sum_{j=1}^r e^{-\frac{\bar c}{\sqrt{2}} (d_\rho(y,\ell_1)-j)} \leq \bar C\,.
\]
Summing over $y\in\bY$, this is at most $\sum_{y\in \bY} \bar C \leq \bar C \sum_{y\in \bY} \fm(\tilde W_y)$.
\end{proof}

\begin{lemma}[Analogue of {\cite[Lemma 4.20]{GL19b}}]\label{lem:C2-vs-X2-and-X4}There exists $\bar C$ 
\[ \sum_{f\in \bC_2} \sum_{g\in \bX_B^\cI} e^{-\bar c d(f,g)} \leq \bar C\,.\]
\end{lemma}
\begin{proof}
For $f: \rho(f)\notin S_n$, this sum was already controlled by a constant by Lemma~\ref{lem:F-vs-X2-and-X4}. 
For $f$ projecting into $S_n$, if $f$ is a ceiling face, nested in any wall of $\cI \setminus \cS_{x,S}$ then it must be at height $\hgt(\cC_{\bW})$, and its contribution to the above sum is at most a constant by Lemma~\ref{lem:F-vs-X2-and-X4}. 
It remains to consider $f\in \tilde W_y\cup\lceil \tilde W_y\rceil$ for some $y\in S_n \setminus \bY$, and $g\in\sX_i$ for $i>j^*$; by ({\tt A2}),
\[ (i-1)/2 < d(\lceil\tilde W_y\rceil,\sX_i) \leq d(f,g)\,. \]
Since $i> j^*$, by criterion ({\tt A1}), $\fm(\sX_i) < i-1$. Combining these, we get
\[ \sum_{f\in\bC_2}\sum_{g\in\bX_B^\cI} e^{-\bar c d(f,g)} \leq C \sum_{i>j^*} [3\fm(\sX_i)+4] e^{ - \bar c (i-1)/2} \leq 4 C \sum_{i>j^*} i e^{-\bar c (i-1)/2}\leq C'\,,\]
for some other $C'$. 
\end{proof}

It remains to control the interactions in $\cJ$ between the pillar and the truncated interface. In~\cite{GL19b} this was handled with the analogues of the above, but demanded a more complicated algorithmic map that iteratively considered the possible distances of $\cJ$ with the non-marked walls in order to determine which increments to delete. The introduction of isolated pillars allowed us to remove these steps from the definition of $\Phi_{\iso}$ and simplifies this analysis. Namely, situations in which the radius of congruence $\br$ is attained by interactions between $\cP_{x,S}^{\cJ}$ and the truncated interface $\cJ\setminus \cP_{x,S}^{\cJ}$ are independently controlled as follows.

\begin{claim}\label{clm:J-face-interactions}
There exists $\bar C$ (independent of $L,h$) such that for all $\cJ\in \bI_{\bW}$ with $\cJ\restriction_{S_n}\in \Iso_{x,L,h}$, we have  
\begin{align*}
    \sum_{f\in \cJ\setminus \cP_{x,S}^{\cJ}} \sum_{g\in \cP_{x,S}^{\cJ}} e^{ - \bar c d(f,g)}\le \bar C\,.
\end{align*}
\end{claim}
\begin{proof}
Recall the covering sets $\bF_{\triangledown},\bF_{\parallel},\bF_{-},\bF_{\curlyvee}$, and $\bF_{\ext}$ from~\eqref{eq:interface-containment-definitions} with respect to $(x,\bW,S_n)$, so that the first sum is at most a sum over $f\in \bF_{-}\cup \bF_{\curlyvee}\cup \bF_{\ext}$, and the second sum is at most a sum over $g\in \bF_{\triangledown} \cup \bF_{\parallel}$. Let us first consider the contributions from $f\in \bF_{-}$ with $g\in \bF_{\parallel}$. Their contribution is evidently bounded as
\begin{align*}
    \sum_{f\in \bF_{-}} \sum_{g\in\bF_{\parallel}} \le  \sum_{j\ge 1} C e^{ - \bar c j} \le C'\,,
\end{align*}
for some universal constant $C'$. The remaining pairs of $f$ and $g$ are shown to be at most $\bar C e^{ - \bar c L}$ by Lemma~\ref{lem:iso-pillar-interactions}, concluding the proof. 
\end{proof}

\subsection{Bounding the weight gain under \texorpdfstring{$\Phi_{\iso}$}{Phi\_iso}}
In this section, we use the preliminary lemmas from the previous section to control the contribution from the interactions through $\g$ to the weight-gain under $\Phi_{\iso}$ and show that they are dominated by $\fm(\cI;\cJ)$. 

\begin{proof}[\textbf{\emph{Proof of Proposition~\ref{prop:partition-function-contribution}}}]
In what follows, let $\bar C$ denote the maximum of those constants among Claims~\ref{clm:W-and-H-vs-everything}--\ref{clm:J-face-interactions}.
By Theorem~\ref{thm:cluster-expansion}, the left-hand side of Proposition~\ref{prop:partition-function-contribution} is at most,
\begin{align}
    |\sum_{f\in\cI} \g(f,\cI) - \sum_{f'\in\cJ} \g(f',\cJ)| &\leq \sum_{f\in \bX_A^\cI} |\g(f,\cI)| +\sum_{f\in \bV} |\g(f,\cI)| \label{eq:part-function-splitting-I} \\
    & \qquad + \sum_{f'\in \bX_A^\cJ} |\g(f',\cJ)| + \sum_{f'\in \bH} |\g(f',\cJ)|+ \sum_{f'\in W_{x,\parallel}^{\mathbf h}} |\g(f',\cJ)| \label{eq:part-function-splitting-J}\\
    &\qquad + \sum_{f\in \bB} 
    |\g(f,\cI)-\g(\theta_{\udarrow}f,\cJ)| + \sum_{f\in \bX_B^\cI} |\g(f,\cI)- \g(\lrvec\theta f,\cJ)|
      \label{eq:part-function-splitting-IJ} \,.
\end{align}
The first and third terms above are evidently at most $\bar K |\bX_A^\cI|\le \bar C \fm(\cI;\cJ)$ by Claim~\ref{clm:X1-and-X3-vs-everything}. The second and fourth terms are at most $\bar K \bar C \fm(\bV)$ by Claim~\ref{clm:W-and-H-vs-everything}. The fifth term is at most $\bar K |W_{x,\parallel}^{\mathbf h}|$ which is at most $2\bar K \fm(\cI;\cJ)$ by~\eqref{eq:m(I;J)-WxJ}.  
It remains to consider the two sums in line~\eqref{eq:part-function-splitting-IJ}, which by~\eqref{eq:g-exponential-decay} satisfy,
\begin{align*}
\sum_{f\in \bX_B^\cI} |\g(f,\cI)- \g(\lrvec\theta f,\cJ)| &\leq
\sum_{f\in \bX_B^\cI} \bar K e^{-\bar c\br(f,\cI; \lrvec\theta f,\cJ)}\,,  \\
 \sum_{f\in \bB} |\g(f,\cI)- \g(\theta_\udarrow f,\cJ)| &\leq
\sum_{f\in \bB} \bar K e^{-\bar c\br(f,\cI; \theta_\udarrow f,\cJ)}\,.
\end{align*}
Consider the right-hand sides according to the face $g$ attaining $\br(f,\cI; \lrvec\theta,\cJ)$ and $\br(f,\cI;\theta_{\udarrow},\cJ)$ respectively.
\begin{enumerate}[(i)]
\item If $g \in \bX_A^\cI\cup\bX_A^\cJ$, both these sums are at most 
\begin{align*}
     \sum_{g\in \bX_A^\cI\cup \bX_A^\cJ} \bar K \Big[ \sum_{f\in \bX_B^\cI} (e^{ - \bar c d(f,g)}+ e^{- \bar c d(\lrvec \theta f,g)})+ \sum_{f\in \bB}(e^{ - \bar c d(f,g)} + e^{ -\bar c d(\theta_{\udarrow} f, g)})\Big]\,.
\end{align*}
This is then at most $4\bar K \bar C \fm(\cI;\cJ)$ by Claim~\ref{clm:X1-and-X3-vs-everything}. 

\item If $g\in \bV \cup \bH \cup W_{x,\parallel}^{\mathbf h}$, both these sums are at most 
\begin{align*}
     \sum_{g\in \bV \cup \bH \cup W_{x,\parallel}^{\mathbf h}} \bar K \Big[ \sum_{f\in \bX_B^\cI} (e^{ - \bar c d(f,g)}+ e^{- \bar c d(\lrvec \theta f,g)})+ \sum_{f\in \bB}(e^{ - \bar c d(f,g)} + e^{ -\bar c d(\theta_{\udarrow} f, g)})\Big]\,.
\end{align*}
which is at most $12 \bar K \bar C \fm(\cI;\cJ)$ by Claim~\ref{clm:W-and-H-vs-everything} and~\eqref{eq:|W_x^J|}. 

\item For $g\in \bX_B^\cI\cup\bX_B^\cJ$, we only need to consider the second sum ($f\in \bB$) since if $f\in \bX_B^\cI$, the radius $\br$ cannot be attained by $g\in \bX_B^\cI\cup \bX_B^\cJ$ as all increments in $\bX_B^\cI$ are shifted by the same vector under $\lrvec \theta$ to obtain $\bX_{B}^{\cJ}$.  
Turning to the sum over $f\in\bB$, it splits into the following: 
\begin{align*}
    \bar K \sum_{g\in \bX_B^\cI} \sum_{f\in \bC_1\cup \bC_2 \cup \mathbf{Fl}} (e^{ - \bar c d(f,g)}+ e^{ - \bar c d(\theta_\udarrow f,\lrvec \theta g)})
\end{align*}
The second term here is evidently bounded by $\bar C$ by the sum in Claim~\ref{clm:J-face-interactions}. 
For the first term, the contribution from $f\in \bC_1$ is at most $\bar C \bar K \fm(\cI; \cJ)$ by Lemma~\ref{lem:C1-vs-X2-and-X4}; the contribution from $f\in \bC_2$ is at most $\bar C \bar K $ by Lemma~\ref{lem:C2-vs-X2-and-X4}; the contribution from $f\in \mathbf{Fl}$ is at most $\bar C\bar K$  by Lemma~\ref{lem:F-vs-X2-and-X4}.

\item For $g \in \bB$, the first sum can be expressed as 
\begin{align*}
    \bar K \sum_{f\in \bX_B^\cI} \sum_{g\in \bC_1\cup \bC_2\cup \mathbf{Fl}} (e^{ - \bar c d(f,g)}+ e^{ - \bar c d(\lrvec \theta f, \theta_\udarrow g)})\,.
\end{align*}
As in (iii), this is at most $3\bar C \bar K + \bar C \bar K  \fm(\cI;\cJ)$ by Lemmas~\ref{lem:F-vs-X2-and-X4} and~\ref{lem:C1-vs-X2-and-X4}--\ref{clm:J-face-interactions}. 

For the second sum, in which $f \in \bB$, we use the decomposition according to innermost nesting ceiling of $(\bB^i)_{i\le s}$ per Lemma~\ref{lem:B-vs-B}, and collect in $\bB_{\ext}$ all faces of $\bB$ that are not in any of $\bB^i$, i.e., are not nested in any wall in $\bV$. Note, firstly, that for $f\in \bB_{\ext}$, the radius cannot be attained by $g\in \bB_{\ext}$ as $\theta_{\udarrow} \bB_{\ext} = \bB_{\ext}$. It therefore must be attained by a wall face in $\bB$, which means it must be attained by a face in $\bB^i$ for some $i= 1,\ldots,s$. The contribution from these terms can be bounded by 
\begin{align*}
    \sum_{f\in \bB_{\ext}} \sum_{i=1}^s \sum_{g\in \bB^i} e^{ - \bar c d(f,g)} \le \bar C \fm(\bV)\,,
\end{align*}
using Lemma~\ref{lem:B-vs-B}.  Next consider $f\in \bB^i$ for some $i=1,\ldots,s$. Since all of $\bB^i$ undergoes the same vertical shift under the deletion of $\bV$, if the radius is attained in $\bB$, it must be attained by a face of $\bB_{\ext}$ or $\bB^j$ for some $j\ne i$: in particular, it must be attained either by a face projecting into $\rho(\hull{\cC}_i)^c$, or by a face in $\bB^j$ for some $j$ satisfying $\cC_i \subset \rho(\cC_j)^c$ (i.e., $\cC_j$ nested in $\cC_i$). These contributions are collected in 
\begin{align*}
    \sum_i \sum_{f\in \bB^i} \sum_{g\in \sF(\mathbb Z^3): \rho(g)\notin \rho(\hull \cC_i)} e^{ - \bar c d(f,g)}  + \sum_{i} \sum_{f\in \bB^i} \sum_{j:\cC_i \subset \rho(\hull \cC_j)^c} \sum_{g\in \bB^j} e^{-\bar c d(f,g)}\,.
\end{align*}
By Lemma~\ref{lem:B-vs-B}, the first sum is at most  $\bar C \fm(\bV)$; the second sum can be rewritten as 
\[
\sum_{j} \sum_{g\in \bB^j} \sum_{i:\cC_i \subset \rho(\hull \cC_j)^c} \sum_{f\in \bB^i} e^{- \bar c d(f,g)} \leq \sum_j \sum_{g\in \bB^j} \sum_{f\in \sF(\mathbb Z^3): \rho(f)\notin \rho(\hull \cC_i)} e^{ - \bar cd(f,g)} \leq \bar C \fm(\bV)\,,
\]
again by Lemma~\ref{lem:B-vs-B}. 
\end{enumerate}
Combining the above bounds with Claim~\ref{clm:m(W)}, we find that the two terms in~\eqref{eq:part-function-splitting-IJ} are themselves at most $C\fm(\cI;\cJ)$ for some other $C$. Together with the bounds on~\eqref{eq:part-function-splitting-I}--\eqref{eq:part-function-splitting-J} we conclude the proof. 
\end{proof}

\subsection{Bounding the multiplicity of \texorpdfstring{$\Phi_\iso$}{Phi\_iso}}
Fix $x, L,h',h$ for the remainder of this section and let $\Phi_{\iso} = \Phi_{\iso}(x,L,h)$. 
We bound the multiplicity of $\Phi_\iso$, by, for each $\cJ\in \Phi_{\iso}(\bI_{\bW_n}\cap E_{x,S}^{h'})$, defining an injective map, called a \emph{witness}, $\Xi=\Xi_{\cJ,x}$ on $\Phi_{\iso}^{-1}(\cJ)$, and bounding the cardinality of $\Xi(\{\cI\in\Phi_{\iso}^{-1}(\cJ) \,:\; \fm(\cI;\cJ)=M\})$. 

\medskip
\noindent \textbf{Construction of the witness.}
Fix any $\cJ\in \Phi_{\iso}(\bI_{\bW_n}\cap E_{x,S}^{h'})$. For each $\cI \in \Phi_{\iso}^{-1}(\cJ)$, our witness $\Xi_{\cJ,x}(\cI)$ will be a face subset of $\sF_{\mathbb Z^3}$ consisting of at most $1+ |\bY|$ many $*$-connected components, and decorated by a coloring of its faces by one of $\red,\blue,\green$, constructed as follows. 
\begin{enumerate}
    \item Include the faces of $\bX_{A}^{\cI}$ and color them $\green$. 
    \item Process the faces $y\in \bY$ via some lexicographic order $(y^1 , y^2 , \ldots,y^{|\bY|})$: for $i = 1,\ldots, |\bY|$,
\begin{itemize}
    \item Let the \blue\ faces of $\Upsilon_{y^i}$ be the set  
\[ \Gamma_{y^i} = \Theta_{\textsc{st}}\Clust(\tilde \fW_{y,S}) \setminus \bigcup_{j<i} \Gamma_{y^j}\,.\]
\item For each standard wall $W$ in $\Gamma_{y^i}\setminus \Theta_{\textsc{st}}\tilde \fW_{y^i,S}$, add to $\Upsilon_{y^i}$ the shortest path of \red\ faces of $\cL_{0,n}$ connecting $W$ to its innermost nesting wall in $\Theta_{\textsc{st}}\tilde \fW_{y^i,S}$. 
\item Also, connect the walls of $\Theta_{\textsc{st}}\tilde \fW_{y^i,S} \cap \Gamma_{y^i}$ by connecting, via a path of $\red$ faces in $\cL_{0,n}$ each one except the outermost one, to its innermost nesting wall among $\Theta_{\textsc{st}}\tilde \fW_{y^i,S} \cap \Gamma_{y^i}$. 
\end{itemize}
\end{enumerate}

\medskip
\noindent \textbf{Reconstructing $\cI$ from the witness.} To see that this yields a valid  ``witness" of the pre-image $\cI$ (i.e., is injective), we show that from a witness $\Xi(\cI)$ and the interface $\cJ$, one can uniquely reconstruct $\cI$. 

\begin{lemma}\label{lem:witness-injective}
For every $\cJ \in \Phi_\iso(\bI_{\bW})$, the map $\Xi_{\cJ,x}$ is injective on $\Phi_\iso^{-1}(\cJ)$. 
\end{lemma}

\begin{proof}It suffices to show that from a given $\cJ$ and any element of  $\Xi_{\cJ, x}(\Phi_{\iso}^{-1}(\cJ))$ we can recover, uniquely, $\cI\in \Phi_\iso^{-1}(\cJ)$.  
From a witness in $\Xi_{\cJ,x}(\Phi_{\iso}^{-1}(\cJ))$, we reconstruct $\cI$ by reconstructing its spine $\cS_{x,S}$ together with the standard wall representation of $\cI \setminus \cS_{x,S}$. 
\begin{enumerate}
    \item In order to reconstruct the spine $\cS_{x,S}$:
    \begin{enumerate}
        \item Extract $\bX_A^\cI$ as the connected set of \green\ faces of $\Xi(\cI)$. 
        \item Extract $\bX_B^\cJ$ by taking (the bounding faces of) all cells in $\cP_{x,S}^{\cJ}$ above $\hgt(v_{j^\star+1})$ (this height is read off from $\bX_A^\cI$).
    \item Obtain $\cS_{x,S}$ by horizontally shifting $\bX_B^\cJ$ so that its bottom cell aligns with the top cell of~$\bX_A^\cI$.
    \end{enumerate}
    \item In order to reconstruct the standard wall representation of $\cI\setminus \cS_{x,S}$:
    \begin{enumerate}
    \item Collect the standardizations of all walls of $\cJ\setminus \cP_{x,S}^{\cJ}$, and add all $\blue$ faces of $\Xi(\cI)$.
    \end{enumerate}
\end{enumerate}
Given $\cS_{x,S}$ and the standard wall collection $\Theta_{\textsc{st}}(\tilde W_z)_{z\in \cL_{0,n}}$, we obtain $\cI$ by first recovering the interface $\cI \setminus \cS_{x,S}$ via Lemma~\ref{lem:interface-reconstruction}, then appending to that $\cS_{x,S}$. 
\end{proof}

\noindent \textbf{Enumerating over possible witnesses.} 
It remains to enumerate over the set of all possible witnesses of interfaces in $\Phi_{\iso}^{-1}(\cJ)$ with excess area $\fm(\cI;\cJ) = M$ and show it is at most exponential in $M$. 
\begin{lemma}\label{lem:multiplicity-witness}
There exists some universal $C_\Phi>0$ such that for every $M\geq 1$, and every $\cJ \in \Phi_{\iso}(\bI_{\bW})$,
\begin{align*}
|\{\Xi_{\cJ,x} (\cI): \cI\in \Phi^{-1}_{\iso}(\cJ)\,,\, \fm(\cI;\cJ) = M\}| \leq C_\Phi^{L^3 M}\,.
\end{align*}
\end{lemma}
Combining the above lemma with Lemma~\ref{lem:witness-injective} would immediately imply the desired Proposition~\ref{prop:multiplicity}. 

\begin{proof}[\textbf{\emph{Proof of Lemma~\ref{lem:multiplicity-witness}}}]
Evidently it suffices to separately enumerate over the number of choices for $\green$ faces, $\bX_{A}^{\cI}$, and the number of $\{\red,\blue\}$ colored faces constituting $\bigcup_{i=1 ,\ldots,|\bY|} \Upsilon_{y^i}$ and show that each of these are at most $C^{L^3 M}$ for some universal constant $C$. 

\medskip
\noindent \emph{Enumerating over the faces in $\bX_{A}^{\cI}$.}
Let us first enumerate over the choices of $\green$ faces, i.e., faces of $\bX_{A}^{\cI}$. For every witness of an $\cI$ having $\fm(\cI;\cJ) = M$,  $\bX_{A}^{\cI}$ forms a $*$-connected face-set of size at most $|\bX_{A}^{\cI}|\le j^* + \sum_{j=1}^{j^*} \fm(\sX_j)\le 6L^3 \fm(\cI;\cJ)$ by Claim~\ref{clm:m(W)}. This face-set is rooted at $v_1$; to enumerate over the number of choices of $v_1$, notice that it must be in a ball of radius $\fm(\tilde \fW_{v_1,S}) \le 3M$ of $x+ (0,0,\hgt(\cC_\bW))$, and thus crudely there are at most $(3M)^3$ such choices.  Putting these together, by Fact~\ref{fact:number-of-walls} the number of choices for $\bX_{A}^{\cI}$ is at most $CM^3 s^{ 6L^3 M}$.    

\medskip
\noindent \emph{Counting the number of faces in $\Upsilon_{y^i}$.}
Let us now enumerate over the choices of $\blue$ and $\red$ faces of $\Xi(\cI)$. We first bound the total number of $\blue$ and $\red$ faces; the number of $\blue$ faces of $\Xi(\cI)$ is clearly at most $|\bV|\le 2\fm(\bV)\le 3M$ by Claim~\ref{clm:m(W)}. 

We can bound the number of $\red$ faces by bounding the number of $\red$ faces in each $\Upsilon_{y^i}$ and summing over $i = 1,\ldots, |\bY|$. Consider first the number of red faces added to connect walls in $\Gamma_{y^i}\setminus \tilde \fW_{y^i,S}$ to their innermost nesting walls in $\tilde \fW_{y^i,S}$. For every such wall $W$, it is in $\bV$ because it is in the wall-cluster of some wall of $\tilde \fW_{y^i,S}$, while not being in $\tilde \fW_{y^i,S}$; in particular, it is closely nested in some wall of $\tilde \fW_{y^i,S}$, and therefore the number of faces we added here were at most $\fm(W)$. Since no wall is double counted, the total contribution from such added $\red$ faces is at most $\sum_{y\in \bD\setminus \bY} \fm(\tilde W_y) \le 6\fm(\cI;\cJ)$ by Claim~\ref{clm:m(W)}. 

To bound the number of red faces added to connect walls $W$ of $\Theta_{\textsc{st}}\tilde \fW_{y^i,S}\cap\Gamma_{y^i}$ notice that every wall in $\bigcup_{y^i\in \bY} \tilde \fW_{y^i,S}$ is connected to at most one wall interior to it by $\red$ faces in this step of the processing. The distance of a standard wall $W$ to another standard wall interior to it $W'$ along $\cL_{0,n}$ is at most $\fm(W)$ so that the total number of $\red$ faces added in this stage is $\fm(\bigcup_{y\in \bY} \tilde \fW_{y,S})$ which is at most $6\fm(\cI;\cJ)$ by Claim~\ref{clm:m(W)}. 

\medskip
\noindent \emph{Enumerating over the sets $\Upsilon_{y^i}$.}
By the above, the total number of faces in $\bigcup_{i} \Upsilon_{y^i}$ is $CM$. At the same time, by construction, every $\Upsilon_{y^i}$ is a $*$-connected set of faces, rooted at $y^i$, and the number of such roots is at most $|\bY|\le CM$. In order to enumerate over the number of face subsets $\bigcup_{i}\Upsilon_{y^i}$, let us begin by enumerating over the locations of the at most $CM$ roots, and then over the $*$-connected face-sets $\Upsilon_{y^i}|$ each of size $M_i$ such that $\sum_{i} M_i \le CM$. Towards this, split $\mathbf Y$ into $\bY_1$ and $\bY_2$ where $\bY_1 = \{x,\rho(v_1), y^\dagger,y^*\}$ and $\bY_2 = \bY \setminus \bY_1$ are the indices marked by step 6 of $\Phi_\iso$.  
 It suffices to separately enumerate over $(\Upsilon_{y^i})_{y^i\in \bY_1}$ and $(\Upsilon_{y^i})_{y^i\in \bY_2}$.
 
 We claim that for each of the at-most four elements $y \in \bY_1$ having $M_i>0$, there are at most $M^3$ choices for $y$. Evidently if $y  = \{x\}$ there is no choice here. If $y \in \{\rho(v_1),y^\dagger\}$, then $y$ must be nested in $\tilde \fW_{x,S}$, and thus the number of such choices is at most $\fm(\fW_{x,S})^2 \le \fm(\bV)^2\le CM^2$. If $y = \{y^*\}$, then $d_\rho(y^*, \bX_{A}^\cI) \le j^*$, so that $d(y^*,x)\le j^* + |\bX_A^\cI|\le CL^3 M$ by  Claim~\ref{clm:m(W)} and the number of such choices of $y^*$ is at most $(CM)^2$. To each of these, we allocate a $*$-connected rooted face set $(\Upsilon_{y})_{y\in \bY_1}$ of size at most $M_i \le M$, so that by Fact~\ref{fact:number-of-walls}, the number of such choices is at most $C^{4M}$. 
 
We now turn to the contribution from index points in $\bY_2$. Observe, importantly, that for every $y^i\in \bY_2$ for which $M_i>0$, it must be that $M_i \ge \log [d(y^i,x)]$ if $d(y^i,x)\ge L$.
Let us split up $B_{L^3 h}(x)\cap \cL_{0,n}$, into rings $\mathcal A_\ell = B_{2^{\ell+1}}(x)\cap \cL_{0,n} - B_{2^\ell}(x)\cap \cL_{0,n}$. We first enumerate over the number of $y^i$ in $\mathcal A_\ell$ for each $\ell\le C'M$ (we only need to go up to $C'M$ as $\log [2^{C'M}]$ exceeds the total number of faces we have to allot). The number of allocations of at most $|\bY_2|\le |\bY| \le CM$ index points to at most $C'M$ rings is bounded by $2^{(C+C')M}$. Given such an allocation, let $N_\ell$ be the number of root-points in $\mathcal A_\ell$, let $M_\ell = \sum_{i: y^i\in \mathcal A_\ell} M_i$ and enumerate over the number of allocations of $M_\ell$ faces to $N_\ell$ index points. 

For $\ell \le \log_2 L$, there are trivially at most $M^{L^2}$ such allocations. 
Now consider $\ell \ge \log_2 L$: for $y^i \notin B_{L}(x)$, since $M_i \ge \log[d(y^i,x)]$, also $N_\ell \le M_\ell/\ell$. It follows that there are at most 
\begin{align*}
    \binom{2^{2\ell}}{N_\ell} \binom{M_\ell + N_\ell + 1}{N_\ell}\le 2^{2\ell N_\ell} 2^{M_\ell + N_\ell +1} \le 2^{4M_\ell}
\end{align*}
many such choices. 
Multiplying over $\ell$, this implies that the number of possible choices of $(y^i,M_i)$ is at most $C^M$ for some constant $C$. Enumerating over the $*$-connected face sets of size at most $M_i$ rooted at $y^i$, gives a factor of $C^{M_i}$ for each of these by Fact~\ref{fact:number-of-walls}. 

Finally, enumerating over the $\{\red,\blue\}$ colorings of these faces and collecting an additional $2^M$, we find that there exists $C>0$ such that there are at most $C^{L^3 M}$ possible choices of $(\Upsilon_{y^i})_{i = 1,\ldots,|\bY|}$. 
\end{proof}

\subsection*{Acknowledgments} 
R.G.~thanks the Miller Institute for Basic Research in Science for its support. E.L.~was supported in part by NSF grant DMS-1812095.

\bibliographystyle{abbrv}

\bibliography{references}

\end{document}